\documentclass[11pt,a4paper,leqno]{article}
\usepackage{mymacros}
\usepackage{cite}

\newcommand{\LL}{\bm{L}}
\newcommand{\PP}{\bm{P}}

\newcommand{\TT}{\bm{T}}
\newcommand{\EE}{\bm{E}}

\newcommand{\Pg}{\mathsf{P}}
\newcommand{\Eg}{\mathsf{E}}
\newcommand{\Zg}{\mathsf{Z}}

\title{Stochastic dynamics and the Polchinski equation: an introduction}

\author{Roland Bauerschmidt\footnote{University of Cambridge, DPMMS \& Courant Institute, NYU. E-mail: {\tt bauerschmidt@cims.nyu.edu}.}
\and Thierry Bodineau\footnote{IHES, CNRS. E-mail: {\tt bodineau@ihes.fr}.}
\and Benoit Dagallier\footnote{University of Cambridge, DPMMS \& Courant Institute, NYU. E-mail: {\tt bd2543@cims.nyu.edu}.}}
\date{April 23, 2024}

\begin{document}

\maketitle
\begin{abstract}
  This introduction surveys a renormalisation group perspective on log-Sobolev inequalities and related properties of stochastic dynamics.
  We also explain the relationship of this approach to related recent and less recent developments such as
  Eldan's stochastic localisation and the F\"ollmer process,
  the Bou\'e--Dupuis variational formula and the Barashkov--Gubinelli approach,
  the transportation of measure perspective, and the classical analogues of these ideas for Hamilton--Jacobi equations
  which arise in mean-field limits.
\end{abstract}
\setcounter{tocdepth}{2}
\tableofcontents

\section{Introduction}

Functional inequalities have been thoroughly studied  in different contexts \cite{MR1849347, MR3155209}
and one important motivation is to quantify the relaxation of stochastic dynamics by using Poincar\'e and (possibly modified) log-Sobolev inequalities
\cite{MR1746301,MR1971582,MR1292280,MR1410112,MR2283379}.
Statistical mechanics offers an interesting setting
to apply these inequalities and to analyse 
 the information they provide in various physical regimes.
Indeed,  one would like to describe the relaxation to equilibrium of lattice gas  and spin dynamics, which are modelled by stochastic evolutions on high-dimensional state spaces. Their continuum limits, often described by (singular) SPDEs, are also of a lot of interest.

The structure of the equilibrium Gibbs measures is sensitive to the occurrence of phase transitions and 
the dynamical behaviour will also be strongly influenced by phase transitions. In the uniqueness regime, namely in absence of a phase transition (typically at high temperatures), one expects that the dynamics relax exponentially fast uniformly in the dimension of the state space with a speed of relaxation diverging when the temperature approaches   the critical point. 
We refer to Sections \ref{sec: Difficulties arising from statistical physics perspective} and \ref{sec: Difficulties arising from continuum perspective} for more details.
In a phase transition regime   (typically corresponding to low temperatures), one expects different types of behaviours  depending strongly on the type of boundary conditions and we will not discuss the corresponding phenomena in these notes. 

The fast relaxation towards equilibrium
in the uniqueness regime (or at least deep in it)
is well understood and we refer to 
\cite{MR1746301,zbMATH02042289, MR1971582} for very complete accounts of the corresponding theory. 
Roughly speaking, it has been shown for a wide range of models that good mixing properties of the equilibrium measure are equivalent to fast relaxation of the dynamics, namely uniform bounds (with respect to the domains and the boundary conditions) on the Poincar\'e or the log-Sobolev constants.  
For the Ising model \cite{MR1746301,MR3844472,MR4586225},
the validity of the mixing properties have been proved in the whole uniqueness regime leading to strong relaxation statements on the dynamics,
and more detailed dynamical features are also understood in that regime \cite{MR3486171,MR3020173}.
For more general systems and in particular continuous spin systems, the picture is much less complete.

The main goal of this survey is to present a different perspective on the derivation of functional inequalities based 
on the renormalisation group theory,
introduced in physics by Wilson \cite{Wilson_Kogut_1974}, and with its continuous formulation emphasised in particular by Polchinski~\cite{polchinski1984269}. 
The renormalisation group was introduced to
study the critical behaviour and the existence of continuum limits
of equilibrium models of statistical physics and quantum field theory from a unified perspective. 
The renormalisation group formalism associates with a Gibbs measure a flow of measures defined in terms of a renormalised potential
(see Section~\ref{sec:LSI-setup}).
We show how this structure can be used to prove log-Sobolev inequalities under a condition on the renormalised potential 
which is a multiscale generalisation of the Bakry-\'Emery criterion (see Section~\ref{sec: Log-Sobolev inequality via a multiscale Bakry--Emery method}).
The renormalised potential obeys a second-order Hamilton--Jacobi type equation
(the Polchinski equation)
with characteristics given by a stochastic evolution
(see Section~\ref{subsec: Pathwise realisation of the Polchinski semigroup}) which coincide with the stochastic process of Eldan's \emph{stochastic localisation} method introduced for very different purposes \cite{Eldan_ICM}.
Section~\ref{sec:stochloc} provides a dictionary to relate both points of view.

An alternative to the multiscale Bakry--\'Emery method  to derive log-Sobolev inequalities
(with much similarity and both advantages and disadvantages)
is   the entropic stability estimate recently established in \cite{2203.04163}
and reviewed  in  Section~\ref{subsec: Entropic stability estimate}.
This estimate applies to the same Polchinki flow, or its equivalent interpretation as stochastic localisation.
It originated in the spectral and entropic independence estimates \cite{MR4232133,2106.04105}
which are similar estimates for a different flow that takes the role of the Polchinski flow in another kind of model.
This analogy has already been highlighted in \cite{2203.04163} to which we refer for a discussion of this relation.
Compared to the established approaches to functional inequalities for statistical mechanical models, which typically rely on spatial decompositions, all of the approaches discussed
here are more spectral in nature. Spectral quantities are more global and therefore allow to capture
for example the near-critical behaviour better.
This is illustrated in a series of applications reviewed in Section~\ref{sec: Applications}.

The Polchinski renormalisation and the stochastic localisation can be seen as two sides of the same coin, sharing thus very similar structures.
In fact, this type of stochastic equations have been considered 
much earlier by F\"ollmer \cite{zbMATH00193914} as an optimal way to generate a target measure. 
In Section~\ref{sec:variational-transport}, the Polchinski renormalisation flow is shown to coincide   
with the optimal stochastic process associated with a suitable \emph{varying} metric.
In applications to statistical mechanics models, this metric captures the notion of \emph{scale}
so that the Polchinski flow (a continuous renormalisation group flow) provides a   canonical way of decomposing the entropy according to scale.
Finally, in Section \ref{subsec: Variational representation of the renormalised potential}, the renormalised potential
is rewritten as a variational principle using the Polchinski flow, known as the  Bou\'e-Dupuis or Borell formula in the generic context, see 
\cite{MR3112438}, whose origin is in stochastic optimal control theory.
This correspondence is at the heart of the Barashkov--Gubinelli variational method \cite{MR4173157}.

In Appendix~\ref{app:HJ}, these results are compared with the corresponding control theory for classical Hamilton--Jacobi equations,
and a simple comparison with the line of research initiated in \cite{MR4275243} is also given.

\section{Background on stochastic dynamics}
\label{sec:background}

\subsection{Motivation: Spin models and their stochastic dynamics}

Our goal is to study dynamical (and also some equilibrium) aspects
of continuous and discrete spin models of statistical mechanics
such as Euclidean field theories or Ising-type spin models.
Throughout this article, $\Lambda$ will be a general finite set (of vertices),
but we have $\Lambda \subset \Z^d$ large in mind,
or $\Lambda \subset \epsilon \Z^d$ approximating $\R^d$ (or a subset of it) when $\epsilon \to 0$
in the case of models defined in the continuum.
Sometimes we identify $\Lambda$ with $[N]=\{1,\dots, N\}$.
Spin fields are then random functions
$\varphi: \Lambda \to T$ where, for example, $T=\R$ in the case of continuous scalar spins
or $T=\{\pm 1\}$ in the case of (discrete) Ising spins.
For discrete spins, we often write $\sigma$ instead of $\varphi$
for a spin configuration.

\paragraph{Continuous spins}

In the setting of continuous spins, the equilibrium Gibbs measures have expectation of the form
\begin{equation}
  \E_{\nu}[F(\varphi)] \propto \int_{\R^\Lambda} e^{-H(\varphi)} \, F(\varphi)\, d\varphi \,
\label{eq: Gibbs measure}
\end{equation}
where the symbol $\propto$ denotes the equality of the measures up to a normalisation factor.
We will refer to $H$ as the action or as the Hamiltonian (depending on the context).
The main class of $H$ that we will focus on are of the following form:
for spins $\varphi = ( \varphi_x )_{x \in \Lambda}$ taking values in $\bbR$ (or vector spins with values in $\bbR^n$),
an interaction matrix $A$, and a local potential $V$,
\begin{equation}
  H(\varphi) = \frac12 (\varphi,A\varphi) + V_0(\varphi), \qquad V_0(\varphi) =\sum_{x\in\Lambda} V(\varphi_x).
 \label{eq: general Hamiltonian}
\end{equation}
Defining the discrete Laplace operator on $\Lambda \subset \Z^d$ by
\begin{equation}
\forall x \in \Lambda: \qquad 
(\Delta^\Lambda f)_x := \sum_{y \in \Lambda: y\sim x} (f_y-f_x),
\label{eq: discrete Laplace}
\end{equation}
a classical choice of interaction is obtained by setting  
$A=- \beta \Delta^\Lambda$ for some (inverse temperature) parameter $\beta >0$.
In this case, the Hamiltonian reads
\begin{equation}
\forall \varphi \in \R^\Lambda: \qquad 
H(\varphi) = \frac{\beta}{4} \sum_{x,y \in \Lambda, \atop x \sim y} (\varphi_x - \varphi_y)^2 
+  V_0(\varphi),
 \label{eq: specific Hamiltonian}
\end{equation}
where the nearest neighbour interaction is denoted by $x \sim y$ and the sum counts each pair $\{x,y\}$ twice.
As an example, a typical choice for the single-spin potential $V$ is the Ginzburg--Landau--Wilson $\varphi^4$ potential,
in which case one usually sets $\beta=1$ (and $r$ has the role of a temperature),
\begin{equation}
\label{eq: Ginzburg--Landau potential}
V(\varphi) = \frac14 g|\varphi|^4 +  \frac12 r |\varphi|^2
\quad \text{with $g>0$ and $r<0$.}
\end{equation}
The following Glauber--Langevin dynamics is reversible for the measure introduced in \eqref{eq: Gibbs measure}:
\begin{equation}
  d\varphi_t = -\nabla H(\varphi_t) \, dt + \sqrt{2} dB_t.
\label{eq: Langevin equation}
\end{equation}
For the choices \eqref{eq: specific Hamiltonian} and \eqref{eq: Ginzburg--Landau potential}, this stochastic differential equation (SDE) reads
\begin{equation}
  d\varphi_t
  = -A\varphi_t \, dt- \nabla V(\varphi_t) \, dt+ \sqrt{2}dB_t
  = \Delta^\Lambda\varphi_t \, dt- g|\varphi_t|^2\varphi_t \, dt
  - r \varphi_t \, dt
  + \sqrt{2}dB_t .
  \label{eq: general Glauber}
\end{equation}

In this survey, we are interested in the long time behaviour of these dynamics when the number of spins is large.  
We will consider two cases: either $\Lambda \to \Z^d$ for the Glauber dynamics of an Ising-type model with continuous spins; or $\Lambda \subset \epsilon\Z^d$ with $\Lambda \to \R^d$ or $\Lambda \to [0,1]^d$ in which case
a suitably normalised version of $\varphi$ describes the solution of a singular SPDE in the limit $\epsilon \to 0$.

\paragraph{Discrete spins}

In the setting of discrete spins, we focus on the Ising model where  $\sigma \in \{\pm 1\}^\Lambda$ and
\begin{equation}
  \E_{\nu}[F(\sigma)] \propto \sum_{\sigma \in \{\pm 1\}^\Lambda} e^{-\frac{\beta}{2} (\sigma, A\sigma)} F(\sigma)
  \label{eq:  Ising model}
\end{equation}
for some symmetric coupling matrix $A$.
Its Glauber dynamics is a continuous- or discrete-time Markov process with local transition rates
$c(\sigma,\sigma^x)$ from a configuration $\sigma$ to $\sigma^x$ where $\sigma^x \in \{\pm 1\}^\Lambda$ denotes
the configuration obtained from $\sigma\in\{\pm 1\}^\Lambda$ by flipping the sign of the spin at $x$. 
The transition rates are assumed to satisfy the detailed-balance condition 
\begin{equation}
\label{eq: detailed balance}
\nu(\sigma)c(\sigma,\sigma^x)=\nu(\sigma^x)c(\sigma^x,\sigma),
\end{equation} 
which implies that the measure \eqref{eq:  Ising model} is invariant.
Typical choices are described below in the next section.
We will be interested in the large time behaviour of the dynamics when $\Lambda \to \bbZ^d$.

\subsection{Generalities on Glauber--Langevin dynamics}
\label{sec:generalities}

We now discuss some standard general properties of the stochastic dynamics such as its (finite-dimensional state space) ergodicity.

\paragraph{Continuous spins}
The Glauber--Langevin dynamics \eqref{eq: Langevin equation}
 is a Markov process with generator
\begin{equation}
\label{eq: Langevin generator}
  \Delta^H = \Delta -(\nabla H,\nabla) = e^{+H} (\nabla, e^{-H}\nabla)
\end{equation}
where
\begin{equation}
  \Delta = \sum_{x\in\Lambda} \ddp{^2}{\varphi_x^2} , \qquad
  (\nabla H,\nabla) = \sum_{x\in\Lambda} \ddp{H}{\varphi_x}\ddp{}{\varphi_x}.
\end{equation}
The state space $\R^\Lambda$ will often be denoted by $X$.
The distribution of the spin configuration evolves in time along the stochastic dynamics and we denote by $m_t$ the distribution 
at time $t$ starting from an initial measure $m_0$: given $F_0:X \to \R$,
\begin{equation}
\label{eq: distribution at time t}
  \E_{m_t} [F_0] = \E_{m_0}[F_t] 
  \quad \text{with}  \quad F_t(\varphi) = \TT_t F (\varphi) := \E_{\varphi_0=\varphi} \qa{F(\varphi_t)} ,
\end{equation}
where $\TT_t$ is the semigroup associated with the generator $\Delta^H$.
In particular, $F_t = \TT_t F$
solves the Kolmogorov backward equation 
\begin{equation}
\label{eq: Kolmogorov backward equation}
\ddp{}{t}F_t=\Delta^H F_t.
\end{equation}

Starting from the SDE, this can be verified using It\^o's formula.
The measure $\nu$, introduced in \eqref{eq: Gibbs measure}, 
is reversible with respect to this dynamics,
and the following integration by parts formula holds for sufficiently smooth $F$:
\begin{equation}
  \E_{\nu} [F(-\Delta^H G)] = \E_{\nu}[(\nabla F,\nabla G)].
\end{equation}
The right-hand side is the Dirichlet form:
\begin{equation}
\label{eq: Dirichlet continuous}
D_\nu(F,G) := \E_{\nu}[(\nabla F,\nabla G)] \quad \text{and} \quad   D_\nu(F) := D_\nu(F,F).
\end{equation}
In particular, the measure $\nu$ is invariant, i.e., if $\varphi_0$ is distributed
according to $\nu$ then $\varphi_t$ also is:
\begin{align}
\ddp{}{t}\E_{\nu} [F_t]  = \E_{\nu} [\Delta^H F_t]  = \E_{\nu} [(\nabla F_t,\nabla 1)]  = 0.
\end{align}
Moreover, we will always impose the following ergodicity assumption:
\begin{equation}
\label{e:ergodicity-continuous}
\forall F_0\in L^2(\nu): \qquad F_t \to \E_\nu[F_0] \qquad \text{in $L^2(\nu)$}.
\end{equation}
In particular, for any bounded smooth functions $F_0: X \to \R$ and $g: \R\to \R$,
\begin{equation}
  \label{e:ergodicity-continuous2}
  \lim_{t\to\infty} \E_\nu[g(F_t)] = g(\E_\nu[F_0]).
\end{equation}
As the next exercise shows, the ergodicity assumption is qualitative if $\Lambda$ is finite
and holds in all examples of interest.

\begin{exercise}
  Show that $\frac12 |\nabla H|^2-\Delta H \to\infty$ as $|\varphi|\to\infty$ implies that $-\Delta^H$ has discrete spectrum on $L^2(\nu)$
  with unique minimal eigenvalue $0$, and deduce \eqref{e:ergodicity-continuous} and \eqref{e:ergodicity-continuous2}.
\end{exercise}

For the discreteness of the spectrum,
one may observe that the multiplication operator $U= e^{\frac12 H}$ is an isometry from $L^2(\nu)$ onto $L^2(\R^N)$
that maps $-\Delta^H$ to the Schr\"odinger operator $-\Delta + W$
on $\R^N$ with $W = \frac14 |\nabla H|^2 -\frac12 \Delta H$.
The result therefore follows from the spectral theorem and the result that a Schr\"odinger operator with a potential $W \in L^1_{\rm loc}(\R^N)$
that is bounded below and satisfies $W\to\infty$ has compact resolvent
\cite[Theorem XIII.67]{MR0493421} (a version of Rellich's theorem).

For further general facts on stochastic dynamics in the continuous setting, we refer to \cite{MR3155209,MR1971582,MR1292280}.
Even though we will not need it, let us also mention that if the distribution of $m_t$ is written as $dm_t = G_t \, d\nu$ where $\nu=m_\infty$ is the invariant measure and $G_t = dm_t/d\nu$ is the density of $m_t$ relative to it, then
\begin{equation}
  \ddp{}{t}G_t = (\Delta^H)^* G_t = \Delta^H G_t
\end{equation}
where $(\Delta^H)^*=\Delta^H$ is the adjoint of $\Delta^H$ with respect to $\nu$.
This can also be expressed as an equation for $m_t$ (interpreted in a weak sense),
which is the Fokker--Planck equation:
\begin{equation}
  \ddp{m_t}{t} = \Delta m_t +(\nabla,m_t\nabla H) = (\nabla, m_t\nabla (\log m_t+H)).
\end{equation}

\paragraph{Discrete spins}

A similar structure can also be associated with discrete dynamics.
In particular, the Glauber dynamics of an Ising model is determined by its local jump rates
$c(\sigma,\sigma^x)$ satisfying the detailed balance condition
as in \eqref{eq: detailed balance}.
For all $F: \Omega \to \R$, where $\Omega=\{\pm 1\}^\Lambda$ is the finite state space,
the generator and Dirichlet form associated with the Glauber dynamics are
\begin{equation}
\label{eq: generator discrete}
  \Delta_c F(\sigma) = \sum_{x\in\Lambda} c(\sigma,\sigma^x) (F(\sigma^x)-F(\sigma))
\end{equation}
and
\begin{equation} 
\label{e:Dirichlet-discrete}
D_\nu (F)
  = -\sum_{\sigma\in\Omega} F(\sigma) \Delta_c F(\sigma) \nu(\sigma)
  = \frac12 \sum_{x\in\Lambda} \sum_{\sigma\in\Omega} c(\sigma,\sigma^x) (F(\sigma^x)-F(\sigma))^2\nu(\sigma),
\end{equation}
where we used the detailed balance condition for the second equality.
We will again write $D_\nu (F,F)$ for the quadratic form associated with $D_\nu (F)$ by polarisation.
As in the continuous setting, we will always impose an irreducibility assumption which is equivalent to the analogue of \eqref{e:ergodicity-continuous}:
\begin{equation} \label{e:ergodicity-discrete}
  \forall F_0: \Omega\to \R: \qquad F_t \to \E_{\nu}[F_0],
\end{equation}
where $F_t(\sigma) = e^{\Delta_c t}F_0(\sigma)= \E_{\sigma_0=\sigma}[F(\sigma_t)]$.
Indeed, assuming irreducibility, the convergence \eqref{e:ergodicity-discrete} is a consequence of the Perron--Frobenius theorem, see, e.g., \cite{MR1490046}.

Many choices of jump rates can be considered, but as long as the jump rates are uniformly bounded from above and below the different
Dirichlet forms are equivalent and the large time behaviour of the microscopic dynamics will be similar.
Often a natural choice of jump rates is that corresponding to the \emph{standard} Dirichlet form.
This choice formally corresponds
to $c(\sigma,\sigma^x)=1$ in \eqref{e:Dirichlet-discrete} which however are not the jump rates of the associated Markov process
because the constant function $1$ does not satisfy the detailed balance condition.
However, rewriting \eqref{e:Dirichlet-discrete} as
\begin{equation}
  D_\nu (F)
  = \frac12 \sum_{x\in\Lambda} \sum_{\sigma\in\Omega} \frac12\qa{c(\sigma,\sigma^x)+c(\sigma^x,\sigma)\frac{\nu(\sigma^x)}{\nu(\sigma)}} (F(\sigma^x)-F(\sigma))^2\nu(\sigma),
\end{equation}
we see that the standard Dirichlet form corresponds to the jump rates (satisfying detailed balance)
\begin{equation}
c(\sigma,\sigma^x) = \frac12\pa{1 + \frac{\nu(\sigma^x)}{\nu(\sigma)}}.
\end{equation}
Another popular choice are the heat-bath jump rates which are given by
\begin{equation}
  c^{\rm HB}(\sigma,\sigma^x)
  = \frac{\nu(\sigma^x)}{\nu(\sigma)+\nu(\sigma^x)}
  = \pa{1+\frac{\nu(\sigma)}{\nu(\sigma^x)}}^{-1}
\end{equation}
with the corresponding Dirichlet form 
\begin{equation}
  D_\nu^{\rm HB}(F)
  = \frac12
  \sum_{x\in\Lambda} \sum_{\sigma\in\Omega} \Psi(\nu(\sigma),\nu(\sigma^x)) (F(\sigma)-F(\sigma^x))^2, \qquad \Psi(a,b)=\frac{ab}{a+b}.
\end{equation}
The Metropolis jump rates correspond to $\Psi(a,b)= \min\{a,b\}$.
For further general discussion of Glauber dynamics in the discrete case, see \cite{MR1746301} and \cite{MR1490046,MR2283379}.

As in \eqref{eq: distribution at time t}, the distribution at time $t$ will be denoted by $m_t$.

\subsection{Log-Sobolev inequality}

In the above examples (with $\Lambda$ finite), one always has the qualitative ergodicity
$m_t \to \nu= m_\infty$
which amounts to an irreducibility condition.
One of the main questions we are interested in is how fast this convergence is.
A very good measure for the distance between $m_t$ and $\nu= m_\infty$, with many further
applications, is the relative entropy:
\begin{equation}
  \bbH(m_t |\nu) = \E_{\nu} [F_t \log F_t] = \ent_{\nu}(F_t ), \qquad F_t = \frac{dm_t}{d\nu}.
\end{equation}
More generally, when $F$ is nonnegative but does not necessarily satisfy $\E_\nu \qa{F} = 1$, define
\begin{equation}
\label{eq: def entropie}
  \ent_\nu (F)= \E_{\nu} [\Phi(F)]- \Phi(\E_\nu [F]), \qquad \Phi(x)= x\log x.
\end{equation}
The relative entropy is not symmetric and thus not a metric, but it has many very useful
properties making it a good quantity, and it controls the total variation distance by Pinsker's inequality:
\begin{equation} \label{e:Pinsker}
  \|m_t-\nu\|_{\rm TV}^2  \leq 2 \bbH(m_t|\nu).
\end{equation}
One of the most important properties is that the relative entropy decreases under the dynamics.
We begin with the continuous case.

\begin{proposition}[de Bruijn identity]
\label{prop: de Bruijn identity}
Consider the (continuous spin) stochastic dynamics \eqref{eq: Langevin equation} with invariant measure $\nu$ and 
Dirichlet form $D_{\nu}$ defined in \eqref{eq: Dirichlet continuous}.
Then for  $F_t(\varphi) = \E_{\varphi_0=\varphi} \qa{F(\varphi_t)}$ as in \eqref{eq: distribution at time t},%
\begin{equation} 
\label{e:dBidentity}
    \ddp{}{t} \ent_{\nu}(F_t) = -D_{\nu}(\log F_t, F_t) = 
    - I_\nu (F_t) \leq 0,
\end{equation}
where the Fisher information is defined in terms of the Dirichlet form \eqref{eq: Dirichlet continuous}:
\begin{equation} 
\label{e:Idef}
I_\nu (F_t)
  := \E_{\nu} \qa{\frac{(\nabla F_t)^2}{F_t}}
  = 4D_\nu \big(  \sqrt{F_t} \big).
\end{equation} 
\end{proposition}

\begin{proof}
  Since $\Phi(\E_\nu[F_t]) = \Phi(\E_\nu[F_0])$ is independent of $t$ and recalling that $\Delta^H$ is defined in~\eqref{eq: Langevin generator},
  \begin{align}
  \label{eq: Modified dirichlet}
    \ddp{}{t} \ent_{\nu}(F_t)
    &= \ddp{}{t} \E_{\nu}[\Phi(F_t)]
    = \E_\nu [\Phi'(F_t)\dot F_t]
    = \E_\nu [\Phi'(F_t)\Delta^H F_t]
      \nnb
    &= -D_\nu (\Phi'(F_t),F_t)
    = -D_\nu (\log F_t+1,F_t)
    = -D_\nu (\log F_t,F_t).
\end{align}
To complete the identity \eqref{e:dBidentity}, it is enough to notice that 
\begin{equation}  \label{e:Idifferent}
 D_\nu (\log F_t,F_t) = \E_{\nu}[(\nabla \log F_t, \nabla F_t)]
  = \E_{\nu} \qa{\frac{(\nabla F_t)^2}{F_t}}
  = 4 \E_{\nu} [(\nabla \sqrt{F_t})^2].
\qedhere
\end{equation} 
\end{proof}

Using the identity \eqref{e:dBidentity}, the decay of the entropy can be quantified
in terms of the \emph{log-Sobolev constant} which will be a key quantity we study.

\begin{definition}
\label{def: LSI}
A probability measure $\nu$ on $X=\bbR^N$,  satisfies the log-Sobolev inequality (LSI) with respect to $D_\nu$ if 
there is a constant $\gamma>0$ such that the following holds for any smooth, compactly supported function $F:X\to\R_+$:
\begin{equation} 
\label{e:LSI}
\ent_{\nu} (F) \leq \frac{2}{\gamma} D_\nu(\sqrt{F}).
\end{equation}
The largest  choice of $\gamma$ in this inequality is the log-Sobolev constant (with respect to $D_\nu$).
The normalisation with the above factor $2$ is convenient (see Proposition~\ref{prop:spectralgap}).
\end{definition}

The upshot is that the exponential decay of the relative entropy $\bbH(m_t|\nu)$ of the distribution $m_t$ (defined in \eqref{eq: distribution at time t}) along the flow of the Glauber-Langevin dynamics,
\begin{equation}
\label{eq: exponential decay}
  \bbH(m_t|\nu) \leq e^{-2\gamma t}\, \bbH(m_0|\nu),
\end{equation}
follows, by Gr\"onwall's lemma, from
\begin{equation}
\ddp{}{t} \bbH(m_t|\nu)\leq -2\gamma\, \bbH(m_t|\nu),
\end{equation}
which, by the de Bruijn identity, is a consequence of the log-Sobolev inequality \eqref{e:LSI}.
Thus the log-Sobolev constant provides a quantitative estimate on the speed of relaxation 
of the dynamics towards its stationary measure.
This is one of the main motivations for deriving the log-Sobolev inequality.

The log-Sobolev inequality \eqref{e:LSI} has also other consequences. Especially, it  is equivalent to the hypercontractivity of the associated Markov semigroup.
The hypercontractivity was, in fact, its original motivation \cite{MR0420249},
see also \cite{MR0210416,doi:10.1063/1.1664760}.

\begin{theorem}[Hypercontractivity \cite{MR0420249}]
\label{thm: Hypercontractivity}
  The measure $\nu$ satisfies the log-Sobolev inequality \eqref{e:LSI} with constant $\gamma$
  if and only if
  the associated semigroup $\TT_t$ is hypercontractive:
  \begin{equation}
    \|\TT_tF\|_{L^{q(t)}(\nu)} \leq \|F\|_{L^p(\nu)}
    \quad \text{with} \quad 
    \frac{q(t)-1}{p-1} = e^{2\gamma t}.
\end{equation}
\end{theorem}

We note that the hypercontractivity does not follow in this form from the modified log-Sobolev inequality
which will be introduced in \eqref{e:mLSI} below for dynamics with discrete state spaces.

\medskip

More generally, the log-Sobolev inequality is part of a larger class of functional inequalities.
In particular, it implies the spectral gap inequality (also called Poincar\'e inequality).

\begin{proposition}[Spectral gap inequality] \label{prop:spectralgap}
  The log-Sobolev inequality 
  with constant $\gamma$ implies the spectral gap inequality (also called Poincar\'e inequality)
  with the same constant:
  \begin{equation}
    \var_\nu [F] \leq \frac{1}{\gamma} \E_\nu \qa{ (\nabla F)^2 }.
  \end{equation}
\end{proposition}

The same conclusion holds assuming the modified log-Sobolev inequality \eqref{e:mLSI} below instead of the log-Sobolev inequality.
The proof follows by applying the log-Sobolev inequality to the test function $1 + \varepsilon F$ and then letting 
$\varepsilon$ tend to 0, see, e.g., \cite[Lemma~2.2.2]{MR1490046}.

\medskip

We refer to \cite{MR3155209} for an in-depth account on related functional inequalities and to \cite{MR1849347} for applications of the log-Sobolev inequality to the concentration of measure phenomenon.

\bigskip

For discrete spin models, the counterpart of Proposition \ref{prop: de Bruijn identity} is  the following proposition.
\begin{proposition}[de Bruijn identity]
\label{prop: de Bruijn identity discrete}
Consider the discrete dynamics with invariant measure $\nu$ \eqref{eq:  Ising model} and 
 Dirichlet form $D_{\nu}$ defined in \eqref{e:Dirichlet-discrete}.
 Then for  $F_t(\sigma) = \E_{\sigma_0=\sigma} \qa{F(\sigma_t)}$,
 \begin{equation} 
 \label{e:dBidentity bis}
    \ddp{}{t} \ent_{\nu}(F_t) = -D_{\nu}(\log F_t, F_t)   \leq 0.
  \end{equation}
\end{proposition}
The proof is identical to \eqref{eq: Modified dirichlet} replacing the continuous generator by $\Delta_c$
defined in \eqref{eq: generator discrete}.
As the chain rule no longer applies in the discrete setting, the Fisher information cannot be recovered.  Nevertheless the exponential decay of the entropy \eqref{eq: exponential decay}  can be established under  the modified log-Sobolev inequality (mLSI), i.e., if there is $\gamma >0$ such that for any function $F: \{ \pm 1\}^\Lambda \mapsto \bbR^+$:
\begin{equation} 
\label{e:mLSI}
\ent_{\nu} (F) \leq \frac{1}{2\gamma} D_{\nu} (\log F, F).
\end{equation}
In view of Exercise~\ref{ex:DlogDsqrt}, this inequality is weaker than the standard log-Sobolev inequality \eqref{e:LSI},
a point discussed in detail in \cite{MR2283379}.

\begin{exercise} \label{ex:DlogDsqrt}
  In the discrete case, the different quantities in \eqref{e:Idifferent} are not equal.
  Verify the inequality $4(\sqrt{a}-\sqrt{b})^2 \leq (a-b)\log(a/b)$ for $a,b>0$ and hence show $D_\nu(\log F,F)\geq 4D_\nu(\sqrt{F})$.
\end{exercise}

\subsection{Bakry--\'Emery theorem}

In verifying the log-Sobolev inequality for spins taking values in a continuous space, a very useful criterion is the Bakry--\'Emery theorem which 
applies to log-concave probability measures.

\begin{theorem}[Bakry--\'Emery \cite{MR889476}]
\label{thm:BE}
Consider a probability measure on $X = \R^N$ (or a linear subspace) of the form \eqref{eq: Gibbs measure} and assume 
that there is $\lambda>0$ such that as quadratic forms:
\begin{equation}
  \forall \varphi \in X:\qquad
  \He H(\varphi) \geq \lambda \id.
\end{equation}
Then the  log-Sobolev constant of $\nu$ satisfies $\gamma \geq \lambda$.
\end{theorem}

For quadratic $H$, one can verify (by a simple choice of test function) that in fact the equality $\gamma = \lambda$ holds.
An equivalent way to state the assumption $\He H(\varphi) \geq \lambda \id$ is to say that $H$
can be written as $H(\varphi)= \frac12 (\varphi,A\varphi)+V_0(\varphi)$ with
a symmetric matrix $A\geq \lambda \id$
and $V_0$ convex.

\begin{proof}[Proof of Theorem~\ref{thm:BE}]
The entropy \eqref{eq: def entropie} can be estimated by interpolation along the semigroup \eqref{eq: distribution at time t} of  the Langevin dynamics associated with $\nu$. Setting $F_t = P_t (F)$, we note that
\begin{equation}
\label{eq: time interpolation entropy}
\ent_\nu (F)= \E_{\nu} [\Phi(F)]- \Phi(\E_\nu [F])
= \E_{\nu} [\Phi(F_0) - \Phi(F_\infty ) ]
\end{equation}
where we used that the dynamics converges to the invariant measure \eqref{e:ergodicity-continuous2} which implies
\begin{equation}\label{eq: ergodicity assumption}
  \Phi(\E_\nu[F]) = \lim_{t\to\infty} \E_\nu[\Phi(F_t)].
\end{equation}
Indeed, we may assume that $F$ takes values in a compact interval $I \subset (0,\infty)$. 
By the positivity of the semigroup $F_t$ then takes values in $I$ for all $t\geq 0$, and we can replace $\Phi$ by a bounded smooth function $g$ that
coincides with $\Phi$ on $I$.
By the de Bruijn identity \eqref{e:dBidentity} (and using that $\E_\nu[F_t]$ is independent of $t$), therefore
\begin{equation}
  \label{eq: de Bruijn bis}
  \ent_\nu (F)
  = -\int_0^\infty dt \; \ddp{}{t} \E_{\nu} [\Phi(F_t)]
  = \int_0^\infty dt \; I_\nu(F_t).
\end{equation}

Provided that the Fisher information $I_\nu(F_t)$ introduced in \eqref{e:Idef} satisfies
  \begin{equation}
  \label{eq: 2nd derivative entropy}
I_\nu (F_t) \leq e^{-2 \lambda t} I_\nu (F_0),
  \end{equation}
the log-Sobolev inequality follows from \eqref{eq: time interpolation entropy} and \eqref{eq: de Bruijn bis}
by integrating in time:
\begin{equation}
\ent_\nu (F)  = \int_0^\infty dt \; I_\nu (F_t) 
\leq \frac{1}{2\lambda}  I_\nu (F_0)
= \frac{2}{\lambda} D_\nu(\sqrt{F}). 
\end{equation}

\medskip
To prove \eqref{eq: 2nd derivative entropy}, differentiate again:
\begin{equation}
  \ddp{}{t}I(\nu_t|\nu)
  = \ddp{}{t}\E_\nu \qa{\frac{(\nabla F_t)^2}{F_t}}
  = \E_\nu \qbb{ (\ddp{}{t}-\Delta^H) \frac{(\nabla F_t)^2}{F_t} },
\end{equation}
where we used that $\E_\nu[\Delta^HG] =0$ for every sufficiently nice function $G: X \to \R$.
It is an elementary but somewhat tedious exercise to verify that
\begin{equation} \label{e:BE-pf}
  (\ddp{}{t}-\Delta^H) \frac{(\nabla F_t)^2}{F_t}
  = -2 F_t \qbb{ \underbrace{|\He \log F_t|_2^2}_{\geq 0} + (\nabla \log F_t, \underbrace{\He H(\varphi)}_{\geq \lambda \id} \nabla \log F_t)}.
\end{equation}
Hence
\begin{equation}
  \ddp{}{t} I_\nu (F_t)
  \leq
  -2\lambda\E_\nu [F_t(\nabla \log F_t, \nabla \log F_t)]
  =
  -2\lambda\E_\nu \qa{\frac{(\nabla F_t)^2}{F_t}}
  = -2\lambda I_\nu (F_t)
\end{equation}
which implies the claim \eqref{eq: 2nd derivative entropy}.
\end{proof}

\subsection{Decomposition and properties of the entropy}
\label{sec: Decomposition and properties of the entropy}

The Bakry--\'Emery criterion (Theorem \ref{thm:BE}) implies the validity of the log-Sobolev inequality for all the Gibbs measures with strictly convex potentials. For more general measures, the  log-Sobolev inequality 
is often derived by decomposing the entropy thanks to successive conditionings.
We are going to sketch this procedure below.

Assume that the expectation under $\nu$ is of the form
\begin{equation}
\label{eq: decomposition nu1 nu2}
\E_\nu  \qa{F} : = \E_\nu  \qa{F (\varphi_1,\varphi_2) } =\E_2 \qa{ \E_1 \qa{F  (\varphi_1,\varphi_2) \big| \varphi_2 } } =  \E_2 \qa{ \E_1 \qa{F}},
\end{equation}
where $\E_1 [ \cdot ] :=\E_1 [ \cdot | \varphi_2 ]$ is the conditional measure with respect to the variable $\varphi_2$. 
Then the entropy can be split into two parts:
\begin{align}
\label{eq: entropy decomposition}
  \ent_{\nu}(F)
  &= 
    \E_{\nu} [\Phi(F)]- \Phi(\E_\nu [F])
           \nnb
  &= \E_2 \big[ \underbrace{\E_1[\Phi(F)] -\Phi(\E_1 [F])}_{\ent_1 (F) } \big] 
  + {\underbrace{\E_2 \qa{\Phi(\E_1[F])} -
  \Phi(\E_2 \qa{\E_1 [F]} )}_{\ent_2(\E_1 [F])}}.
\end{align}
Given $\varphi_2$, the first term involves the relative entropy of a simpler measure $\E_1 ( \cdot | \varphi_2)$   as the integration refers only to the coordinate $\varphi_1$:
\begin{equation}
\ent_1 (F) (\varphi_2) =  \E_1[\Phi(F)  | \varphi_2 ] -\Phi(\E_1 \qa{ F  | \varphi_2} ).
\end{equation}
The strategy is to estimate this term (uniformly in $\varphi_2$) by the desired Dirichlet form acting only on $\varphi_1$.
The second term $\E_2 \qa{\Phi(\E_1[F])}$ is more complicated because the expectation $\E_1 [ F( \cdot, \varphi_2)  | \varphi_2 ]$ is inside the relative entropy. 
For a product measure $\nu = \nu_1 \otimes \nu_2$, this term can be estimated easily (as recalled in Example \ref{example: tensorisation} below).
In this way, the log-Sobolev inequality for $\nu = \nu_1 \otimes \nu_2$ is reduced to establishing log-Sobolev inequalities for the simpler measures  $\nu_1$ and $\nu_2$.

In general, the conditional expectations are intertwined and the second term $\ent_2(\E_1 [F])$ is much more difficult to estimate. 
There are  two general strategies: either one also bounds this term   by the desired Dirichlet form (and thus one somehow has to move the expectation out 
of the entropy) or one bounds it by $\kappa \ent_\nu (F)$ with $\kappa < 1$. In the latter case, the estimate reduces to the first term,
at the expense of an overall factor $(1-\kappa)^{-1}$.

For a given  measure $\nu$, the entropy decomposition~\eqref{eq: entropy decomposition} 
can be achieved with different choices of the measures $\E_1, \E_2$. 
The optimal  choice depends on the structure of the measure $\nu$. 
In this survey, we focus on Gibbs measures of the form \eqref{eq: Gibbs measure} which arise naturally in statistical mechanics. 
The renormalisation group method 
constitutes a framework to study such Gibbs measures 
(see \cite{MR3969983} for an introduction and references)
and provides strong insight on a good entropy decomposition. 
This is the core of the method presented in Section \ref{sec_gaussian_integration_and_polchinski}, 
which is based on the Polchinski equation, a continuous version of the renormalisation group.

\begin{example}[Tensorisation]
\label{example: tensorisation}
  Assume that probability measures $\nu_1$ and $\nu_2$ satisfy log-Sobolev inequalities with the constants
  $\gamma_1$ and $\gamma_2$. Then the product measure $\nu = \nu_1\otimes \nu_2$ also satisfies a log-Sobolev inequality
  with constant $\gamma = \min  \{ \gamma_1,\gamma_2 \}$ (with the natural Dirichlet form on the product space).
\end{example}

\begin{proof}
For simplicity, assume that $\nu_1, \nu_2$ are probability measures on $\bbR$ and denote by $\E_1, \E_2$  their expectations so that $\E_\nu \qa{F} =\E_2 \qa{ \E_1 \qa{F} }$ for functions $F( \varphi_1, \varphi_2)$.
As discussed  in \eqref{eq: entropy decomposition}, the entropy can be decomposed as 
\begin{equation}
  \ent(F) = \E_2[\ent_1(F)] + \ent_2(G) 
  \quad \text{with} \quad G(\varphi_2) = \E_1[F(\cdot,\varphi_2)].
  \end{equation}
  The log-Sobolev inequalities for $\E_1$ and $\E_2$ imply that for $\gamma = \min \{ \gamma_1,\gamma_2 \}$:
\begin{equation}
2\gamma \ent(F) \leq \E_2[D_1(\sqrt{F})] + D_2(\sqrt{G}).
\end{equation}
It remains to recover the Dirichlet form associated with the product measure $\nu$: 
\begin{equation}
D_\nu (\sqrt{F}) = \E_\nu \qa{ (\partial_{\varphi_1} \sqrt{F})^2 + (\partial_{\varphi_2} \sqrt{F})^2 }
\label{eq: dirichlet product}.
\end{equation}
The first derivative is easily identified:
\begin{equation}
\E_2[D_1(\sqrt{F})] = \E_2 \qa{\E_1 \qa{ (\partial_{\varphi_1} \sqrt{F})^2} }
= \E_\nu \qa{(\partial_{\varphi_1} \sqrt{F})^2}.
\end{equation}
For the second derivative:
\begin{equation}
\partial_{\varphi_2} \sqrt{G(\varphi_2)} = \frac{1}{2 \sqrt{G(\varphi_2)}} \E_1[ \partial_{\varphi_2} F(\cdot,\varphi_2)] = \frac{1}{\sqrt{G(\varphi_2)}} 
\E_1[ \sqrt{F(\cdot,\varphi_2)} \; \partial_{\varphi_2} \sqrt{ F(\cdot,\varphi_2)} ] ,
\end{equation}
so that by Cauchy-Schwarz inequality,  we deduce that 
\begin{equation}
D_2(\sqrt{G}) = \E_2 \qa{ \big( \partial_{\varphi_2} \sqrt{G(\varphi_2)} \big)^2 } \leq 
\E_2 \qa{ \E_1 \qa{ \left( \partial_{\varphi_2} \sqrt{ F(\cdot,\varphi_2)} \right)^2 } }
= \E_\nu \qa{  (\partial_{\varphi_2} \sqrt{F})^2 }.
\label{eq: interchange product case}
\end{equation}
This reconstructs the Dirichlet form \eqref{eq: dirichlet product} and completes the proof.

A similar argument applies in the discrete case, see for example \cite[Lemma~2.2.11]{MR1490046}.
More abstractly, the tensorisation  of the log-Sobolev constant also follows
from the equivalence between the log-Sobolev inequality and hypercontractivity (which tensorises more obviously). \end{proof}

We conclude this section by stating useful variational characterisations of the entropy.
\begin{proposition}[Entropy inequality]
\label{prop: variational characterisation entropy}
The entropy of a function $F \geq 0$ can be rewritten as   
\begin{equation}
\label{e:entineq}
    \ent_\nu(F) = \sup \hbb{\E_\nu[FG] : \quad \text{Borel functions $G$ such that } \E_\nu[e^G] \leq 1}
  \end{equation}
  with equality if $G = \log(\frac{F}{\E_\nu[F]})$, or as
  \begin{equation} 
  \label{e:entineq3}
    \ent_\nu(F) = \sup \hbb{\E_\nu [F \log F -F \log t - F + t]: t> 0 }
  \end{equation}
  with equality if $t = \E_\nu[F]$.
  Finally, one has (also called the entropy inequality):
  \begin{equation} 
  \label{e:entineq2}
  \ent_\nu(F) = \sup \hbb{ \E_\nu [FG] - \E_\nu[F] \log \E_{\nu}[e^G]:
  \quad \text{Borel functions $G$}}
  \end{equation}
  with equality if $G = \log F$.
\end{proposition}

\begin{proof}
Since $\ent_\nu(F) = \E_\nu \qa{ F\log \frac{F}{\E_\nu[F]} }$,
to show \eqref{e:entineq},
it is enough to consider the case $\E_\nu[F] = 1$, by homogeneity of both sides.
Applying Young's inequality
\begin{equation} 
\label{eq: Young's inequality}
\forall a \geq 0, b \in \bbR: \qquad 
a b \leq a \log a - a + e^b ,
\end{equation}
with $\E_\nu[e^G] \leq 1$, we get 
\begin{equation}
\E_\nu[FG] \leq \ent_\nu(F) -1 + \E_\nu[e^G] \leq \ent_\nu(F).
\end{equation}
This implies \eqref{e:entineq} as the converse inequality holds with $G = \log F$.

The variational formula \eqref{e:entineq3} follows directly from Young's inequality by choosing $a=\E_\nu[F]$ and $b = \log t$:
\begin{equation}
  \ent_\nu(F) = \E_\nu[F\log F] - a \log a \leq \E_\nu[F\log F] - ab -a +e^b
  = \E_\nu[F \log F - F \log t - F + t],
\end{equation}
where again \eqref{e:entineq3} follows since equality holds with $t=\E_\nu[F]$.

To show \eqref{e:entineq2}, we may again assume $\E_\nu[F]=1$. Then apply Jensen's inequality with respect to the probability measure $d\nu^F = F\, d\nu$:
\begin{equation}
  \log \E_{\nu}[e^G] = \log \E_{\nu^F}[F^{-1}e^G] \geq \E_{\nu^F}[\log(F^{-1}e^G)] = - \E_\nu[F\log F] + \E_\nu[FG],
\end{equation}
with equality if $G = \log F$.
\end{proof}

The Holley--Stroock criterion for the log-Sobolev inequality is a simple consequence of \eqref{e:entineq3}, see e.g.,
the presentation in \cite{MR1292280}.

\begin{exercise}[Holley--Stroock criterion]
\label{exo: Holley--Stroock criterion}
  Assume a measure $\nu$ satisfies the log-Sobolev inequality with constant $\gamma$. Then the measure
  $\nu^F$ with $d\nu^F/d\nu =F$ satisfies a log-Sobolev inequality with
  constant $\gamma^F \geq (\inf F/\sup F)\gamma$.
\end{exercise}

\subsection{Difficulties arising from statistical physics perspective}
\label{sec: Difficulties arising from statistical physics perspective}

To explain the difficulties arising in the derivation of log-Sobolev inequalities and to motivate our set-up of renormalisation,
we are going to consider lattice spin systems with continuous spins and Hamiltonian of the form \eqref{eq: specific Hamiltonian}. 
The strength of the interaction is tuned by the parameter $\beta \geq 0$ and the Gibbs measure 
\eqref{eq: Gibbs measure} has a density on $\bbR^\Lambda$ of the form 
\begin{equation}
  \nu (d\varphi) 
  \propto \exp \qa{ - \frac{\beta}{4} \sum_{x,y \in \Lambda, \atop x \sim y} (\varphi_x - \varphi_y)^2 
    -\sum_{x \in \Lambda} V(\varphi_x) } \prod_{x \in \Lambda} d \varphi_x.
\label{eq: specific Gibbs measure}
\end{equation} 
Further examples will be detailed in Section~\ref{sec: Applications to Euclidean field theory}. 

We are interested in the behaviour of the measure~\eqref{eq: specific Gibbs measure} in the limit where the the number of sites $|\Lambda|$ (and thus the dimension of the configuration space) is large.
In this limit, when the potential $V$ is not convex, 
the measure can have one or more phase transitions at critical values of the parameter $\beta$. 
These phase transitions separate regions of values of $\beta$ between which the measure $\nu$ has a different correlation structure, 
different concentration properties, and so on. The speed of convergence of the associated Glauber dynamics~\eqref{eq: general Glauber} is also affected. 
See the book~\cite{MR3752129}
or \cite[Chapter~1]{MR3969983}
for background on phase transitions in statistical mechanics.

To analyse the log-Sobolev inequality for the measure~\eqref{eq: specific Gibbs measure}, 
note that the lack of convexity precludes the use of the Bakry--\'Emery criterion (Theorem \ref{thm:BE}), 
and the Holley-Stroock criterion (Exercise~\ref{exo: Holley--Stroock criterion}) is not effective due to the large dimension of the configuration space when $\beta>0$.

On the other hand, when $\beta=0$, the Gibbs measure is a product measure and the log-Sobolev inequality holds 
uniformly in $\Lambda$, 
with the same constant as for the single spin $|\Lambda|=1$ measure (Example~\ref{example: tensorisation}). 
This tensorisation property has been generalised for $\beta$ small enough
in terms of mixing conditions
and for some spin systems up to the critical value $\beta_c$, see \cite{MR1746301} for a review.
Indeed, for $\beta$ small, the interaction between the spins is small, 
in the sense that one can show that correlations between spins decay exponentially in their distance.
At distances larger than a correlation length $\xi_\beta < \infty$ approximate independence between the spins is then recovered. 
By splitting the domain $\Lambda$ into boxes (of size larger than $\xi_\beta$), and using appropriate conditionings the system can  be analysed as a renormalised model of weakly interacting spins \cite{MR1715549}. 
In so-called second order phase transitions, 
when $\beta$ approaches the critical $\beta_c$ the correlation length diverges as a function of $\beta_c-\beta$, and so does the inverse log-Sobolev constant, i.e., the dynamics slows down (as can usually be verified by simple test functions in the spectral gap or log-Sobolev inequality).
Spins are thus more and more correlated for $\beta$ close to $\beta_c$.
Nevertheless  in some cases \cite{MR1746301,MR3020173,MR3486171} the  strong dynamical mixing properties were
derived up to the critical value $\beta_c$ by using a strategy which however can be seen as a (large) perturbation of the product case with respect to $\beta >0$. 
For this reason,  it seems difficult to extract the precise divergence of the log-Sobolev constant near $\beta_c$ with this type of approach.

To study the detailed static features of measures of the form~\eqref{eq: specific Gibbs measure} close to $\beta_c$, 
different types of renormalisation schemes have been devised with an emphasis on the Gaussian structure of the interaction.
In many cases one expects that the long range structure at the critical point is well described in terms of a Gaussian free field \cite{Wilson_Kogut_1974}.
Compared with the previously mentioned approaches, the perturbation theory no longer uses the product measure as a reference, but  the Gaussian free field.
In the following, this structure will serve as a guide to decompose the entropy as alluded to in \eqref{eq: entropy decomposition}.
Before describing this procedure in Section~\ref{sec_gaussian_integration_and_polchinski}, 
the elementary example of the Gaussian free field,
which illustrates the difficulty of many length scales equilibrating at different rates, 
is presented in Example~\ref{example_free_field} below.

\subsection{Difficulties arising from continuum perspective}
\label{sec: Difficulties arising from continuum perspective}

The long-distance problem discussed in the previous subsection is closely related to the short-distance problem
occurring in the study of continuum limits as they appear in quantum field theory and weak interaction limits,
which also arise as invariant measure of (singular) SPDEs.
In field theory, one is interested in the physical behaviour of a measure that is defined not on a lattice,
but in the continuum (say on $L\T^d$ with $\T^d = [0,1)^d$ the $d$-dimensional torus), formally reading:
\begin{equation}
\nu_L(d\varphi)
\propto e^{-H_L(\varphi)}\prod_{x\in L\T^d}d\varphi_x
.
\label{eq: continuum measure}
\end{equation}
Now $\varphi$ should be a (generalised) function from $L\T^d$ to $\R$, and a typical example for $H_L$ would be the continuum $\varphi^4$ model, defined for $g >0$ and $r \in\R$ by:
\begin{equation}
\label{eq: continuous potential}
H_L(\varphi) 
=
\frac{1}{2}\int_{L\T^d}|\nabla \varphi|^2\, dx
+ \int_{L\T^d}\Big[\frac{g}{4}\varphi(x)^4 + \frac{r}{2}\varphi(x)^2\Big]\, dx
.
\end{equation}
Of course, the formal definition~\eqref{eq: continuum measure} does not make sense as it stands,
and a standard approach  to understand such measures is
as a limit of measures defined on lattices $\Lambda_{\epsilon,L} = L\T^d\cap \epsilon\Z^d$ with a vanishing $\epsilon$:
\begin{equation}
\nu_{\epsilon,L}(d\varphi)
\propto
e^{-H_{\epsilon,L}(\varphi)}\prod_{x\in \Lambda_{\epsilon,L}}d\varphi_x,
\end{equation}
for a discrete approximation $H_{\epsilon,L}$ of $H_L$ of the form
\begin{equation} 
H_{\epsilon,L} (\varphi)  = \frac{\epsilon^{d-2}}{4} \sum_{y\sim x}\big(\varphi_y-\varphi_x\big)^2
+ \epsilon^d\sum_{x\in\Lambda_{\epsilon,L}} V^\epsilon(\varphi_x) ,
\label{eq: approx HL}
\end{equation}
where the potential $V^\epsilon$ is of the form \eqref{eq: Ginzburg--Landau potential}.
In the example of the $\varphi^4$ model, it turns out that such a limit can be constructed if $d<4$,
but in $d\geq 2$
the coefficient $\mu$ of the potential must be tuned correctly as a function of $\epsilon \to 0$,
see Section~\ref{sec: Applications to Euclidean field theory} for details.
This tuning is known as addition of ``counterterms'' in quantum field theory. These are the infamous infinities arising there.
The relation to the statistical physics (long-distance) problem is that these limits correspond to statistical physics models near a phase transition, with
scaled (weak) interaction strength ($\epsilon^d g \to 0$ as $\epsilon \to 0$).
In particular, due to the counterterms, the resulting measures are usually again very non-convex microscopically,
precluding the use of the Bakry--\'Emery theory and the Holley--Stroock criteria.
Nonetheless
the regularisation parameter $\epsilon$ is not expected to have any influence on the physics of the model, 
in the sense that the existence of a phase transition, the speed of the Glauber dynamics, concentration properties, and so on,
should all be uniform in the small scale parameter $\epsilon$ and depend only on the large scale parameter $L$.  
Techniques to control the regularisation parameter $\epsilon$ are often simpler than for the large scale problem near the critical point,
but they are also based on renormalisation arguments relying on comparisons with the Gaussian free field 
(corresponding to a quadratic $H_{\epsilon,L}$).

\bigskip
Finally,
the problem of relaxation at different scales is illustrated next in this simple model.

\begin{example}[Free field dynamics]
\label{example_free_field}
Consider now the Gaussian free field dynamics corresponding to $V^\epsilon =0$ in \eqref{eq: approx HL}:
\begin{equation}
\label{eq: GFF example dynamics}
  d\varphi_t = -A\varphi_t \, dt + \sqrt{2} dW_t,
\end{equation}
where $A$ is the Laplace operator on $\Lambda_{\epsilon,L}$ as in~\eqref{eq: approx HL}
and the white noise is defined with respect to the inner produce $\epsilon^d \sum_{x\in\Lambda_{\epsilon,L}} u_xv_x$, i.e., each $W_t(x)$
is a Brownian motion of variance $\epsilon^{-d}$.
On the torus $\Lambda_{\epsilon,L}= L\T^{d} \cap \epsilon\Z^d$ of mesh size $\epsilon$ and side length $L$, the eigenvalues of the Laplacian are
\begin{equation}
p \in \Lambda_{\epsilon,L}^* = (-\frac{\pi}{\epsilon},\frac{\pi}{\epsilon}]^d \cap \frac{2\pi}{L}\Z^d
\qquad   \lambda(p) = \epsilon^{-2}\sum_{i=1}^d 2(\cos(\epsilon p_i)-1) \underset{|p|\lesssim 1}{\approx} -|p|^2.
\end{equation}
All Fourier modes of $\varphi$ evolve independently according to Ornstein--Uhlenbeck processes:
\begin{equation}
\label{eq: modes eigenvector}
p \in \Lambda_{\epsilon,L}^*: \qquad
  d\hat\varphi(p) = -\lambda(p)\hat\varphi(p) \, dt + \sqrt{2} d\hat W_t(p),
\end{equation}
where the $\hat W(p) = (\hat W_t(p))_{t}$ are independent standard Brownian motions for $p \in \Lambda^*$.
In particular, small scales corresponding to $|p| \gg 1$ converge very quickly to equilibrium,
while the large scales $|p| \ll 1$ are slowest. Thus the main contribution to the log-Sobolev constant comes from the large scales
and we expect that a similar structure remains relevant in many interacting systems close to a critical point.
\end{example}

In both the statistical and continuum perspectives, for measures with 
an interaction $V_0(\varphi) = \sum_{x\in\Lambda} V(\varphi_x) \neq 0$ on top of the free field interaction,
the main difficulties result from the simple fact that the local (in real space) interaction do not interact well with the above Fourier decomposition.
The Polchinski flow that we will introduce in the next section can be seen as
a replacement for the Fourier decomposition, in which the Fourier variable $p$ takes the role of scale, by a
smoother scale decomposition.

\section{Gaussian integration and the Polchinski equation}
\label{sec_gaussian_integration_and_polchinski}

In this section, we first review abstractly a continuous renormalisation procedure,
which goes back to Wilson \cite{Wilson_Kogut_1974} and Polchinski \cite{polchinski1984269, MR914427} in physics
in the context of equilibrium phase transitions and quantum field theory (viewed as an problem of statistical mechanics in the continuum).
We then explain how the entropy of a measure can be decomposed by this method in order to derive a 
log-Sobolev inequality via a multiscale Bakry--\'Emery criterion.

\subsection{Gaussian integration}

For $C$ a positive semi-definite matrix on $\R^N$, we denote by
$\Pg_C$ the corresponding Gaussian measure  with covariance $C$ and by $\Eg_C$ its expectation.
The measure $\Pg_C$ is supported on the image of $C$.
In particular, if $C$ is strictly positive definite on $\R^N$,
\begin{equation} \label{e:Gauss-E}
  \Eg_C[F] \propto \int_{\R^N} e^{-\frac12 (\zeta,C^{-1}\zeta)} F(\zeta)\, d\zeta.
\end{equation}
A fundamental property of the Gaussian measure is its semigroup property: if $C=C_{1}+C_2$ with  $C_1, C_2$ also positive semi-definite then
\begin{equation} 
\label{e:Gauss-conv}
  \Eg_C[F(\zeta)] =
  \Eg_{C_2}[\Eg_{C_1}[F(\zeta_1+\zeta_2)]],
\end{equation}
corresponding to its probabilistic interpretation that if $\zeta_1$ and $\zeta_2$ are independent Gaussian random variables
then $\zeta_1+\zeta_2$ is also Gaussian and the covariance of $\zeta_1+\zeta_2$ is the sum of the covariances of $\zeta_1$ and $\zeta_2$.

As discussed in Section~\ref{sec: Decomposition and properties of the entropy}, recall that
our goal is to decompose the entropy of a measure by splitting this measure into simpler parts as in \eqref{eq: entropy decomposition}.
The above Gaussian decomposition will be a basic step for this. 
As we have seen in Example \ref{example_free_field}, the dynamics of a spin or particle system close to a phase transition will depend on a very large number of modes and it will be necessary to iterate the decomposition 
\eqref{e:Gauss-conv} many times in order to decouple all the relevant modes.
In fact, it is even convenient to introduce a  continuous version of the decomposition~\eqref{e:Gauss-conv}, as follows.

\medskip

For a covariance matrix $C$ as above, define an associated Laplace operator $\Delta_C$ on $\R^N$:
\begin{equation} 
\label{e:DeltaC}
  \Delta_C = \sum_{i,j} C_{ij} \ddp{^2}{\varphi_i\partial \varphi_j},
\end{equation}
and write $(\cdot,\cdot)_C$ for the inner product associated with the covariance $C$:
\begin{equation}
\label{eq: scalar product C}
  (u,v)_C = \sum_{i,j} C_{ij} u_iv_j
  \quad \text{and} \quad 
  |u|_C^2 = (u)_C^2 = (u,u)_C.
\end{equation}
The standard scalar product  is denoted by $(u,v) = \sum_i u_i v_i$.

Let $t\in [0,+\infty] \mapsto C_t$ be a function of positive semidefinite matrices on $\R^N$ increasing continuously as quadratic forms to a matrix $C_\infty$.
More precisely, we assume that $C_t = \int_0^t \dot C_s \, ds$ for all $t$,
where $t\mapsto \dot C_t$ is a bounded  cadlag (right-continuous with left limits) function with values in the space of positive semidefinite matrices
that is the derivative of $C_t$ except at isolated points. We say that $C_\infty = \int_0^\infty \dot C_s\, ds$ is a \emph{covariance decomposition}
and 
 write $X  \subset \R^N$ for the image of $C_\infty$.
We emphasise that the (closed) interval $[0,+\infty]$ parametrising the covariances has no special significance and that all constructions
will be invariant under appropriate reparametrisation. For example, one can equivalently use $[0,1]$.

\begin{proposition} \label{prop:heat}
  For a $C^2$ function $F: X \to \R$, let $F_t = \Pg_{C_t} * F$, i.e., $F_t(\varphi)= \Eg_{C_t} \qa{F(\varphi+\zeta)}$.
  Then for all $t$ which are not discontinuity points of $\dot C_t$,
  \begin{equation}
    \ddp{}{t} F_t = \frac12 \Delta_{\dot C_t} F_t, \qquad F_0 = F.
  \end{equation}
Thus the Gaussian measures $\Pg_{C_t}$ satisfy the heat equation
\begin{equation}
\label{eq: evolution proba PgCt}
  \ddp{}{t} \Pg_{C_t} = \frac12 \Delta_{\dot C_t} \Pg_{C_t},
\end{equation}
interpreted in a weak sense if $C_t$ is not strictly positive definite.
\end{proposition}

\begin{proof}
  In the case that $C_{t-\epsilon}$ is strictly positive definite for some $\epsilon>0$ (and by monotonicity then also for all larger times),
  this is a direct computation from \eqref{e:Gauss-E}.
  For $t$ that are not discontinuity points of $\dot C_t$, the image $X_s$ of $C_s$ is independent of $s \in [t-\epsilon,t+\epsilon]$
  and one has the representation \eqref{e:Gauss-E} on $X_t$.
\end{proof}

Alternatively, one can prove the proposition using It\^o's formula.
For any $\varphi \in \R^N$, define the process 
\begin{equation}
\label{eq: continuous gaussian decomposition}
\forall t \geq 0: \quad    \zeta_t = \varphi + \int_0^t  \sqrt{\dot C_s} \; dB_s \in \R^N,
\end{equation}
where $(B_s)_s$ is a Brownian motion taking values in $\R^N$. By construction $\zeta_t$ is a Gaussian variable with mean $\varphi$ and variance $C_t = \int_0^t \dot C_s \, ds$. In particular $F_t(\varphi) = \bbE \qa{F(\zeta_t)}$ and by It\^o's formula,
\begin{equation}
\forall t \geq0: \qquad 
\ddp{}{t} F_t(\varphi) = \bbE \qa{ \frac12 \Delta_{\dot C_t}  F(\zeta_t)} = \frac12 \Delta_{\dot C_t}  \Eg_{C_t} \qa{F(\varphi+\zeta)}
= \frac12 \Delta_{\dot C_t}  F_t(\varphi),
\end{equation}
with the derivative interpreted as the right-derivative at the discontinuity points of $\dot C_t$.

Note that the decomposition \eqref{eq: continuous gaussian decomposition} is the natural extension of the discrete decomposition $\zeta = \zeta_1 + \zeta_2$. 
Given a covariance matrix $C$, many decompositions are possible, such as:
\begin{equation}
C = \int_0^\infty \dot C_s \, ds 
\quad \text{with} \quad  \dot C_s = C \; 1_{s \in [0,1)}.
\end{equation}
For a  given model from statistical mechanics, it will be important to adjust the decomposition according to the specific spatial structure (of $\Lambda$) of this model. 
In Example \ref{example_free_field}, the Gaussian free field has covariance matrix $A^{-1}$
with $A$ the discrete Laplace operator as in \eqref{eq: discrete Laplace}. 
There are many decompositions of $A^{-1}$ of the form $\int_0^\infty \dot C_s \, ds$ and the best choice will dependent on the application.
Nevertheless a suitable decomposition should capture the mode structure of the  decomposition \eqref{eq: modes eigenvector} in order to separate the different scales in the dynamics. 
Indeed, the key idea of a renormalisation group approach is to integrate the different scales one after the other in order.
This is especially important in strongly correlated systems, in which the different scales do not decouple,
and integrating some scales has an important effect on the remaining scales: the interaction potential will get renormalised.
We refer to Section \ref{sec: Applications} for several applications.

\subsection{Renormalised potential and Polchinski equation}
\label{sec:LSI-setup}

In this section, we define the Polchinski flow and analyse its structure. 
A simple explicit example is worked out in Example~\ref{sec:1d_example} below.
We will focus on probability measures $\nu_0$ supported on a linear subspace $X \subset \R^N$.
By considering the measure $\nu_0(A-a)$ for $a\in \R^N$,
this also includes measures supported on an affine subspace which is of interest for 
conservative dynamics.
For generalisations to non-linear spaces, see Section~\ref{sec:geometry}.

Let $C_\infty = \int_0^\infty \dot C_t\, dt$ be a covariance decomposition,
and consider a probability measure $\nu_0$ on $X$ with expectation 
 given by
\begin{equation}
\label{e:nu0-Cinfty}
\E_{\nu_0} [F]
  \propto
  \Eg_{C_\infty} \qa{e^{- V_0(\zeta)} F(\zeta)} ,
\end{equation}
with a potential $V_0:X \to \R$, where the Gaussian expectation acts on the variable $\zeta$.
To avoid technical problems, we always assume in the following that $V_0$ is bounded below.
We are going to use the Gaussian representation introduced in the previous subsection in order to
decompose the measure $\nu_0$. For this, let us first introduce some notation.

\begin{definition} 
\label{def:renormpot}
For $t>s>0$, $F: X \to \R$ bounded, and $\varphi \in X$, define:
\begin{itemize}
\item
the \emph{renormalised potential} $V_t$:
\begin{equation} 
\label{e:V-def}
V_t(\varphi) = - \log \Eg_{C_t} \qa{e^{-V_0(\varphi+\zeta)}};
\end{equation}
\item the  \emph{Polchinski semigroup} $\PP_{s,t}$:
\begin{equation} 
  \label{e:P-def-bis}
  \PP_{s,t}F(\varphi) = e^{V_t(\varphi)} \Eg_{C_t-C_s}
  \qa{e^{-V_s(\varphi+\zeta)} F(\varphi+\zeta)};
\end{equation}
\item the \emph{renormalised measure} $\nu_t$:
\begin{equation} 
  \label{e:nu-def-bis}
  \E_{\nu_t} [F] = \PP_{t,\infty}F(0) = e^{V_\infty(0)} \Eg_{C_\infty-C_t} \qa{e^{-V_t(\zeta)} F(\zeta)},
\end{equation}
\end{itemize}
where all the Gaussian expectations apply to $\zeta$.
\end{definition}

We stress that the renormalised measure $\nu_t$ evolving according to the Polchinski semigroup 
is different from the measure $m_t$ in  \eqref{eq: distribution at time t} evolving along the flow of the Langevin 
dynamics (which we will not discuss directly in this section).

Note that in \eqref{e:nu-def-bis}, $e^{+V_\infty(0)}$ is the normalisation factor of the probability measure $\nu_t$.
More generally, the function $V_\infty$ is equivalent to the moment generating function of the measure $\nu_0$: 
changing variables from $\zeta$ to $\zeta + C_\infty h$,
\begin{align} \label{e:Vinfty}
  V_\infty(C_\infty h)
  = -\log \Eg_{C_\infty}[e^{-V_0(C_\infty h+\zeta)}]
  &= \frac12 (h, C_\infty h) -\log \Eg_{C_\infty}[e^{-V_0(\zeta)}e^{(h,\zeta)}]
    \nnb
  &= \frac12 (h, C_\infty h) -\log \E_{\nu_0}[e^{(h,\zeta)}] + V_\infty(0).
\end{align}

The renormalised measure $\nu_t$ is related to $\nu_0$ by the following identity.
\begin{proposition}
\label{prop: mesure decomposition 1 pas}
For $t \geq 0$ and any $F : X \mapsto \R$ such that the following quantities make sense,
\begin{equation} 
\label{e:polchinski-semigroup level 0}
\E_{\nu_0} \qa { F } =   \E_{\nu_t} \qa{ \PP_{0,t} F } .
\end{equation}
\end{proposition}

\begin{proof}
Starting from \eqref{e:nu0-Cinfty}, from the Gaussian decomposition \eqref{e:Gauss-conv} we get
\begin{align}
\Eg_{C_\infty} \qa{e^{- V_0(\zeta)} F(\zeta)} 
&= \Eg_{C_\infty - C_t} \qa{ \Eg_{C_t} \qa{ e^{- V_0(\varphi + \zeta)} F(\varphi + \zeta)} } \nnb
& =  \Eg_{C_\infty - C_t} \qa{ e^{- V_t(\varphi)}  \; e^{+ V_t(\varphi)} \Eg_{C_t} \qa{ e^{- V_0(\varphi + \zeta)} F(\varphi + \zeta)} } \nnb
& \propto  \E_{\nu_t} \qa{ \PP_{0,t} F } ,
\end{align}
where $\zeta$ is integrated with respect to $\Eg_{C_t}$ and $\varphi$ with respect to $\Eg_{C_\infty - C_t}$.
We used the definitions \eqref{e:P-def-bis} and \eqref{e:nu-def-bis} in the last line,
and recall that $\propto$ is an equality up to a normalising factor so that $\nu_t$ is a probability measure.
This completes the proof of the identity \eqref{e:polchinski-semigroup level 0}.
\end{proof}

Using the definition \eqref{e:V-def} of $V_t$,  
the action of the Polchinski semigroup \eqref{e:P-def-bis} can be interpreted as a conditional expectation  
with respect to $\varphi$:
\begin{equation} 
\label{e:fluctuationmeasure-def} 
  \PP_{0,t}F(\varphi) = 
\frac{   \Eg_{C_t}  \qa{e^{-V_0(\varphi+\zeta)} F(\varphi+\zeta)}}{   \Eg_{C_t}  \qa{e^{-V_0(\varphi+\zeta)}}}
 =:  \E_{\mu_t^\varphi} [F(\zeta)].
\end{equation}
This defines a probability measure $\mu_t^\varphi$  called the \emph{fluctuation measure}. 
Assuming $C_t$ is invertible and  changing variables from $\varphi+\zeta$ to $\zeta$,
the fluctuation measure can be written equivalently as
\begin{equation} 
\label{e:fluctuationmeasure-alt}
  \mu_t^\varphi(d\zeta) = e^{+V_t(\varphi)}
  e^{-\frac12(\varphi-\zeta,C_t^{-1}(\varphi-\zeta)) - V_0(\zeta)} \, d\zeta
  \propto e^{-\frac12( \zeta,C_t^{-1} \zeta) + (\zeta,C_t^{-1} \varphi) - V_0(\zeta)} \, d\zeta .
\end{equation}
Besides the addition of an external field $C_t^{-1} \varphi$, the structure of this new measure is similar to the one of the original measure  $\nu_0$ introduced in \eqref{e:nu0-Cinfty}, but the  covariance of the Gaussian integration is now  $C_t$. 
By construction $C_t \leq C_\infty$, so that the Hamiltonian of the conditional measure \eqref{e:fluctuationmeasure-def} is more convex and will  hopefully be easier to handle. 
The fluctuation measure is central in the \emph{stochastic localisation} framework which will be presented in 
Section~\ref{sec:stochloc}.

For all bounded function $F: X \to \R$ and all $t>0$, the identity \eqref{e:polchinski-semigroup level 0} reads
\begin{equation}
  \E_{\nu_0}[F] = \E_{\nu_t}[\PP_{0,t}F(\varphi)] = \E_{\nu_t}[\E_{\mu_t^\varphi}[F(\zeta)]],
\end{equation}
where $\varphi$ denotes the variable of $\nu_t$ and $\zeta$ the variable of $\mu_t^\varphi$.
This  is therefore an instance of the measure decomposition  \eqref{eq: decomposition nu1 nu2} by successive conditionings.
The splitting of the covariance $C_\infty = C_\infty - C_t + C_t$ will be chosen so that the field  $\zeta$ encodes the local interactions, which correspond to the fast scales of the dynamics, and $\varphi$ the long range part of the interaction, associated with the slow dynamical modes.
Integrating out the short scales boils down to considering a new test function 
$\PP_{0,t} F(\varphi)$ and a measure $\nu_t$ \eqref{e:nu-def-bis} which is expected to have better properties than the original measure $\nu_0$. 
This is illustrated in a one-dimensional case in Example~\ref{sec:1d_example}.

Models from statistical mechanics often involve a multiscale structure when approaching the phase transition. For this reason, it is not enough to split the measure into two parts as in \eqref{e:polchinski-semigroup level 0}. The renormalisation procedure is based on a recursive procedure with successive integrations of the fast scales in order to simplify the measure step by step. 
As an example, let us describe  a two step procedure: for $s < t$, splitting the covariance into 
$C_s, C_t- C_s, C_\infty - C_t$, can be achieved by applying twice \eqref{e:polchinski-semigroup level 0}
\begin{equation} 
\label{eq: semigroup structure}
\E_{\nu_0} \qa { F } =   \E_{\nu_s} \qa{ \PP_{0,s} F }
= \E_{\nu_t} \qa{ \PP_{s,t} \Big( \PP_{0,s} F \Big)} = \E_{\nu_t} \qa{ \PP_{0,t} F } .
\end{equation}
Thus $\PP_{s,t}$ inherits a semigroup property from the nested integrations.
For infinitesimal renormalisation steps, we are going to show in Proposition \ref{prop:polchinski2} that 
the Polchinski semigroup is in fact a Markov semigroup with a structure reminiscent of the Langevin semigroup
\eqref{eq: Kolmogorov backward equation}.
To implement this renormalisation procedure, one has also to control the renormalised measure $\nu_t$. 
For infinitesimal renormalisation steps, its potential $V_t$  evolves according to the following Hamilton--Jacobi--Bellman equation, known as \emph{Polchinski equation}.
\begin{proposition} 
\label{prop:polchinski1}
  Let $(C_t)$ be as above, and let $V_0 \in C^2$.
  Then for every $t$ such that $C_t$ is differentiable
  the renormalised potential $V_t$ defined in \eqref{e:V-def} satisfies the \emph{Polchinski equation}
  \begin{equation}
    \label{e:polchinski-bis}
    \ddp{}{t} V_t = \frac12 \Delta_{\dot C_t} V_t - \frac12 (\nabla V_t)_{\dot C_t}^2
  \end{equation}
where $\Delta_{\dot C_t}$ was defined in \eqref{e:DeltaC} and 
the scalar product in \eqref{eq: scalar product C}.
\end{proposition}

\begin{proof}
Let $Z_t(\varphi)= \Eg_{C_t}[e^{-V_0(\varphi+\zeta)}]$.
By Proposition~\ref{prop:heat},
it follows that the Gaussian convolution
acts as the heat semigroup with time-dependent generator $\frac12 \Delta_{\dot C_t}$, i.e.,
if $Z_0$ is $C^2$ in $\varphi$ so is $Z_t$ for any $t>0$,
that $Z_t(\varphi)>0$ for any $t$ and $\varphi$, and that for any $t>0$
such that $C_t$ is differentiable, 
\begin{equation}
  \ddp{}{t} Z_t = \frac12 \Delta_{\dot C_t} Z_t, \quad Z_0=e^{-V_0}.
\end{equation}
Since $Z_t(\varphi)>0$ for all $\varphi$, its logarithm $V_t = -\log Z_t$ is well-defined and satisfies the Polchinski equation
\begin{equation}
  \ddp{}{t} V_t = - \frac{\ddp{}{t} Z_t}{Z_t}
  = -\frac{\Delta_{\dot C_t} Z_t}{2Z_t}
  = -\frac12 e^{V_t} \Delta_{\dot C_t} e^{-V_t}
  = \frac12 \Delta_{\dot C_t} V_t - \frac12 (\nabla V_t)_{\dot C_t}^2
  .\qedhere
\end{equation}
\end{proof}

The semigroup structure is analysed in the following proposition. As mentioned above, we assume $V_0$ to be bounded below to avoid technical problems.
\begin{proposition}
\label{prop:polchinski2}
  The operators $(\PP_{s,t})_{s\leq t}$ form a time-dependent Markov semigroup with generators $(\LL_t)$,
  in the sense that
 \begin{equation}
 s\leq r \leq t: \qquad
 \PP_{t,t} = \id
 \quad \text{and} \quad \PP_{r,t}\PP_{s,r}=\PP_{s,t},
 \end{equation}
and $\PP_{s,t}F \geq 0$ if $F \geq 0$ with  $\PP_{s,t}1=1$.
  
  Furthermore for all $t$ at which $C_t$ is differentiable
  (respectively $s$ at which $C_s$ is differentiable),
  \begin{equation} 
  \label{e:polchinski-generator}
s \leq t: \qquad     \ddp{}{t} \PP_{s,t}F = \LL_t  \PP_{s,t} F,
    \qquad
    -\ddp{}{s} \PP_{s,t}F = \PP_{s,t} \LL_s F,
 \end{equation}
  for all smooth functions $F$, where $\LL_t$ acts on a smooth function $F$ by
  \begin{equation}  
    \label{e:polchinski-L}
    \LL_tF = \frac12 \Delta_{\dot C_t} F - (\nabla V_t, \nabla F)_{\dot C_t}.
  \end{equation}
  The measures $\nu_t$ evolve dual to $(\PP_{s,t})$ in the sense that
  \begin{equation} 
  \label{e:polchinski-semigroup}
    \E_{\nu_t} \qa{ \PP_{s,t} F }= \E_{\nu_s} \qa { F }\quad (s \leq t),
        \qquad
    -\ddp{}{t} \E_{\nu_t} \qa{F} = \E_{\nu_t} \qa{ \LL_t F}.
  \end{equation}
\end{proposition}

The operator $\LL_t$ in \eqref{e:polchinski-L} is obtained by linearising the Polchinski equation \eqref{e:polchinski-bis} and has a structure similar to the generator $\Delta^H$
defined in \eqref{eq: Langevin generator}. 
According to \eqref{e:polchinski-semigroup}, the renormalised measure evolves according to
$- \partial_t \nu_t =  \LL_t^* \nu_t$ where $\LL_t^*$ is the formal adjoint of $\LL_t$ (with respect to the Lebesgue measure).
This is another way to rephrase the Polchinski equation \eqref{e:polchinski-bis}.

\begin{proof} 
By assumption, $V_0$ is bounded below. 
The weak convergence of the Gaussian measure $\Pg_{C_t-C_s}$ to the Dirac measure at $0$ when $t\downarrow s$ thus implies $\PP_{t,t}=\id$. 
The semi-group property, i.e. $\PP_{r,t}\PP_{s,r} = \PP_{s,t}$ for any $s\leq r\leq t$,
then follows from \eqref{eq: semigroup structure}. 
The definition~\eqref{e:P-def-bis} also implies continuity since $\|\PP_{s,t}F\|_\infty\leq \|F\|_\infty$ for each bounded $F$.
Equation~\eqref{e:P-def-bis} also implies that $\PP_{s,t}F\geq 0$ if $F\geq 0$. 

To verify that the generator $\LL_t$ of the Polchinski semigroup is given by \eqref{e:polchinski-L},
set  for $s<t$:
\begin{equation}
F_{s,t}(\varphi) = \PP_{s,t}F(\varphi) = e^{V_t(\varphi)} \Eg_{C_t-C_s}[e^{-V_s(\varphi+\zeta)} F(\varphi+\zeta)].
\end{equation}
Computing the time derivatives using Propositions~\ref{prop:heat} and~\ref{prop:polchinski1}, this leads to
\begin{align}
  \ddp{}{t} F_{s,t}
  &= (\ddp{}{t} V_t) F_{s,t} + e^{V_t} \frac12 \Delta_{\dot C_t} \Eg_{C_t-C_s}[e^{-V_s(\cdot+\zeta)} F(\cdot+\zeta)]
    \nnb
  &= (\ddp{}{t} V_t) F_{s,t} + e^{V_t} \frac12 \Delta_{\dot C_t} (e^{-V_t} F_{s,t})
    \nnb
  &= (\ddp{}{t} V_t) F_{s,t} - (\frac12 \Delta_{\dot C_t} V_t) F_{s,t} + \frac12 (\nabla V_t)_{\dot C_t}^2 F_{s,t} + \frac12 \Delta_{\dot C_t} F_{s,t} - (\nabla V_t, \nabla F_{s,t})_{\dot C_t}
    \nnb
  &= \frac12 \Delta_{\dot C_t} F_{s,t} - (\nabla V_t, \nabla F_{s,t})_{\dot C_t}
    \nnb
  &= \LL_t F_{s,t}
    ,
\end{align}
which is the first equality in \eqref{e:polchinski-generator}.
The second equality in \eqref{e:polchinski-generator} follows analogously.

The first equality in \eqref{e:polchinski-semigroup} holds as in Proposition \ref{prop: mesure decomposition 1 pas}.  
The second identity follows by taking  derivatives in $s$ in the first identity and then using \eqref{e:polchinski-generator} so that 
\begin{equation}
\ddp{}{s} \E_{\nu_s} \qa{ F }
=
\ddp{}{s} \E_{\nu_t} \qa{ \PP_{s,t} F }  =  \E_{\nu_t} \qa{ \ddp{}{s} \PP_{s,t} F }  = - \E_{\nu_t} \qa{   \PP_{s,t}\LL_s F }.
\end{equation}
For $t=s$ then $\PP_{s,s}\LL_s F = \LL_sF$ and the second identity  in \eqref{e:polchinski-semigroup} is recovered.
\end{proof}

\subsection{Log-Sobolev inequality via a multiscale Bakry--\'Emery method}
\label{sec: Log-Sobolev inequality via a multiscale Bakry--Emery method}

In this section, the Polchinski renormalisation  is used to derive a log-Sobolev inequality under a criterion on the renormalised potentials, which can be interpreted as a multiscale 
condition generalising the strict convexity of the Hamiltonian in the Bakry--\'Emery criterion (Theorem \ref{thm:BE}).

We impose the following technical \emph{continuity assumption} analogous to \eqref{e:ergodicity-continuous2}:
for all bounded smooth functions $F: X \to \R$ and $g:\R \to \R$,
\begin{equation} 
\label{e:continuity}
\lim_{t \to \infty}  \E_{\nu_t} \qa{g(\PP_{0,t}F)} = g \big(  \E_{\nu_0} \qa{F} \big).
\end{equation}
This can be easily checked in all examples of practical interest.

\begin{theorem} 
\label{thm:LSI-mon}
Consider a measure $\nu_0$ of the form \eqref{e:nu0-Cinfty} associated with a covariance decomposition  $\dot C_t$ differentiable for all $t$ (see Section~\ref{sec:LSI-setup}), and assume also \eqref{e:continuity}.
  
Suppose there are real numbers $\dot\lambda_t$ (allowed to be negative) such that
\begin{equation} 
  \label{e:assCt-mon}
  \forall \varphi \in X, t> 0:\qquad
  \dot C_t \He V_t(\varphi) \dot C_t - \frac{1}{2} \ddot C_t \geq \dot\lambda_t \dot C_t,
\end{equation}
and define
\begin{equation} 
  \label{e:gammadef-mon}
  \lambda_t = \int_0^t \dot \lambda_s \, ds,
  \qquad \frac{1}{\gamma} = \int_0^\infty e^{-2\lambda_t} \, dt .
\end{equation}
Then $\nu_0$ satisfies the log-Sobolev inequality
\begin{equation} 
\label{e:LSI-mon}
\ent_{\nu_0} \qa{ F }   \leq \frac{2}{\gamma} \E_{\nu_0}\qa{(\nabla \sqrt{F})^2_{\dot C_0}}.
\end{equation}
\end{theorem}

Contrary to the Bakry--\'Emery criterion (Theorem~\ref{thm:BE}), the initial potential $V_0$ is not required to be convex. The relevant parameter is an integrated estimate \eqref{e:gammadef-mon} on the Hessian of the renormalised potentials $V_t$. Thus if one can prove that the renormalisation flow improves the non-convexity of the original potential so that the integral in \eqref{e:gammadef-mon} is finite, then the log-Sobolev inequality holds.
In the case of convex potential $V_0$, the convexity is preserved by the Polchinski equation (see Proposition \ref{prop:convex}) and the Bakry--\'Emery criterion can be recovered. In general, the analysis of the renormalised potential $V_t$
is model dependent.

The covariances $\dot C_t$ play the role of an inverse metric on $X$. In our examples of
interest, this metric becomes increasingly coarse approximately implementing the ``block spin renormalisation picture''.
See Section~\ref{sec:geometry} for further discussion of this.

\begin{remark} \label{rk:LSI-mon1}
With the same proof, the log-Sobolev inequality \eqref{e:LSI-mon}
can be generalised to one for each of the renormalised measures $\nu_s$:
\begin{equation}
    \ent_{\nu_s}(F)
    \leq \frac{2}{\gamma_s} \E_{\nu_s}\qa{ (\nabla \sqrt{F})^2_{\dot C_s} },
    \qquad
    \frac{1}{\gamma_s} = \int_s^\infty e^{-2(\lambda_u-\lambda_s)} \, du.
  \end{equation}
\end{remark}

\begin{remark} 
\label{rk:LSI-mon2}
  The condition \eqref{e:assCt-mon}--\eqref{e:LSI-mon} is invariant under reparametrisation in $t$.
  For example, if $a: [0,+\infty] \to [0,+\infty]$ is a smooth reparametrisation,
  set
  \begin{equation}
    C^a_t = C_{a(t)}, \qquad V^a_t = V_{a(t)}.
  \end{equation}
  Then $\dot C^a_t = \dot a(t) \dot C_{a(t)}$ and $\ddot C^a_t = \ddot a(t) \dot C_{a(t)} + \dot a(t)^2 \ddot C_{a(t)}$ and therefore
  \eqref{e:assCt-mon} is equivalent to
  \begin{equation}
    \dot C_t^a \He V_t^a \dot C_t^a - \frac12 \ddot C_t^a = \dot a(t)^2 \qB{ \dot C_{a(t)} \He V_{a(t)} \dot C_{a(t)} - \frac12 \ddot C_{a(t)}} - \frac12 \ddot a(t) \dot C_{a(t)} \geq \dot \lambda^a_t \dot C_t^a
  \end{equation}
  with
  \begin{equation}
    \dot\lambda^a_t
    = \dot a(t)\dot\lambda_{a(t)} - \frac12 \frac{\ddot a(t)}{\dot a(t)}
    = \dot a(t)\dot\lambda_{a(t)} - \frac12 \ddp{}{t} \log \dot a(t)
    .
  \end{equation}
  Thus \eqref{e:gammadef-mon} becomes
  \begin{equation}
    \lambda^a_t 
    = \int_0^t \dot \lambda^a_s \, ds
    = \int_0^t \dot a(s) \dot\lambda_{a(s)} \, ds  - \frac12 \log \frac{\dot a(t)}{\dot a(0)},
  \end{equation}
  and hence
  \begin{align}
    \dot C_0^a  \int_0^\infty e^{-2\lambda^a_t} \, dt
    &= \frac{\dot C_0^a }{\dot a(0)} \int_0^\infty e^{-2\int_0^t \dot\lambda_{a(s)} \dot a(s) \, ds}\,  \dot a(t) \, dt
    \nnb
    &= \dot C_0 \int_0^\infty e^{-2\int_0^{u} \dot\lambda_{u} \, du}\, du    = \dot C_0 \int_0^\infty e^{-2\lambda_u}\, du.
  \end{align}

  Analogously, one can parametrise by $[0,T]$ instead of $[0,+\infty]$,
  i.e., use a covariance decomposition $C=\int_0^T \dot C_t\, dt$, and then
  obtain  the same conclusion with $T$ instead of $\infty$ in  the estimates.
\end{remark}
\begin{remark} \label{rk:LSI-mon3}
  For a covariance decomposition such that $\dot C_t$ is not differentiable for all $t$, an alternative criterion that does not involve $\ddot C_t$ can be formulated, see~\cite[Theorem~2.6]{MR4303014}.
\end{remark}

\begin{proof}[Proof of Theorem~\ref{thm:LSI-mon}]
The proof follows the strategy of the Bakry--\'Emery theorem (Theorem~\ref{thm:BE}), replacing the Langevin dynamics by the Polchinski flow.
We consider a curve of probability measures $(\nu_t)_{t\geq 0}$
and a corresponding dual time-dependent Markov semigroup $(\PP_{s,t})$
with generators $(\LL_t)$ as in Proposition~\ref{prop:polchinski2}.

For $F: X\to \R$  a function  with values in a compact subset $I$ of $(0,\infty)$, we write
$F_t = \PP_{0,t} F \in I$.  Since the function $\Phi$ is smooth on $I$, it 
can be extended to a bounded smooth function on $\R$ and
we deduce from \eqref{e:continuity} that
\begin{equation} 
\label{eq: consequence continuity}
\lim_{t \to \infty}  \E_{\nu_t} \qa{ \Phi (\PP_{0,t}F)} = \Phi \big(  \E_{\nu_0} \qa{F} \big).
\end{equation}
Thus 
\begin{equation}
\label{eq: time interpolation entropy 2}
\ent_{\nu_0} (F)= \E_{\nu_{0}} [\Phi(F)]- \Phi(\E_{\nu_0} [F])
 = -\int_0^\infty dt \; \ddp{}{t} \E_{\nu_t} [\Phi(F_t)].
\end{equation}
It remains to prove the counterpart of the de Bruijn Formula \eqref{e:dBidentity}.
Denoting  $\dot F_t=\ddp{}{t} F_t$,
using first \eqref{e:polchinski-semigroup} and then \eqref{e:polchinski-L},
it follows that
\begin{align} \label{e:dEnt}
  - \ddp{}{t} \E_{\nu_t} [\Phi (F_{t})]
  &=
    \E_{\nu_t} 
    \qbb{
    \LL_t (\Phi (F_{t}))
    -
    \Phi '(F_{t})\dot{F}_{t}
    }
  \nnb
  &
    =
    \E_{\nu_t}
    \qbb{
    \Phi' (F_{t})\LL_t F_t
    +
    \Phi''(F_{t}) \frac{1}{2} (\nabla F_{t})_{\dot C_t}^{2}
    -
    \Phi'( F_{t}) \dot{F}_{t}
    }
  \nnb
  &
    =
    \frac12
    \E_{\nu_t}
    \qbb{
    \Phi''(F_t) (\nabla F_{t})_{\dot C_t}^{2}
    }
    =
    2
    \E_{\nu_t}
    \qbb{
    (\nabla \sqrt{F_{t}})_{\dot C_t}^{2}
    }
    .
\end{align}
Integrating this relation using \eqref{eq: time interpolation entropy 2} gives
\begin{equation}
  \label{e:Ent-P}
  \ent_{\nu_0}(F)
  = 2\int_0^\infty \E_{\nu_{t}} \qa{(\nabla \sqrt{\PP_{0,t} F})_{\dot C_t}^2} \, dt
  .
\end{equation}
The above entropy production formula \eqref{e:dEnt} is analogous to the de Bruijn identity  \eqref{e:dBidentity}
and the entropy decomposition to   \eqref{eq: de Bruijn bis}, but an important difference is that
the reference measure $\nu_t$ here changes as well.
In Section~\ref{sec:background}, we used that $\E_{\nu}[\Delta^H F] =0$ for any $F$ in the derivation of the de Bruijn identity,
but more conceptually what we used is that
the measure $\nu$ satisfies
\begin{equation}
  -\ddp{}{t} \E_{\nu}[\cdot] = \E_{\nu}[\Delta^H (\cdot)],
  \label{eq_IBP_BE_case}
\end{equation}
since both sides are $0$ (because the stationary measure $\nu$ does not depend on $t$). 
In the computation above, both $\nu_t$ and $F_t$ vary with $t$, 
but in a dual way, 
and the analogue of~\eqref{eq_IBP_BE_case} is~\eqref{e:polchinski-semigroup level 0}.

It remains to derive the counterpart of \eqref{eq: 2nd derivative entropy} 
and show that 
  \begin{equation}
  \label{eq: exp decay exchanged gradient}
 \forall t \geq 0: \qquad    (\nabla \sqrt{\PP_{0,t}F})^2_{\dot C_t}   
    \leq  e^{-2\lambda_t} \PP_{0,t} \left[ (\nabla \sqrt{F})^2_{\dot C_0} \right].
  \end{equation}
Plugging this relation in \eqref{e:Ent-P} and recalling that 
$\E_{\nu_{t}} \qa{ \PP_{0,t} \left[ (\nabla \sqrt{F})^2_{\dot C_0} \right]} = 
\E_{\nu_0} \qa{  (\nabla \sqrt{F})^2_{\dot C_0}}$, 
the log-Sobolev inequality \eqref{e:LSI-mon} is recovered.

\medskip

We turn now to the proof of \eqref{eq: exp decay exchanged gradient}.
The following lemma is essentially the Bakry--\'Emery argument adapted to the Polchinski flow.
\begin{lemma}
\label{lem:LSI-pf-identity}
  Let $\LL_t$, $\PP_{0,t}$, $\dot C_t$, $V_t$ be as in Section~\ref{sec:LSI-setup}.
  Then the following identity holds
  for any $t$-independent positive definite matrix $Q$:
  \begin{equation} \label{e:LSI-pf-identity}
    (\LL_{t}-\partial_t)(\nabla \sqrt{\PP_{0,t}F})^2_{Q}
    = 2(\nabla \sqrt{\PP_{0,t}F},\He V_t \dot C_t \nabla \sqrt{\PP_{0,t}F})_{Q}
    + \frac14 (\PP_{0,t}F) |\dot C_t^{1/2} (\He \log \PP_{0,t}F) Q^{1/2}|_2^2
    ,
  \end{equation}
  where $|M|_2^2 = \sum_{p,q}|M_{pq}|^2$ denotes the squared Frobenius norm of a matrix $M=(M_{pq})$.
\end{lemma}
The derivation of the lemma is postponed. Applying it  with $Q=\dot C_t$ implies
  \begin{multline} \label{e:LSI-pf-identity-t-mon}
    (\LL_{s}-\partial_s)(\nabla \sqrt{\PP_{0,s}F})^2_{\dot C_s}
    = 2(\nabla \sqrt{\PP_{0,s}F},\He V_s \dot C_s \nabla \sqrt{\PP_{0,s}F})_{\dot C_s}
    - (\nabla \sqrt{\PP_{0,s}F})^2_{\ddot C_s}
    \\
    + \frac14 (\PP_{0,s}F) |\dot C_s^{1/2} (\He \log \PP_{0,s}F) \dot C_s^{1/2}|_2^2
    .
  \end{multline}
  By the assumption \eqref{e:assCt-mon} and since the last term is positive, it follows that
  \begin{equation} \label{e:LSI-pf-D-contract}
    (\LL_{s}-\partial_s)(\nabla \sqrt{\PP_{0,s}F})^2_{\dot C_s}
    \geq 2\dot\lambda_{s} (\nabla \sqrt{\PP_{0,s}F})^2_{\dot C_s}
    .
  \end{equation}
  Equivalently,
  $\psi(s) := e^{-2\lambda_t+2\lambda_s} \PP_{s,t} \left[ (\nabla \sqrt{\PP_{0,s}F})^2_{\dot C_s} \right]$ satisfies $\psi'(s) \leq 0$ for $s<t$.
  This implies $\psi(t) \leq \psi(0)$ so that \eqref{eq: exp decay exchanged gradient} holds. 
\end{proof}

At first sight, the proof of Theorem~\ref{thm:LSI-mon} may seem mysterious, but the idea is simply to iterate the entropy decomposition
\eqref{eq: entropy decomposition} by using the Polchinski flow to decompose the measure into its scales. 
To illustrate this, let us consider a discrete decomposition of the entropy using the Polchinski flow.
Given $\delta>0$ and the sequence $(t_i = i \delta)_{i \geq 0}$, one has 
\begin{align}
\ent_{\nu_0} (F) & = \E_{\nu_{0}} [\Phi(F)]- \Phi(\E_{\nu_0} [F])\nnb
& = \sum_i \E_{\nu_{t_i}} [\Phi(  \PP_{0,t_i} (F))] - \E_{\nu_{t_{i+1}}} [\Phi(  \PP_{0,t_{i+1}} (F))] \nnb
&  = \sum_i \E_{\nu_{t_{i+1}}} [ \PP_{t_i,t_{i+1}} \Phi(  \PP_{0,t_i} (F)) -  \Phi(  \PP_{0,t_{i+1}} (F))]   \nnb
&= \sum_i \E_{\nu_{t_{i+1}}} [ \ent_{\PP_{t_i,t_{i+1}}}  (  \PP_{0,t_i} (F)) ]  .
\end{align}
The measure $\PP_{t_i,t_{i+1}}$ associated with a small increment satisfies a log-Sobolev inequality 
as the associated Gaussian covariance $C_{t_{i+1}} - C_{t_i}$ is tiny for $\delta$ small (so that the Hamiltonian corresponding to the measure $\PP_{t_i,t_{i+1}}$ is extremely convex). 
This suggests that for each   interval $[t_i,t_{i+1}]$, one can reduce to estimating 
$\E_{\nu_{t_{i+1}}} [   ( \nabla \sqrt{\PP_{0,t_i} (F)})^2_{\delta \dot C_{t_i}} ]$ and the delicate issue is then to interchange $\nabla$ and $\PP_{0,t_i}$ (note that a similar step already occurred even in the product case \eqref{eq: interchange product case}).
Such a discrete decomposition was implemented in \cite{MR4061408} to derive a spectral gap for certain models.
The proof of Theorem~\ref{thm:LSI-mon} relies on the limit where $\delta$ tends to 0 which greatly simplifies the argument as the analytic structure of the Polchinski flow kicks in.

\begin{proof}[Proof of Lemma \ref{lem:LSI-pf-identity}]
For a more detailed proof, see \cite[Lemma~2.8]{MR4303014}.
One can first verify the so-called `Bochner formula':
\begin{equation} \label{e:Bochner}
  (\LL_{t}-\partial_t)(\nabla \PP_{0,t}F)^2_{Q}
  = 2(\nabla \PP_{0,t}F, \He V_t \dot C_t \nabla \PP_{0,t}F)_{Q}
  + |\dot C_t^{1/2} \He \PP_{0,t}F Q^{1/2}|_2^2.
\end{equation}
The claim \eqref{e:LSI-pf-identity} then follows: writing $F$ instead of $\PP_{0,t}F$ for short,
dropping other $t$-subscripts, 
\begin{equation} \label{e:Bochner-pf1}
  (\LL_{t}-\partial_t)(\nabla \sqrt{F})^2_{Q}
  = \frac{(\LL_{t}-\partial_t) (\nabla F)^2_Q}{4F}
  - \frac{(\nabla F)^2_Q (\LL_{t}-\partial_t) F}{4F^2}
  - \frac{(\nabla (\nabla F)^2_Q, \nabla F)_{\dot C}}{4F^2}
  + \frac{(\nabla F)^2_Q (\nabla F)_{\dot C}^2}{4F^3}.
\end{equation}
Using $(\LL_t-\partial_t)F=0$ and \eqref{e:Bochner} the right-hand side equals that in \eqref{e:LSI-pf-identity} since
\begin{equation} \label{e:Bochner-pf2}
  F |\dot C^{1/2} \He \log F Q^{1/2}|_2^2
  =
  \frac{|\dot C^{1/2}\He F Q^{1/2}|_2^2}{F}
  - \frac{(\nabla (\nabla F)^2_Q, \nabla F)_{\dot C}}{F^2}
  + \frac{(\nabla F)^2_Q (\nabla F)_{\dot C}^2}{F^3}.
\end{equation}
To see this, observe that the left-hand side is (with summation convention)
\begin{equation}
  F \dot C_{ij} Q_{kl} (\He \log F)_{ik}(\He \log F)_{jl}
  =
  F \dot C_{ij} Q_{kl} (\frac{F_{ik}}{F} - \frac{F_iF_k}{F^2})(\frac{F_{jl}}{F} - \frac{F_jF_l}{F^2}),
\end{equation}
and the right-hand side is
\begin{equation}
  \dot C_{ij} Q_{kl} \qa{ \frac{F_{ik} F_{jl}}{F} - \frac{(F_kF_l)_iF_j}{F^2} + \frac{F_iF_jF_kF_l}{F^3}}.
\end{equation}
So both are indeed equal.
\end{proof}

\subsection{Derivatives of the renormalised potential}
\label{sec: Derivatives of the renormalised potential}

Checking the multiscale assumption in Theorem~\ref{thm:LSI-mon} boils down to controlling the Hessian of the renormalised potential $V_t$. For a well chosen covariance decomposition, the structure of the potential $V_t$ is often expected to improve along the flow of the Polchinski equation \eqref{e:polchinski-bis}. In particular, one may hope that $V_t$ becomes more convex. This is illustrated in the Example~\ref{sec:1d_example} below which considers the case of a single variable.
However, for a given microscopic model the convexification can be extremely difficult to check.
Some examples where it is possible are discussed in Section~\ref{sec: Applications}.

Even though the analysis of the derivatives of $V_t$ is model dependent, we state a few  
general identities for these derivatives which will be used later. 
\begin{lemma} 
\label{lem:dVdH}
Let  $U_t=\nabla V_t$ and $H_t=\He V_t$. Then 
  \begin{equation}
  \label{e:dVdH}
    \partial_t U_t = \LL_tU_t, \qquad 
    \partial_t H_t = \LL_tH_t - H_t \dot C_t H_t.
  \end{equation}
  Moreover, for all $f \in X$ and $t\geq s \geq 0$,
  \begin{align}
    \label{e:PtdV}
    (f,\nabla V_t) &= \PP_{s,t}(f,\nabla V_{s}),\\
    \label{e:PtHessV}
    (f,\He V_tf) &= \PP_{s,t} (f, \He V_{s}f)
                   - \qB{ \PP_{s,t}((f, \nabla V_{s})^2) -(\PP_{s,t}(f, \nabla V_{s}))^2}.
  \end{align}
\end{lemma}

\begin{proof}
\eqref{e:dVdH} follows by differentiating \eqref{e:polchinski-bis}.
  
To recover   \eqref{e:PtdV}, we recall from \eqref{e:V-def} that 
$V_t (\varphi) = - \log \Eg_{C_t - C_s} \qa{e^{-V_s(\varphi+\zeta)}}$.
Identity \eqref{e:PtdV} follows by differentiating and then identifying $\PP_{s,t}$ by \eqref{e:P-def-bis}
\begin{equation} 
\nabla V_t (\varphi) = \frac{\Eg_{C_t - C_s} \qa{e^{-V_s(\varphi+\zeta)} \nabla V_s(\varphi+\zeta)}}{\Eg_{C_t - C_s} \qa{e^{-V_s(\varphi+\zeta)}}}
= 
\PP_{s,t}(\nabla V_{s}) (\varphi).
\end{equation}
Identity \eqref{e:PtHessV} can then be obtained by taking an additional derivative in the previous expression.
\end{proof}

Alternatively, one can rewrite the derivatives of the renormalised potential in terms of the fluctuation measure 
$\mu_t^\varphi$ introduced in \eqref{e:fluctuationmeasure-alt}.
\begin{lemma}
\label{lem:fluctuationmeasure} 
The first derivative of the renormalised potential is related to an expectation  
\begin{align}
\label{eq: nabla Vt mean}
  \nabla V_t(\varphi)
 = \E_{\mu_t^\varphi}[\nabla V_0(\zeta)] =\E_{\mu_t^\varphi}[C_t^{-1}(\varphi-\zeta)].       
\end{align}
The second derivative is encoded by a variance under the fluctuation measure
\begin{align}
\label{eq: hess Vt variance}
\forall f \in X: \qquad   (f,\He V_t(\varphi)f)
  &= \E_{\mu_t^\varphi}[(f,\He V_0(\zeta) f)] - \var_{\mu_t^\varphi} \big( ( f,\nabla V_0(\zeta))  \big)\\
      &= (f,C_t^{-1}f) - \var_{\mu_t^\varphi}\big( (C_t^{-1}f,\zeta) \big), \nonumber
\end{align}
where the second equalities hold if $C_t$ is invertible.
\end{lemma}

\begin{proof}
The first part of \eqref{eq: nabla Vt mean} follows from the identity $\nabla V_t = \PP_{0,t} (\nabla V_0)$ obtained in \eqref{e:PtdV} and the identification of the fluctuation measure $\mu_t^\varphi$ with the semigroup $\PP_{0,t}$ in \eqref{e:fluctuationmeasure-def}. The second equality is obtained by an integration by parts using the form 
\eqref{e:fluctuationmeasure-alt} of the fluctuation measure.

In the same way the first equality in \eqref{eq: hess Vt variance} is deduced from \eqref{e:PtHessV} by identifying $\PP_{0,t}$ and $\mu_t^\varphi$. The second equality follows by 
differentiating $\varphi \mapsto \E_{\mu_t^\varphi}[C_t^{-1}(\varphi-\zeta)]$ and using \eqref{e:fluctuationmeasure-alt}. 
\end{proof}

\medskip

We consider the case of convex potentials and show that they remain convex along the Polchinski flow.
\begin{proposition} 
\label{prop:convex}
  Assume that $V_0$ is convex. Then $V_t$ is convex for all $t \geq 0$.
\end{proposition}

Thus the standard Bakry--\'Emery criterion can be recovered from Theorem~\ref{thm:LSI-mon}:
if $\He H \geq \lambda\id$ one can choose $A=\lambda \id$ and $\He V_0 \geq 0$
and Theorem~\ref{thm:LSI-mon} guarantees the log-Sobolev inequality with constant
\begin{equation}
  \frac{1}{\gamma} \leq \int_0^\infty e^{-\lambda t} \, dt = \frac{1}{\lambda}.
\end{equation}
This follows from criterion \eqref{e:assCt-mon} applied with $\dot C_t = e^{- \lambda t} \id$ so that 
$\ddot C_t = - \lambda e^{- \lambda t} \id$ and $\dot \lambda_t = \lambda/2$. Note that the decomposition 
$\dot C_t=\id$ on $[0,T]$ with $T=\frac{1}{\lambda}$ (see Remark~\ref{rk:LSI-mon2}) could have been used instead.

\begin{proof}[Proof 1]
  If $V_0$ is convex, then $e^{-V_t(\varphi)}$ is the marginal of the measure $\propto e^{-V_0(\varphi+\zeta)} \, \Pg_{C_t}(d\zeta)$, with density log-concave in $(\zeta,\varphi)$. 
  A theorem of Pr\'ekopa then implies that $V_t$ is convex. 
  It is also possible to directly compute the Hessian:
  the Brascamp--Lieb inequality \cite[Theorem 4.1]{BraLieAppl} states that if a probability measure $\propto e^{-H}$ has strictly convex potential $H$ then
  \begin{equation} \label{e:BL}
     \var(F) \leq \E[(\nabla F, (\He H)^{-1} \nabla F)].
  \end{equation}
  Thus by the first identity in \eqref{eq: hess Vt variance} and then applying the Brascamp--Lieb inequality 
  to estimate the variance, we get
  \begin{align}
    \He V_t(\varphi)
    &= \E_{\mu_t^\varphi}[\He V_0(\zeta)] - \var_{\mu_t^\varphi}(\nabla V_0(\zeta))
      \nnb
    &\geq \E_{\mu_t^\varphi}
      \qB{\He V_0(\zeta)      -
      \He V_0(\zeta) (C_t^{-1}+ \He V_0(\zeta))^{-1} \He V_0(\zeta))}
      \nnb
    &= \E_{\mu_t^\varphi}\qB{\He V_0(\zeta)(C_t^{-1}+ \He V_0(\zeta))^{-1}C_t^{-1}}
      \nnb
    &= \E_{\mu_t^\varphi}\qB{C_t^{-1/2}C_t^{1/2}\He V_0(\zeta)C_t^{1/2}(\id+ C_t^{1/2}\He V_0(\zeta)C_t^{1/2})^{-1}C_t^{-1/2}}.
  \end{align}
  Therefore, with $\hat H_t = C_t^{1/2}\He V_0 C_t^{1/2} \geq 0$,
  \begin{equation}
    C_t^{1/2}    \He V_t(\varphi) C_t^{1/2}
    \geq \E_{\mu_t^\varphi} \qa{\frac{\hat H_t(\zeta)}{\id+ \hat H_t(\zeta)}} \geq 0.
    \qedhere
  \end{equation}
\end{proof}

\begin{proof}[Proof 2 from {\cite[Theorem~9.1]{MR664497}, \cite[Theorem~3.3]{zbMATH05081815}}]
  This alternative approach puts the emphasis on the PDE structure associated with the renormalised potential by application of the maximum princple.
  We give the gist of the proof and refer to \cite[Theorem~3.3, page 129]{zbMATH05081815} for a complete argument.
  Let $H_t = \He V_t$ with $H_0 > 0$, and recall \eqref{e:dVdH}:
  \begin{equation}
    \ddp{H_t}{t} = \LL_t H_t- H_t\dot C_t H_t.
  \end{equation}
  Now assume there is a first time $t_0>0$ and $\varphi_0\in X$ such that $H_{t_0}(\varphi_0)$ has a $0$ eigenvalue with eigenvector $v_0$, i.e.,
  $H_{t_0}(\varphi_0)v_0 = 0$. Define $f_t(\varphi) = (v_0,H_t(\varphi)v_0)$. 
  Therefore
  \begin{equation}
    \ddp{f_{t_0}(\varphi_{0})}{t} = \LL_{t_0} f_{t_0}(\varphi_0) -(v_0,H_{t_0}(\varphi_0)\dot C_{t_0}H_{t_0}(\varphi_0)v_0)
    \geq 0,
  \end{equation}
  where we used that  $f_{t_0}(\varphi)$ is minimum at $\varphi_0$ so that 
  $\LL_{t_0} f_{t_0} (\varphi) = \frac12 \Delta_{\dot C_{t_0}} f_{t_0} (\varphi) \geq 0$  (by the maximum principle) and that by construction
  $(v_0,H_{t_0}(\varphi_0)\dot C_{t_0}H_{t_0}(\varphi_0)v_0) =0$. 
  This shows that $f_t(\varphi_{0})$ cannot cross $0$ after $t_0$.
  A more careful argument involves regularisation, see \cite[Theorem~3.3]{zbMATH05081815}.
\end{proof}

We end this section with a rescaling property of the Polchinski equation.
  
\begin{example} \label{ex:V-F}
  Similarly to \eqref{e:Vinfty}, write the renormalised potential as
  \begin{equation} \label{e:V-F}
    V_t(\varphi) = \frac12 (\varphi, C_t^{-1} \varphi) + F_t(C_t^{-1} \varphi) 
    , 
  \end{equation}
  where $F_t(h) = V_t(0)-\log\Eg_{C_t}[e^{-V_0(\zeta)+(h,\zeta)}]$ is the normalised log partition function of the fluctuation measure at external field $h$.
  Then the Polchinski equation for $V$ is equivalent to  a different Polchinski equation for $F$:
  \begin{equation}
    \ddp{}{t} F_t
    =  \frac12 \Delta_{\dot \Sigma_t} F_t - \frac12 (\nabla F_t)_{\dot \Sigma_t}^2 + \tr(C_t^{-1}\dot C_t),
    \qquad \text{where } \dot\Sigma_t = C_t^{-1}\dot C_t C_t^{-1}.
  \end{equation}
  Note that $\tr(C_t^{-1}\dot C_t)$ is only a constant.
  Indeed,
  $F_t(h) =  V_t(C_th)-\frac12 (h,C_th)$ and thus
  \begin{equation}
    \nabla F_t =  C_t \nabla V_t-C_t h,\qquad                 
    \Delta_{\dot \Sigma_t} F_t = \Delta_{\dot C_t} V_t - \tr(C_t^{-1}\dot C_t),
  \end{equation}
  and
  \begin{align}
    \ddp{}{t} F_t &= \frac12 \Delta_{\dot C_t}V_t - \frac12 (\nabla V_t)_{\dot C_t}^2 +(\nabla V_t, \dot C_t h) -\frac12 (h, \dot C_t h)
                    \nnb
    &=
      \frac12 \Delta_{\dot C_t}V_t -\frac12 (h-\nabla V_t)_{\dot C_t}^2 
      \nnb
      &= \frac12 \Delta_{\dot \Sigma_t}F_t -\frac12 (\nabla F_t)_{\dot \Sigma_t}^2 + \tr(C_t^{-1}\dot C_t)
        .
   \end{align}
\end{example}

\subsection{Example: Convexification along the Polchinski flow for one variable}
\label{sec:1d_example}
The aim of this section is to illustrate the claims that the renormalised measure  becomes progressively  simpler and convex along the Polchinski flow using a simple one variable example. 
Let $H:\R\to\R$ be a $C^2$ potential that is strictly convex outside of a segment: $\inf_{|x|\geq M} H''(x)\geq c>1$ for some $c,M>0$,
but assume that $\inf_{\R} H''<0$, 
and consider the measure:
\begin{equation}
\nu_0(d\varphi)
\propto e^{-H(\varphi)} \, d\varphi
\propto
\exp\Big[-\frac{\varphi^2}{2} - V_0(\varphi)\Big]\, d\varphi
,\qquad 
V_0(\varphi):= H(\varphi) - \frac{\varphi^2}{2}
. 
\label{eq_def_nu_0_1d}
\end{equation}
In~\eqref{eq_def_nu_0_1d}, the Gaussian part $\frac12 \varphi^2$ is singled out to define the Polchinski flow, but this is just a convention up to redefining $V_0$.

By assumption, $\nu_0$ is not log-concave, and 
the Bakry--\'Emery criterion (Theorem~\ref{thm:BE}) does not apply. 
Let us stress that there are many ways to obtain a log-Sobolev inequality for the above measure. 
Our goal, however, is to  exemplify that, using the Polchinski flow, how one can still use a convexity-based argument, 
the multiscale Bakry--\'Emery criterion of Theorem~\ref{thm:LSI-mon}, 
by relying on the convexity of the renormalised measures $\nu_t$ that will be more log-concave than $\nu_0$.

The Polchinski flow is defined in terms of a covariance decomposition, 
which is supposed to decompose the Gaussian part of $\nu_0$ into contributions from different scales. 
In the statistical mechanics examples discussed in Sections~\ref{sec: Difficulties arising from statistical physics perspective}--\ref{sec: Difficulties arising from continuum perspective},
the notion of scale was linked with the geometry of the underlying lattice (e.g., small scales corresponding to information pertaining to spins at small lattice distance). In the single variable case, 
there is no geometry, thus the Gaussian part does not have any structure. 
The only meaningful decomposition, 
written here on $[0,1]$ instead of $[0,\infty)$ for convenience, is therefore:
\begin{equation}
\forall t\in[0,1]: \qquad 
C_t 
:= 
t \id,
\qquad 
\dot C_t 
= 
\id.
\end{equation}
The corresponding renormalised potential reads:
\begin{equation}
e^{-V_t(\varphi)} 
=
\frac{1}{\sqrt{2\pi t}}\int_{\R} \exp\Big[-\frac{\zeta^2}{2t} -V_0(\zeta+\varphi)\Big]\, d\zeta,
\label{eq_V_t_1d_example}
\end{equation}
and the renormalised measure $\nu_t$ defined in \eqref{e:nu-def-bis} and fluctuation measure $\mu_t^\varphi$ 
defined in \eqref{e:fluctuationmeasure-alt} are respectively given by:
\begin{equation}
\nu_t(d\varphi)
\propto \exp\Big[-\frac{\varphi^2}{2(1-t)} - V_t(\varphi)\Big]\, d\varphi
,
\quad
\mu^\varphi_t(d\zeta)
\propto 
\exp\Big[-\frac{\zeta^2}{2t} + \frac{\zeta\varphi}{t}- V_0(\zeta )\Big]\, d\zeta
.
\label{eq_nu_t_mu_t_1d_example}
\end{equation}
Note that in terms of the original Hamiltonian $H(\zeta) = V_0(\zeta) + \frac{\zeta^2}{2}$, the fluctuation measure $\mu_t^\varphi$ is more convex than the initial one:
\begin{equation}
\label{eq: stricte convexite fluct mesure}
\mu^\varphi_t(d\zeta)
\propto 
\exp\Big[-\frac{\zeta^2}{2} \big( \frac{1}{t} - 1 \big) + \frac{\zeta\varphi}{t} - H(\zeta )\Big]\, d\zeta
.
\end{equation}

In other words, $e^{-V_t}$ is the convolution of $e^{-V_0}$ with the heat kernel on $\R$ at time $t$, 
and the Polchinski equation becomes the following well-known Hamilton--Jacobi--Bellman equation:
\begin{equation}
\partial_t V_t
=
\frac{1}{2}\partial^2_\varphi V_t - \frac{1}{2}(\partial_\varphi V_t)^2
.
\label{eq_Polchinski_1d}
\end{equation}
The motivation for the Polchinski decomposition was that one progressively integrates ``small scales'' to recover a measure $\nu_t$ acting on ``large scales'' that one hopes to be better behaved. 
In the present case, 
the only notion of scale refers to the size of fluctuations of the field:\\

$\bullet$ Even though $V_0$ may vary a lot on small values of the field, 
the convolution with the heat kernel at time $t$ means $V_t$ is roughly constant on values much smaller than $\sqrt{t}$.  
Thus small details of $V_0$ are blurred, and $V_t$ varies more slowly than $V_0$. 
This is the translation to the present case of the general idea that ``small scales'' (i.e., values below $\sqrt{t}$) have been removed from $V_t$ and the renormalised potential only sees the ``large scales'' (values above $\sqrt{t}$).\\

$\bullet$ Convolution also improves convexity, in the sense that the renormalised measure $\nu_t$ is more log-concave that $\nu_0$. 
Since $x\mapsto \frac{1}{2(1-t)}\zeta^2$ becomes increasingly convex as $t$ approaches $1$, proving this statement boils down to proving a lower bound on $\partial^2_\varphi V_t$ uniformly on $t\in[0,1]$. 
Semi-convexity estimates for solutions of the Polchinski equation~\eqref{eq_Polchinski_1d} are an active subject of research, 
connected with optimal transport with entropic regularisation, see, e.g.,~\cite{MR4098037,chewi2022entropic,conforti2022weak,conforti2023quantitative} and references therein and Section~\ref{sec:Foellmer} below. 
Informally, the convexity of $V_t$ is given by that of $V_0$, plus an ``entropic'' contribution due to the $\frac{1}{2t}\zeta^2$ term. 
In the present simple case, one can directly compute:
\begin{equation}
  \partial^2_\varphi V_t(\varphi)
  =
  \frac{1}{t}-\frac{1}{t^2}\var_{\mu^\varphi_t}(\zeta) .
\end{equation}
This is an instance of the formula of Lemma~\ref{lem:fluctuationmeasure} valid for a general covariance decomposition. 
It is an example of a general feature of the multiscale Bakry-\'Emery criterion:  
the log-Sobolev constant, which  is not a priori related to spectral properties of the model,
can be estimated by lower bounds on $\partial^2_\varphi V_t$ which are related to variance bounds, i.e., spectral information.
	
\begin{exercise}
  \label{exo: Brascamp-Lieb}
  Using the Brascamp-Lieb inequality \eqref{e:BL} for $t<t_0$ with $t_0^{-1} := -\inf_{\R} V_0''$, 
  and the fact that $\mu^\varphi_t$ satisfies a spectral gap inequality with constant $C$ uniformly in $\varphi\in\R$ and $t>0$, we deduce:
  \begin{equation}
	  \partial^2_\varphi V_t
    \geq 
    \dot \lambda_t := 
    {\bf 1}_{[0,t_0/2]}(t)\Big(\frac{-1}{t_0-t}\Big)  
    + {\bf 1}_{[t_0/2,1]}(t)\Big(\frac{1}{t}-\frac{C}{t^2}\Big)    .
    \label{eq_semiconv_1d}
  \end{equation}	
\end{exercise}

The uniform lower bound~\eqref{eq_semiconv_1d} confirms that $\nu_t$ gets more log-concave as $t$ approaches $1$. 
Injecting the bound~\eqref{eq_semiconv_1d} into the multiscale Bakry--\'Emery criterion of Theorem~\ref{thm:LSI-mon} provides a bound
on the log-Sobolev constant $\gamma$ of $\nu_0$ in terms on parameter $t_0$.

Let us reiterate that one could have obtained a bound on the log-Sobolev constant by standard combination of the usual Bakry--\'Emery and Holley--Stroock criteria,
and here just illustrated that the semi-convexity condition of Theorem~\ref{thm:LSI-mon} remains effective in non-convex cases.
Theorem~\ref{thm:LSI-mon} becomes especially useful in situations with a large state space, 
where the combination of the Bakry-\'Emery and Holley-Stroock criteria do not yield dimension-independent bounds on the log-Sobolev constant 
while effective methods to control the semi-convexity may still exist, see Section~\ref{sec: Applications}.

\subsection{Aside: Geometric perspective on the Polchinski flow}
\label{sec:geometry}

There is a structural resemblance of the renormalisation group flow with geometric flows like the Ricci flow.
The matrices $\dot C_t$ take the role of the inverse of a metric $g_t$ (depending on the flow parameter $t$).
We sketch
the interpretation of the $\dot C_t$ as a scale-dependent metric on the space of fields,
and the natural extension of the above construction in the presence of a non-flat metric.

Suppose $(X,g)$ is a Riemannian manifold.
The metric and its components in coordinates are denoted
\begin{equation}
  g = (g_{ij}),
  \qquad
  g^{ij} = (g^{-1})_{ij},
  \qquad |g| = |\det g|,
\end{equation}
the volume form is
\begin{equation}
  m_{g}(d\varphi) = \sqrt{|g|} \, d\varphi,
\end{equation}
and the covariant derivative and Laplace-Beltrami operator are (with summation convention)
\begin{equation}
  \nabla_g = g^{ij} \partial_j,
  \qquad
  {\rm div}_g U = \frac{1}{\sqrt{|g|}} \partial_i(\sqrt{|g|} U^i),
  \qquad \Delta_g F = \frac{1}{\sqrt{|g|}} \partial_i(\sqrt{|g|}g^{ij}\partial_j F).
\end{equation}
In particular,
\begin{equation}
  \He_g f = \nabla_g\nabla_g f = g^{ik}\partial_k (g^{jl} \partial_l f),
\end{equation}
and
\begin{equation}
  (\nabla_g F)_g^2 = g(\nabla_g F, \nabla_g F) = g^{ij} (\partial_i F) (\partial_j F).
\end{equation}
For a $t$-dependent metric $g_t$, the volume form changes according to
\begin{equation}
  \ddp{}{t} dm_{g_t}
  = (\ddp{}{t} \log \sqrt{|g_t|}) \, dm_{g_t}
  = \frac12 \tr(g_t^{-1}\dot g_t) \, dm_{g_t}
  .
\end{equation}
The Ricci curvature tensor associated with the metric $g$ is denoted $\operatorname{Ric}_g$.

\begin{remark}
The notation  $\Delta_{g_t}$ for the Laplace-Beltrami operator
is \emph{different} from our previous notation for
the covariance-dependent Laplacian $\Delta_{\dot C_t}$ from \eqref{e:DeltaC}.
Indeed, the notation differs by an inverse in the index: The Gaussian Laplacian $\Delta_{\dot C_t}$ corresponds to a Laplace-Beltrami operator if $g_t^{-1} = \dot C_t$.
Thus the infinitesimal covariance $\dot C_t$ plays the role of the \emph{inverse} of a metric.
\end{remark}

\begin{example}
  In the above notation, we can reformulate the previous construction as follows: The covariance decomposition is written as
  \begin{equation}
    A^{-1} = \int_0^\infty g_t^{-1} \, dt.
  \end{equation}
  Then the Polchinski equation reads
  \begin{equation}
    \ddp{}{t}V_t = \frac12 \Delta_{g_t} V_t - \frac12 (\nabla_{g_t}V_t)_{g_t}^2,
  \end{equation}
  and the condition  \eqref{e:assCt-mon} for the log-Sobolev inequality becomes:
  \begin{equation}
    \He_{g_t} V_t + \frac12 \dot g_t \geq \dot\lambda_t g_t.
  \end{equation}
\end{example}

\begin{example}
  Suppose that $A$ is a Laplace operator on $\Lambda \subset \Z^d$ and that $\dot C_t = e^{-tA}$ is its heat kernel.
  Thus the metric $g_t = e^{+tA}$ is the inverse heat kernel. This means that
  \begin{equation}
    |f|_{g_t} \leq 1 \qquad \Leftrightarrow \qquad |e^{tA}f| \leq 1,
  \end{equation}
  i.e., $f = e^{-tA} g$ for some $|g|\leq 1$ where $|\cdot|$ denotes the standard Euclidean norm.
  Therefore the unit ball in the metric $g_t$ corresponds
  to elements $f$ that are obtained by smoothing out elements of the standard unit ball by the heat kernel up to time $t$. In this sense, the 
  geometry associated with $g_t$ implements an approximate block spin picture (in which block averaging has been replaced by convolution with
  a heat kernel).
\end{example}

We now consider the natural extension to a non-flat metrics.
The Laplacian in the presence of a potential $H$ and metric $g$ is
\begin{equation}
  \Delta_{g}^H = e^{H} {\rm div}_{g} (e^{-H} \nabla_g F).
\end{equation}
The analogue of Lemma~\ref{lem:LSI-pf-identity} (with  $Q=\dot C_t$) is as follows.

\begin{lemma} \label{lem:LSI-pf-identity-Ric}
\begin{equation}
  (\LL_t-\partial_t)(\nabla_g \sqrt{F})_g^2 = (\dot g + 2 \He_g V_t+ \operatorname{Ric}_{g}) (\nabla_g \sqrt{F}, \nabla_g \sqrt{F}) + \frac14 |\He_g \log F|_g^2.
\end{equation}
\end{lemma}

\begin{proof}
The Bochner formula with a background metric \cite[Theorem C.3.3]{MR3155209} implies:
\begin{equation}
  (\LL_t-\partial_t)(\nabla_g F)_g^2 = (\dot g + 2 \He_g V_t+ \operatorname{Ric}_{g}) (\nabla_g F, \nabla_g F) + |\He_g F|_g^2.
\end{equation}
The Bakry--\'Emery version (with $\sqrt{F}$ instead of $F$) then follows as 
in the proof of Lemma~\ref{lem:LSI-pf-identity}.
\end{proof}

For $t\in [0,T)$ with $T\in (0,+\infty]$, assume that $g_t$ is a given $t$-dependent metric
and that $\Zg_t \, dm_{g_t}$ evolves according to the associated backward heat equation:
\begin{equation}
  \ddp{}{t} \Zg_t \, dm_{g_t} = (-\frac12 \Delta_{g_t} \Zg_t) \, dm_{g_t}.
\end{equation}
The measure $\Zg_t \, dm_{g_t}$ takes the role of the Gaussian measure with covariance $C_\infty-C_t$.
Assume:
\begin{align}
  \ddp{}{t} V_t &= \frac12 \Delta_{g_t} V_t - \frac12 (\nabla_{g_t} V_t)_{g_t}^2
  \\
  \ddp{}{t} F_t &= \frac12 \Delta_{g_t} F_t - (\nabla_{g_t} V_t, \nabla_{g_t} F_t) = \LL_t F_t.
\end{align}
The last equation defines the semigroup $\PP_{s,t}F$ with generators $\LL_t$.
The renormalised measure $\nu_t$ is defined by
\begin{equation}
  \E_{\nu_t}[F] \propto \int F e^{-V_t} \Zg_t \, dm_{g_t}.
\end{equation}
One can check that
\begin{equation}
  \ddp{}{t} \E_{\nu_t} [\PP_{0,t}F] = 0,
\end{equation}
and that the renormalised measure $\nu_t$ again evolves in a dual way to $\PP_{s,t}$: for $t<T$,
\begin{equation}
  \ddp{}{t} \E_{\nu_t}[F] = -\E_{\nu_t} [\LL_t F].
\end{equation}
The analogue of the continuity assumption \eqref{e:continuity} is
\begin{equation} \label{e:continuity-Ric}
  \E_{\nu_t}[g(\PP_{0,t}F)] \to g(\E_{\nu_0}[F]), \qquad (t\to T). 
\end{equation}
Since the evolution of $\Zg_t$ is in general not explicit in the nonflat case, differently from before,
this is now an assumption that seems difficult to verify.
The same proof as that of Theorem~\ref{thm:LSI-mon} using
Lemma~\ref{lem:LSI-pf-identity-Ric} instead of Lemma~\ref{lem:LSI-pf-identity} gives the following condition for the log-Sobolev inequality.

\begin{theorem} \label{thm:LSI-mon-Ric}
  Assume that the continuity assumption \eqref{e:continuity-Ric} holds.
  Suppose there are $\dot\lambda_t$ (allowed to be negative) such that
  \begin{equation} \label{e:assCt-mon-Ric}
    \forall \varphi \in X, t> 0:\qquad
    \He_{g_t(\varphi)} V_t(\varphi) + \frac12 \operatorname{Ric}_{g_t(\varphi)} + \frac{1}{2} \dot g_t(\varphi) \geq \dot\lambda_t g_t(\varphi),
  \end{equation}
  and define
  \begin{equation} \label{e:gammadef-mon-Ric}
    \lambda_t = \int_0^t \dot \lambda_s \, ds,
    \qquad \frac{1}{\gamma} = \int_0^T e^{-2\lambda_t} \, dt.
  \end{equation}
  Then $\nu_0$ satisfies the log-Sobolev inequality
  \begin{equation} \label{e:LSI-mon-Ric}
    \ent_{\nu_0}(F)
    \leq \frac{2}{\gamma} \E_{\nu_0}\qa{ (\nabla \sqrt{F})^2_{g_0}}.
  \end{equation}
\end{theorem}

We currently do not know of any interesting applications of the generalised set-up of Theorem~\ref{thm:LSI-mon-Ric}
over that of Theorem~\ref{thm:LSI-mon},
but it would be very interesting to find some.

Some references with related constructions (though different motivation) in the context of the
Ricci flow appeared in \cite{0211159}
and then \cite{MR2666905,MR2507614}
and more recently \cite{MR3864508,MR3869036,MR4311115,MR3790068}.

\subsection{Aside: Entropic stability estimate}
\label{subsec: Entropic stability estimate}

An approach different from the Bakry--\'Emery method to prove (modified) log-Sobolev inequalities,
using the same Polchinski flow, is the entropic stability estimate which underlies \cite{2203.04163}
and has its origins in the spectral and entropic independence conditions introduced in~\cite{MR4232133} and \cite{2106.04105}.
In \cite{2203.04163}, this method is applied from the stochastic localisation perspective whose
equivalence with the Polchinski flow is discussed in Section~\ref{sec:stochloc}.
In this section, we rephrase the entropic stability strategy of \cite{2203.04163} with the notations of the Polchinski  flow to explain the connection with the Bakry--\'Emery method.

\medskip

Let us first introduce some notation.
For a probability measure $\mu$ on $X$ (with all exponential moments)
and $h \in X$ write $T_h\mu$ for the tilted probability measure:
\begin{equation}
  \frac{dT_h\mu}{d\mu}(\zeta) = \frac{e^{(h,\zeta)}}{\E_\mu[e^{(h,\zeta)}]},
\end{equation}
and $\cov(\mu)$ for the covariance matrix of $\mu$.
The key estimate is a proof of the entropic stability from a covariance assumption.

\begin{lemma}[Entropic stability estimate {\cite[Lemmas~31 and 40]{2203.04163}}] 
\label{lem:entstab}
  Let $\mu$ be a probability measure on $X$, let $\dot\Sigma$ be a positive semi-definite
  matrix, and assume there is $\alpha >0$ such that
  \begin{equation}
 \label{eq: assumption covariance eldan}
 \forall h \in X:\qquad \dot \Sigma \cov(T_h\mu) \dot \Sigma \leq \alpha \dot\Sigma.
  \end{equation}
  Then for all nonnegative $F$ with $\E_\mu[F]=1$,
  \begin{equation}
  \label{eq: entropic stability}
    \frac12 (\E_{\mu}[F \zeta]-\E_{\mu}[\zeta])_{\dot \Sigma}^2 
    \leq
    \alpha \ent_\mu(F).
  \end{equation}
  In {\cite[Definition 29]{2203.04163}}, this inequality is called \emph{$\alpha$-entropic stability} of the measure $\mu$ with respect to the function $\psi(x,y) = \frac12 (x- y)_{\dot \Sigma}^2 $.
\end{lemma}

The proof of Lemma \ref{lem:entstab} is postponed to the end of this section and we first show how 
the entropic stability estimate implies a (possibly modified) log-Sobolev inequality, see  \cite[Proposition~39]{2203.04163}.

\begin{corollary}
\label{cor: entropic stabiliy}
Assume there are $\alpha_t>0$ such that
\begin{equation}
\label{e:ass-entstab}
\forall \varphi\in X:\qquad
\dot C_t \He V_t (\varphi)\dot C_t  - \dot C_t C_t^{-1}\dot C_t \geq -\alpha_t \dot C_t,
\end{equation}
or equivalently, with $\dot\Sigma_t = C_t^{-1} \, \dot C_t \, C_t^{-1}$,
\begin{equation}
\label{e:ass-entstab bis}
  \forall \varphi \in X: \qquad \dot\Sigma_t\cov(\mu_t^{\varphi}) \dot\Sigma_t \leq \alpha_t\dot \Sigma_t.
\end{equation}
Then the measure $\mu_t^\varphi$, i.e., $\PP_{0,t}(\cdot)(\varphi)$,
satisfies the following $ \alpha_t$-entropic stability: for all $\varphi \in X$,
\begin{equation} 
\label{e:statement-entstab}
  2 (\nabla \sqrt{\PP_{0,t}F})_{\dot C_t}^2 (\varphi)
  = 2 (\nabla \sqrt{\E_{\mu_t^\varphi}[F]})_{\dot C_t}^2
  \leq
  \alpha_t \ent_{\mu_t^\varphi}(F)
  =
  \alpha_t \qB{ \PP_{0,t}\Phi(F) (\varphi) -\Phi \big( \PP_{0,t}F (\varphi) \big)}.
\end{equation}
This implies the following entropy contraction: for any $s>0$,
\begin{equation}
\label{eq: LSI intermediaire}
  \ent_{\nu_0}(F) \leq e^{\int_s^\infty\alpha_u\, du} \; \E_{\nu_s}\qB{\ent_{\mu_s^\varphi}(F)}.
\end{equation}
\end{corollary}

The condition \eqref{e:ass-entstab} on $\alpha_t$ is very similar to  the multiscale Bakry--\'Emery condition \eqref{e:assCt-mon} on $\dot \lambda_t$, but
not identical. We recall that \eqref{e:assCt-mon} reads
\begin{equation} 
\label{e:assCt-mon 2}
    \dot C_t \He V_t(\varphi) \dot C_t - \frac{1}{2} \ddot C_t  \geq \dot\lambda_t \dot C_t.
\end{equation}
The log-Sobolev constant and the closely related entropy contraction are estimated by the time integrals of $\alpha_t$ and $\dot \lambda_t$, respectively.
In terms of the two strategies for proving log-Sobolev inequalities discussed in Section~\ref{sec: Decomposition and properties of the entropy},
the Bakry--\'Emery method corresponds more to the first strategy (but also see Remark~\ref{rk_MSBE} for a rephrasing in terms of the second strategy)
while the entropic stability estimate applies the second strategy of entropy contraction.

The multiscale Bakry--\'Emery condition \eqref{e:assCt-mon 2} is less singular for $t$ close to $0$ as $C_t$ typically vanishes as $t \to 0$ while $\dot C_t$ does not.
To see this, consider the simple one-variable case with the covariance function $c_t = \int_0^t \dot c_u \, du \in \R$ (with $\dot c_u >0$) and assume that $\He V_t(\varphi) \geq 0$. Then the optimal choices are $-2\dot \lambda_t = \ddot c_t  \, \dot c_t^{-1}$ and $\alpha_t = \dot c_t  \, c_t^{-1}$ so that
\begin{equation}  
\label{eq: divergence 0}
-2 \int_s^t \dot \lambda_u \, du = \log \dot c_t - \log \dot c_s
\quad \text{and} \quad 
\int_s^t \alpha_u  \, du = \log c_t - \log c_s.
\end{equation}
In particular, one cannot take the limit $s \to 0$ in \eqref{eq: LSI intermediaire} in which the measure $\mu_s^\varphi$ degenerates, and
the $\ent_{\mu_s^\varphi}(F)$ term in \eqref{eq: LSI intermediaire} must be treated differently for $s$ small
to recover a log-Sobolev inequality.
This last step is called \emph{annealing via a localization scheme} in \cite{2203.04163}.
For example, one can use that since covariance $C_s$ vanishes as $s\to 0$ the measure $\mu_{s}^\varphi$ 
becomes uniformly log-concave so that by the standard Bakry--\'Emery criterion (Theorem~\ref{thm:BE}), the measure 
satisfies a log-Sobolev inequality which can then be plugged into \eqref{eq: LSI intermediaire} to complete the log-Sobolev inequality for $\nu_0$. 

\medskip

As a last remark before proving Corollary~\ref{cor: entropic stabiliy},
we emphasise that \eqref{e:statement-entstab} is a contraction estimate for the expected entropy of the fluctuation measure $\mu_t^\varphi$.
More precisely, as in the proof of the multiscale Bakry--\'Emery criterion in \eqref{e:dEnt},
the expectation of the left-hand side of \eqref{e:statement-entstab} is
\begin{equation}
  \ddp{}{t}\E_{\nu_t}\big[\ent_{\mu^\varphi_t}(F)\big] = -\ddp{}{t} \ent_{\nu_t}(\PP_{0,t}F) =   2 \E_{\nu_t}\qb{ (\nabla \sqrt{\PP_{0,t}F})_{\dot C_t}^2 (\varphi)},
\end{equation}
where the first equality follows from the independence of $t$ of the following entropy decomposition:
\begin{equation}  \label{eq:entdecompcontract}
  \ent_{\nu_0}(F)
  =
  \ent_{\nu_t}(\PP_{0,t}F) + \E_{\nu_t}\big[\ent_{\mu^\varphi_t}(F)\big]
  ,\qquad
  \PP_{0,t}F(\varphi) = \E_{\mu^\varphi_t}[F].
\end{equation}
Thus the entropic stability estimate  \eqref{e:statement-entstab} leads to a differential inequality for $\E_{\nu_t}\big[\ent_{\mu^\varphi_t}(F)\big]$.
The multiscale Bakry-\'Emery criterion also gives an entropy contraction but for $\nu_t$ instead of $\mu_t^\varphi$,
i.e., it leads to a differential inequality for  $\ent_{\nu_t}(\PP_{0,t}F)$ as follows.

\begin{remark} \label{rk_MSBE}
Assume there are $\dot\lambda_t \in \R$ such that
  \begin{equation}
\forall \varphi\in X:
\qquad
\dot C_t\He V_t(\varphi)\dot C_t - \frac{1}{2}\ddot C_t\geq \dot \lambda_t \dot C_t
.
\label{eq_MSBE}
\end{equation}
Then the measure $\nu_t$ satisfies the following entropy contraction:
for any $t>0$, with $\lambda_t=\int_0^t\dot\lambda_s\, ds$,
\begin{equation} \label{eq:entcontract_BE0}
  2 \E_{\nu_t}\qb{(\nabla \sqrt{\PP_{0,t}F})_{\dot C_t}^2 (\varphi)} \leq \qa{\ddp{}{t} \log \int_0^t e^{-2\lambda_s}\, ds} \ent_{\nu_t}(\PP_{0,t}F)
  .
\end{equation}
In particular, one again has an entropy contraction estimate:
\begin{equation} \label{eq:entcontract_BE}
\ent_{\nu_0}(F)
\leq 
\frac{\int_0^\infty e^{-2\lambda_s}\, ds}{\int^t_0 e^{-2\lambda_s}\, ds} 
\E_{\nu_t}\qB{\ent_{\mu^\varphi_t}(F)}
.
\end{equation}
\end{remark}

Note that while the left-hand sides of \eqref{e:statement-entstab} and \eqref{eq:entcontract_BE0}
are identical (after taking expectation over $\nu_t$), the right-hand side of
\eqref{e:statement-entstab} is expressed in terms of the second term  on the right-hand side
of \eqref{eq:entdecompcontract} while  \eqref{eq:entcontract_BE0} is expressed in terms of the first term.

\begin{proof}
The entropy of a test function $F$ decomposes at each scale $t\geq 0$ according to   \eqref{eq:entdecompcontract}.
Under~\eqref{eq_MSBE}, 
the log-Sobolev inequalities for each renormalised measure $\nu_t$ provided by Remark~\ref{rk:LSI-mon1} give, with the same computation as the proof~\eqref{e:dEnt} of the multiscale Bakry-\'Emery criterion:
\begin{align}
\ddp{}{t}\ent_{\nu_t}(F_t)
=
-
2\E_{\nu_t}\big[\big(\nabla \sqrt{F_t}\big)_{\dot C_t}^2\big]
\leq 
- \gamma_t\ent_{\nu_t}(F_t)
  ,
\end{align}
where   $F_t(\varphi) =      \PP_{0,t}F(\varphi) = \E_{\mu^\varphi_t}[F]$ and 
\begin{equation}\gamma_t
= 
\left(\int_t^\infty e^{-2(\lambda_s-\lambda_t)}\, ds\right)^{-1}
=
-\ddp{}{t}\log\int_t^\infty e^{-2\lambda_s}\, ds
.
\end{equation}
This implies for each $t\geq 0$:
\begin{equation}
\ent_{\nu_t}(F_t) 
\leq 
\exp\qa{-\int_0^t\gamma_s\, ds } \ent_{\nu_0}(F_0)
=
\frac{\int_t^\infty e^{-2\lambda_s}\, ds}{\int_0^\infty e^{-2\lambda_s}\, ds}  \ent_{\nu_0}(F_0)
,
\end{equation}
and hence the entropy contraction  \eqref{eq:entcontract_BE} when substituted into \eqref{eq:entdecompcontract}.
\end{proof}

\begin{proof}[Proof of Corollary~\ref{cor: entropic stabiliy}]
From \eqref{e:fluctuationmeasure-def}, recall that the Polchinski semigroup $\PP_{0,t}$ coincides with the fluctuation measure $\mu_t^\varphi$. Thus for smooth $F>0$ and $\varphi \in X$, one has 
\begin{align}
\nabla \log \PP_{0,t}F (\varphi) = \nabla \log \E_{\mu_t^\varphi} [F]
&= 
\frac{\E_{\mu_t^\varphi}[ F \, C_t^{-1} \zeta]}{\E_{\mu_t^\varphi} [F]} 
  -  \E_{\mu_t^{\varphi}}[ C_t^{-1} \zeta]
  \nnb
&= C_t^{-1} \Big( \E_{\mu_t^{\varphi,F}}[  \zeta] -  \E_{\mu_t^{\varphi}}[\zeta] \Big), 
\end{align}
where the measure modified by $F$ is defined by
\begin{equation}
\label{eq: tilt by F}
\frac{d\mu_t^{\varphi,F}}{d \mu_t^\varphi} (\zeta)  = \frac{F(\zeta)}{\E_{\mu_t^\varphi}[F]}.
\end{equation}
Thus, the estimate \eqref{e:statement-entstab}  we are looking for boils down to proving 
an entropic stability result \eqref{eq: entropic stability} for the measure $\mu_t^\varphi$.
Indeed, setting  $\dot\Sigma_t = C_t^{-1} \, \dot C_t \, C_t^{-1}$, then 
\begin{equation}
  2 (\nabla \sqrt{\PP_{0,t}F})_{\dot C_t}^2 (\varphi)
  =
  \frac12 (\nabla \log \PP_{0,t}F)_{\dot C_t}^2 (\PP_{0,t}F) (\varphi)
  = \frac12 (\E_{\mu_t^{\varphi,F}}[\zeta]-\E_{\mu_t^\varphi}[\zeta])_{\dot\Sigma_t }^2 (\varphi) \E_{\mu_t^\varphi}[F],
\end{equation}
and the relative entropy is given by
\begin{equation}
    \ent_{\mu_t^\varphi }(F) = \bbH(\mu_t^{\varphi,F} |\mu_t^\varphi) \; \E_{\mu_t^\varphi}[F]  .
\end{equation}

Assumption \eqref{e:ass-entstab bis} implies the assumption \eqref{eq: assumption covariance eldan} on the covariance of $\mu_t^\varphi$,
\begin{equation}
\label{eq: assumption covariance eldan bis}
\dot \Sigma^{1/2}_t \cov(T_h \mu_t^\varphi) \dot \Sigma^{1/2}_t
=     \dot \Sigma^{1/2}_t \,  \cov( \mu_t^{\varphi + C_t h}) \dot \Sigma^{1/2}_t
\leq \alpha_t \id \qquad \text{for all $h \in X$},
\end{equation}
so that Lemma~\ref{lem:entstab}  gives the claim
\eqref{e:statement-entstab}, i.e.,
\begin{equation} 
  2(\nabla\sqrt{\PP_{0,t}F} (\varphi) )_{\dot C_t}^2 \leq \alpha_t \ent_{\mu_t^\varphi}(F).
\end{equation}

Thus it is enough to show that assumption 
\eqref{e:ass-entstab} is equivalent to the covariance assumption \eqref{e:ass-entstab bis}.
From \eqref{eq: hess Vt variance}, we know that for any $\varphi \in X$,
\begin{equation}
  \He V_t (\varphi) = C_t^{-1} - C_t^{-1} \cov( \mu_t^{\varphi}) C_t^{-1},
  \end{equation}
so that 
\begin{align}
  -\dot C_t \Big( \He V_t (\varphi) - C_t^{-1} \Big) \dot C_t
  &= \dot C_tC_t^{-1} \cov( \mu_t^{\varphi} ) C_t^{-1}\dot C_t
    \nnb
  &=  C_t \dot \Sigma_t \cov( \mu_t^{\varphi} ) \dot \Sigma_t C_t
  \leq \alpha_t C_t \dot\Sigma_t C_t = \alpha_t \dot C_t,
\end{align}
where the inequality holds if and only if $\dot \Sigma \cov( \mu_t^\varphi )\dot \Sigma_t \leq \alpha_t  \, \dot \Sigma_t$.

\bigskip
 
We turn now to the second part of the claim and show how the entropic stability estimate
\eqref{e:statement-entstab} implies the entropy contraction
estimate~\eqref{eq: LSI intermediaire}.
The starting point is the time derivative of the  entropy  \eqref{e:dEnt}:
 \begin{equation}
\ddp{}{t} \E_{\nu_t}\qB{ \PP_{0,t}\Phi(F)-\Phi(\PP_{0,t}F)} 
= 2    \E_{\nu_t} \qbb{ (\nabla \sqrt{ \PP_{0,t} F })_{\dot C_t}^{2}} 
 \end{equation}
which is the counterpart of eq.~(27) in \cite{2203.04163}.
This is bounded thanks to \eqref{e:statement-entstab}:
 \begin{equation}
   \ddp{}{t} \E_{\nu_t}\qB{ \PP_{0,t}\Phi(F)-\Phi(\PP_{0,t}F)} 
   \leq \alpha_t\E_{\nu_t}\qB{ \PP_{0,t}\Phi(F)-\Phi(\PP_{0,t}F)},
 \end{equation}
and thus for any $s <t$, Gr\"onwall's lemma implies
 \begin{equation} \label{e:entstab-last}
   \E_{\nu_t}\qB{ \PP_{0,t}\Phi(F)-\Phi(\PP_{0,t}F)} \leq e^{\int_s^t\alpha_u\, du} \; \E_{\nu_s}\qB{ \PP_{0,s}\Phi(F)-\Phi(\PP_{0,s}F)}.
 \end{equation}
Taking $t\to\infty$, the entropy is recovered from the left-hand side, as in \eqref{eq: consequence continuity},
and therefore \eqref{eq: LSI intermediaire} holds  for any $s>0$.
\end{proof}

We finally prove Lemma~\ref{lem:entstab} following \cite{2203.04163}.

\begin{proof}[Proof of Lemma~\ref{lem:entstab}]
  It suffices to show that, for any $h \in X$,
  \begin{equation}
    \frac12 (\E_{T_h\mu}[\zeta]-\E_{\mu}[\zeta])_{\dot \Sigma}^2 
    \leq
    \alpha \bbH( T_h\mu |\mu)
    =\alpha \ent_\mu\pbb{\frac{e^{(h,\zeta)}}{\E_\mu[e^{(h,\zeta)}]}}
    .
  \end{equation}
  Indeed, for any density $F$ with $\E_\mu[F]=1$ and $\E_{\mu}[F\zeta] =\E_{T_h\mu}[\zeta]$,
  the entropy inequality \eqref{e:entineq2} applied with the test function $G: \zeta \mapsto (h,\zeta)$
  implies
  \begin{equation}
    \ent_{\mu}(F) = \sup_G \hB{\E_{\mu}[FG] - \log \E_{\mu}[e^G]}
    \geq \E_{T_h\mu}[(h,\zeta)] - \log \E_{\mu}[e^{(h,\zeta)}]
    = \bbH( T_h\mu |\mu),
  \end{equation}
  i.e., the relative entropy over probability measures
  with given mean is minimised by exponential tilts $T_h\mu$. 
  Moreover, if there is no $h$ such that
  $\E_{\mu}[F\zeta] =\E_{T_h\mu}[\zeta]$
  the relative entropy is infinite.

From now on, we may assume that $\cov(T_h\mu)$ is strictly positive definite on $X$ for all $h \in X$.
Indeed, otherwise consider the largest linear subspace $X'$ such that $\cov(T_h\mu)$ acts and is strictly positive definite on $X'$,
and note that this subspace is independent of $h$. 
Indeed, let $f\in\R^N$ be such that $\var_{\mu}((f,\zeta))=0$,
and assume without loss of generality that $\E_{\mu}[(f,\zeta)]=0$. 
Then under the assumption of exponential moments also 
$\E_{\mu}[(f,\zeta)^4]=0$ and:
\begin{align}
  \var_{T_h\mu}((f,\zeta)) 
  &\propto \frac12 \E_{\mu\otimes \mu}[(f,\zeta-\zeta')^2e^{(h,\zeta+\zeta')}]
    \nnb
  &\leq 
    \frac12 \E_{\mu\otimes \mu}[(f,\zeta-\zeta')^4]^{1/2} \E_{\mu\otimes \mu}[e^{2(h,\zeta+\zeta')}]^{1/2}
    =0
    .
\end{align}  
This implies that $\mu$ (and thus $T_h\mu$) is supported in an affine subspace of $X$ which is a translation of $X'$,
and by recentering one can replace $X$ by $X'$ in the following.
  
  The relative entropy of $T_h\mu$ can be written as:
  \begin{equation}
    \bbH( T_h\mu |\mu)
    =\E_{T_h \mu} \qa{ ( h,\zeta)} - \log(\E_{\mu}[e^{(h,\zeta)}])
    = (h,\theta) - \log(\E_{\mu}[e^{(h,\zeta)}]), \qquad \theta = \E_{T_h \mu}[\zeta]
    .
  \end{equation}
The positive definiteness of $\cov(T_h\mu)$ on $X$ implies that $X \ni h \mapsto \log \E_{\mu}[e^{(h,\zeta)}]$ is strictly convex, and hence
$h \mapsto \theta(h) = \E_{T_h\mu}[\zeta]$ is strictly increasing in any direction of $X$. Let $K$ be the image of $\theta(h)$,
and for $\theta\in K$, let $h(\theta)$ be the inverse function, and then let
  \begin{equation}
    \Gamma(\theta)
    = \bbH(T_{h(\theta)}\mu|\mu)
        = (\theta,h(\theta)) - \log \E_{\mu}[e^{(\zeta,h(\theta))}].
  \end{equation}
  Thus $\Gamma$ be the Legendre transform of the cumulant generating function of $\mu$, and
  \begin{equation}
    \bbH( T_h\mu |\mu) = \Gamma(\E_{T_h\mu}[\zeta]).
  \end{equation}
  In particular, $h(\E_\mu[\zeta]) = 0$ and properties of Legendre transform imply that, in directions of $X$,
  \begin{align}
    \nabla \Gamma(\theta) &=  (h\mapsto \E_{T_{h}\mu}[\zeta])^{-1}|_{h=h(\theta)}
    \\
    \He \Gamma(\theta) &= \qa{ \He \log \E_{\mu}[e^{(h,\zeta)}] \bigg|_{h=h(\theta)}}^{-1} = \cov(T_{h(\theta)}\mu)^{-1}
  \end{align}
  so that
  \begin{equation}
    \nabla \Gamma(\E_\mu[\zeta]) =0, \qquad \He \Gamma(\E_{T_h\mu}[\zeta]) = \cov(T_h\mu)^{-1}.
  \end{equation}
  The assumption $\dot\Sigma \cov(T_h\mu) \dot\Sigma \leq \alpha\dot\Sigma$
  implies, for $\theta \in K$,
  \begin{equation}
    \alpha \He \Gamma(\theta) \geq \dot \Sigma.
  \end{equation}
  Since $f(\theta) = \frac12 (\theta - \E_{\mu}[\zeta]))_{\dot \Sigma}^2$ satisfies $\nabla f(\E_{\mu}[\zeta]) = 0$ and $\He f = \dot\Sigma$,
  therefore for all $h \in \R^N$:
  \begin{equation}
    \alpha \bbH( T_h\mu |\mu) = 
    \alpha \Gamma(\E_{T_h\mu}[\zeta]) \geq \frac12 (\E_{T_h\mu}[\zeta]-\E_{\mu}[\zeta])_{\dot \Sigma}^2 .
    \qedhere
  \end{equation}
\end{proof}

\section{Pathwise Polchinski flow and stochastic localisation perspective}
\label{sec:pathwise}

\subsection{Pathwise realisation of the Polchinski semigroup}
\label{subsec: Pathwise realisation of the Polchinski semigroup}

From Proposition~\ref{prop:polchinski2}, we recall that
the Polchinski semigroup operates from the right:
\begin{equation}
s\leq r \leq t, \qquad 
  \PP_{s,t} = \PP_{r,t}\PP_{s,r}.
\end{equation}
Thus it acts on probability densities relative to the measure $\nu_t$:
if $d\mu_0 = F \, d\nu_0$ is a probability measure then $d\mu_t = \PP_{0,t}F \, d\nu_t$ is again a probability measure.

This should be compared with the more standard situation of
a time-independent semigroup $\TT_{s,t} = \TT_{t-s}$ that is reversible with respect to a measure $\nu$
such as the original Glauber--Langevin semigroup introduced in \eqref{eq: distribution at time t}.
In this case, one has the dual point of view that $\TT$ describes the evolution of an observable:
if $d\mu_0 = F \, d\nu$ is some initial distribution and $d\mu_t = (\TT_tF) \, d\nu$ denotes the distribution at time $t$ then,
by reversibility,
\begin{equation}
  \E_{\mu_t}[G]
  =
  \int G (\TT_tF) \, d\nu = \int (\TT_tG) F\, d\nu
  = \E_{\mu_0}[\TT_tG].
\end{equation}
The dual semigroup can be realised in terms of a Markov process $(\varphi_t)$ as
$\TT_{t}G(\varphi) = \EE_{\varphi_0=\varphi}[G(\varphi_t)]$.

Since the Polchinski semigroup is not reversible and time-dependent, this interpretation
does not apply to the Polchinski semigroup. 
Instead, the Polchinski semigroup $\PP_{s,t}$ can be realised in terms of an SDE that starts at time $t$
and runs time in the negative direction from $t$ to $s<t$: Given $t>0$, a standard Brownian motion $(B_u)_{u\geq 0}$ and $\varphi_t$,
consider the solution to
\begin{equation}
  \label{e:stochastic}
s\leq t, \qquad  \varphi_{s} =\varphi_t - \int_{s}^t \dot C_u \nabla V_u(\varphi_u) \, du + \int_s^t \sqrt{\dot C_u} \, dB_u.
\end{equation}
This is the equation for the (stochastic) characteristics of the Polchinski  equation, 
see Appendix~\ref{app:HJ} for the classical analogue of a Hamilton--Jacobi equation without viscosity term.
By reversing time direction, this backward in time SDE becomes a standard SDE.
Indeed, to be concrete, we will \emph{interpret} \eqref{e:stochastic} as $\varphi_r=\tilde\varphi_{t-r}$ where $\tilde\varphi$
is the solution to the following standard SDE with $\tilde\varphi_0=\varphi_t$ given and $\tilde B_r=B_t-B_{t-r}$:
\begin{equation}
  \label{e:stochastic-forwards-bis}
0 \leq r \leq t,
\qquad  d \tilde\varphi_r = - \dot C_{t-r}  \, \nabla V_{t-r} (\tilde\varphi_r) dr + \sqrt{\dot C_{t-r}} d\tilde B_r.
 \end{equation}
Denoting by $\E_{\varphi_t=\varphi}[\cdot]$ the expectation with respect to the solution $(\varphi_s)_{s\leq t}$ to   \eqref{e:stochastic}
with $\varphi_t= \varphi$ given, the Polchinski semigroup can be represented as follows.

\begin{proposition} \label{prop:stochastic2}
For $s\leq t$ and any bounded $F: X \to \R$,
\begin{equation} \label{e:stochastic2}
  \PP_{s,t}F(\varphi)= \E_{\varphi_t=\varphi}[F(\varphi_s)].
\end{equation}
Thus if $\varphi_t$ is distributed according to  the renormalised measure $\nu_t$ the   backward in time evolution \eqref{e:stochastic} ensures that $\varphi_s$ is distributed according to $\nu_s$ for $s<t$.
\end{proposition}

Our interpretation of this proposition is that,
while the renormalised measures $\nu_t$ are supported on increasing smooth configurations as $t$ grows,
the backward evolution restores the small scale fluctuations of $\nu_0$.
Note that for $s=0$,  the identity \eqref{e:stochastic2} states that
the fluctuation measure $\mu_t^\varphi$ introduced in \eqref{e:fluctuationmeasure-def} is the distribution of  the backward process $\varphi_0$  conditioned on $\varphi$.

\medskip
To verify Proposition~\ref{prop:stochastic2},
we change time direction so that \eqref{e:stochastic} becomes a standard (forward) SDE as follows.
Indeed, as discussed above, set $\tilde \varphi_{r} = \varphi_{t-r}$ and $\tilde B_r=B_t-B_{t-r}$.
Then \eqref{e:stochastic} becomes
\begin{align}
  \tilde \varphi_{r}
  &= \tilde\varphi_0 - \int_{t-r}^t \dot C_u \nabla V_u(\tilde\varphi_{t-u}) \, du + \int_{t-r}^t \sqrt{\dot C_u} \, dB_u
    \nnb
  &= \tilde\varphi_0 - \int_{0}^r \dot C_{t-u} \nabla V_{t-u}(\tilde\varphi_{u}) \, du + \int_{0}^r \sqrt{\dot C_{t-u}} \, d\tilde B_u,
\end{align}
i.e., $\tilde \varphi$ solves the standard SDE \eqref{e:stochastic-forwards-bis}.
It\^o's formula stated for the forward SDE for $\tilde \varphi$ is
\begin{equation}
\label{e:ito-forward}
  d\tilde f(r,\tilde\varphi_r)
  = \ddp{\tilde f}{r}(r,\tilde\varphi_r) + \LL_{t-r} \tilde f(r,\tilde \varphi_r) +     (\nabla \tilde f(r,\tilde \varphi_r), \sqrt{\dot C_{t-r}} \, d\tilde B_r)
  .
\end{equation}
In terms of $\varphi$ rather than $\tilde\varphi$ we will state this as
\begin{equation} 
\label{e:ito}
  df(s,\varphi_s)
  = \ddp{f}{s}(s,\varphi_s) - \LL_s f(s,\varphi_s) + (\nabla f(s,\varphi_s),\sqrt{\dot C_s}dB_s),
\end{equation}
where the left-hand side is interpreted as follows: with $s=t-r$,
\begin{align}
  -d_sf(s,\varphi_s)
  &= d_rf(t-r,\tilde\varphi_{r})
    \nnb
  &=
    -\ddp{f}{s}(t-r,\tilde\varphi_r)
    + \LL_{t-r} f(t-r,\tilde \varphi_r)
    + (\nabla f(t-r,\tilde \varphi_r), \sqrt{\dot C_{t-r}} \, d\tilde B_r)
    \nnb
  &=
    -\ddp{f}{s}(s,\varphi_s)
   + \LL_{s} f(s,\varphi_s)
    - (\nabla f(s,\varphi_s), \sqrt{\dot C_{s}} \, dB_s).
\end{align}
In particular, if $f$ is smooth and bounded and satisfies $(\partial_s-\LL_s)f = 0$ then
\begin{equation} \label{e:ito-expectation}
  \E_{\varphi_t=\varphi}[f(s,\varphi_s)] = f(t,\varphi).
\end{equation}

\begin{proof}[Proof of Proposition~\ref{prop:stochastic2}]
It is enough to prove the claim for bounded smooth $F$ and then extend it by density.
The claim follows from \eqref{e:ito-expectation}
with $f(t,\varphi)=\PP_{s,t}F(\varphi)$ which gives
\begin{equation}
  \E_{\varphi_t=\varphi}[F(\varphi_s)]
  = \E_{\varphi_t=\varphi}[\PP_{s,s}F(\varphi_s)]
  = \E_{\varphi_t=\varphi}[f(s,\varphi_s)] = f(t,\varphi) = \PP_{s,t}F(\varphi).
\end{equation}
This proves \eqref{e:stochastic2}.
For the last statement,
recall from \eqref{e:polchinski-semigroup} that 
$\E_{\nu_s} \qa { F } = \E_{\nu_t} \qa{ \PP_{s,t} F } = \E_{\nu_t} \qa{ F(\varphi_s) }$.
This characterises the distribution at time $s$ of the process \eqref{e:stochastic}.
\end{proof}

Finally, we will consider below an analogue of the backward in time SDE~\eqref{e:stochastic} started at time $t=+\infty$, 
see~\eqref{e:coupling}. 
Equation~\eqref{e:coupling} can analogously be interpreted by reversing time as follows.
Fix any smooth \emph{time-reversing} reparametrisation $a: [0,+\infty] \to [0,+\infty]$.
For simplicity, one can choose $a(t)=1/t$ with $a(0)=+\infty$ and $a(+\infty)=0$.
As in Remark~\ref{rk:LSI-mon2}, set
\begin{equation}
  \tilde C_t = C_{a(t)},
  \qquad
  \tilde V_t = V_{a(t)},
\end{equation}
and also
\begin{equation}
  \tilde \varphi_t = \varphi_{a(t)},
  \qquad
  \tilde B_t=B_{a(t)}.
\end{equation}
Analogously to   \eqref{e:stochastic-forwards-bis},
the solution \eqref{e:coupling} can then be interpreted as $\varphi_t=\tilde \varphi_{a(t)}$ where $\tilde\varphi$
is the solution to the standard SDE:
\begin{equation} \label{e:coupling-reversed}
  d\tilde\varphi_t = \dot C_{a(t)} \dot a(t) \nabla V_{a(t)}(\tilde \varphi_t) \, dt + \sqrt{\dot C_{a(t)}} \, d\tilde B_t, \qquad \tilde\varphi_0=0.
\end{equation}
More generally, as in Remark~\ref{rk:LSI-mon2},  the SDEs~\eqref{e:stochastic}--\eqref{e:coupling} are invariant under reparametrisation, 
thus $[0,+\infty]$ has no special significance and we could have used $[0,1]$ from the beginning instead.

\medskip

We prefer to consider the backward in time evolution corresponding to $\varphi$ (rather than the forward SDE for $\tilde \varphi$) to comply with the
convention that the renormalised potential $V_t$ evolves forward in time according to the Polchinski equation.
From the stochastic analysis point of view, on the other hand, this convention of a stochastic process running backwards in time is less standard and 
related literature which focuses on the SDE rather than the renormalised potential, as we do,
thus uses the opposite convention (see Sections~\ref{sec:stochloc} and~\ref{sec:Foellmer}).

\subsection{Example: Log-Sobolev inequality by coupling}
\label{sec: Example: log-Sobolev inequality by coupling}

Using the representation \eqref{e:stochastic}-\eqref{e:stochastic2}
of the semigroup $\PP_{s,t}$ in terms of the above stochastic process,
one can alternatively prove Theorem~\ref{thm:LSI-mon} using synchronous coupling by adapting the proof from \cite{MR3330820}
for the Bakry--\'Emery theorem.

\begin{proof}[Proof of Theorem~\ref{thm:LSI-mon}]
  Given $t>0$,
  define $(\varphi_s)_{s\leq t}$ and $(\varphi'_s)_{s\leq t}$ as in \eqref{e:stochastic} coupled using the same Brownian motions.
   Then for $s<t$,
   \begin{align}
     e^{-2\lambda_t}  &(\varphi_t -\varphi_t')_{\dot C_t^{-1}}^2
     -
     e^{-2\lambda_s}  (\varphi_s -\varphi_s')_{\dot C_s^{-1}}^2
     \\
     &\qquad
     = \int_s^{t}\bigg[ e^{-2\lambda_u} \Big(-2\dot \lambda_u (\varphi_u -\varphi_u')_{\dot C_{u}^{-1}}^2 - (\varphi_u-\varphi_u')_{\dot C_{u}^{-1}\ddot C_u \dot C_u^{-1}}^2 
     \nnb
     &\hspace{4cm}+ 
     2(\dot C_{u}(\nabla V_{u}-\nabla V_u'), \dot C_{u}^{-1}(\varphi_u-\varphi_u') \Big) \bigg]\, du 
     \ \geq 0,
       \nonumber
   \end{align}
   where the inequality follows from the assumption \eqref{e:assCt-mon} and the mean value theorem.
   Thus
   \begin{equation}
     (\varphi_0 -\varphi_0')_{\dot C_0^{-1}}^2 = e^{-2\lambda_0}(\varphi_0 -\varphi_0')_{\dot C_0^{-1}}^2 \leq  e^{-2\lambda_t} (\varphi_t -\varphi_t')_{\dot C_t^{-1}}^2
   \end{equation}
    with $\varphi_t=\varphi$ and $\varphi_t'=\varphi'$ (the dynamics runs backwards) and the mean value theorem gives
   \begin{align}
      |\PP_{0,t}F(\varphi)-\PP_{0,t}F(\varphi')|
     &    = \big|\E \big[(\nabla F(\psi_0), \varphi_0-\varphi_0')\big]\big|
     \nnb
     &    = 2  \big|\E \Big[\big(\nabla \sqrt{F(\psi_0)}, \sqrt{F(\psi_0)} (\varphi_0-\varphi_0') \big)\Big]\big| 
       \nnb
     & \leq 2  \E \Big[ (\nabla \sqrt{F(\psi_0)})_{\dot C_0}^2 \Big]^{1/2} \; \E \Big[F(\psi_0) (\varphi_0-\varphi_0')_{\dot C_0^{-1}}^2 \Big]^{1/2}\nnb
     & \leq 2 e^{-\lambda_t}   \E \Big[ (\nabla \sqrt{F(\psi_0)})_{\dot C_0}^2 \Big]^{1/2}
       \; \E \big[F(\psi_0) \big]^{1/2} \; 
       \sqrt{(\varphi-\varphi')_{\dot C_t^{-1}}^2} 
   \end{align}
    for some $\psi_0$ between $\varphi$ and $\varphi'$.
    Taking $\varphi-\varphi' = \sqrt{\dot C_t}f$ with $|f|_2 \to 0$ gives
   \begin{equation}
     (\nabla \sqrt{ \PP_{0,t}F })_{\dot C_t}^2
     \leq e^{-2\lambda_t} \PP_{0,t}(\nabla \sqrt F)_{\dot C_0}^2.
   \end{equation}
   This is   \eqref{eq: exp decay exchanged gradient}.
 \end{proof}

\subsection{Example: Coupling with the Gaussian reference measure}
\label{sec:coupling}

Since, by \eqref{e:nu-def-bis},
\begin{equation}
  \E_{\nu_t}[F] = \PP_{t,\infty}F(0),
\end{equation}
one can obtain the following coupling of the field distributed under the measure $\nu_t$ with that of the associated driving  Gaussian field
from the stochastic realisation of $\PP_{t,\infty}$.

\begin{corollary} \label{cor:coupling}
  The distribution of $\nu_t$ is realised by the solution to the SDE
  (which we recall can be interpreted as discussed around \eqref{e:coupling-reversed}):
  \begin{align} \label{e:coupling}
    \varphi_t
    &= - \int_t^\infty \dot C_u \nabla V_u(\varphi_u) \, du + \int_t^\infty \sqrt{\dot C_u} \, dB_u
      \nnb
    &= -\int_t^\infty \dot C_u \nabla V_u(\varphi_u) \, du + \Gamma_t 
    .
  \end{align}
  In particular, at $t=0$, this provides a coupling of the full interacting field $\varphi_0$ with the Gaussian reference field $\Gamma_0$.
\end{corollary}

As an application of the above coupling, one can relate properties of the Gaussian measure to the interacting one,
see for example \cite{MR4399156,MR4665719}.

\subsection{Renormalised potential and martingales}
The stochastic process \eqref{e:stochastic} can also be used to obtain a representation of the renormalised potential as follows.
These are stochastic interpretations of the formulas in Lemma~\ref{lem:dVdH}.

\begin{proposition} \label{prop:V-mart}
\begin{align} \label{e:V-mart}
  V_t(\varphi)
  &= \E_{\varphi_t=\varphi}\qa{V_0(\varphi_0) + \frac12 \int_0^t (\nabla V_s(\varphi_s))^2_{\dot C_s} \, ds}
    \nnb
  &= \E_{\varphi_t=\varphi}\qa{V_0\Big(\varphi - \int_0^t \dot C_s \nabla V_s(\varphi_s)\, ds + \int_0^t \sqrt{\dot C_s} \, dB_s\Big) + \frac12 \int_0^t (\nabla V_s(\varphi_s))^2_{\dot C_s} \, ds}
\end{align}
and $M_s = V_s(\varphi_s) + \frac12 \int_s^t (\nabla V_u(\varphi_u))^2_{\dot C_u} \, du$ is a martingale (with respect to the backward filtration).
\end{proposition}

\begin{proof}
It suffices to show that $M_s$ is a martingale.
By It\^o's formula interpreted as in \eqref{e:ito},
\begin{align}
  dM_s
  &= (\ddp{V_s}{s})(\varphi_s)\, ds
    -\LL_s V_s(\varphi_s) \,ds - \frac12 (\nabla V_s)_{\dot C_s}^2\, ds + \text{martingale}
    \nnb
  &= (\ddp{V_s}{s})(\varphi_s)\, ds
    - \frac12 \Delta_{\dot C_s} V_s(\varphi_s)\, ds + \frac12 (\nabla V_s)_{\dot C_s}^2\, ds  + \text{martingale}.
\end{align}
By Polchinski's equation for $V_s$, the right-hand side is a martingale.
\end{proof}

The gradient and Hessian of the renormalised potential have similar representations.
\begin{proposition}
\begin{equation} \label{e:gradV-stoch}
  \nabla V_t(\varphi) = \PP_{0,t}[\nabla V_0](\varphi) = \E_{\varphi_t=\varphi}[\nabla V_0(\varphi_0)]
\end{equation}
and $M_s = \nabla V_s(\varphi_s)$
is a martingale (always with respect to the backwards filtration). Moreover, 
\begin{equation} \label{e:HessV-stoch}
  \He V_t(\varphi) = \E_{\varphi_t=\varphi}\qa{\He V_0(\varphi_0) - \int_0^t \He V_s(\varphi_s)\dot C_s\He V_s(\varphi_s)\, ds}
\end{equation}
and $M_s = \He V_s(\varphi_s)- \int_s^t \He V_u(\varphi_u)\dot C_u \He V_u(\varphi_u)\, du$ is a martingale.
\end{proposition}

\begin{proof}
Again, by It\^o's formula \eqref{e:ito} and since $U_s=\nabla V_s$ satisfies $\partial_s U_s = \LL_s U_s$ by  \eqref{e:dVdH},
\begin{equation}
  dU_s(\varphi_s)
  = \ddp{U_s}{t}(\varphi_s) \, ds - \LL_s U_s(\varphi_s) \, ds + (\nabla U_s, \sqrt{\dot C_s}\, dB_s) = (\nabla U_s, \sqrt{\dot C_s}\, dB_s).
\end{equation}
Thus $U_t$ is a martingale, and the expression for its expectation also follows.
Similarly, $H_s=\He V_s$ satisfies $\partial_s H_s = \LL_s H_s-H_s\dot C_sH_s$ by  \eqref{e:dVdH}, and therefore
\begin{equation}
  dH_s(\varphi_s)
  = -H_s(\varphi_s)\dot C_s H_s(\varphi_s) \, ds + (\nabla H_s, \sqrt{\dot C_s}\, dB_s)
\end{equation}
so that $H_s(\varphi_s)-\int_s^t \He V_u(\varphi_u)\dot C_u \He V_u(\varphi_u)\, du$ is a martingale.
\end{proof}

\begin{figure}[h] 
\centering
\includegraphics[width=3in]{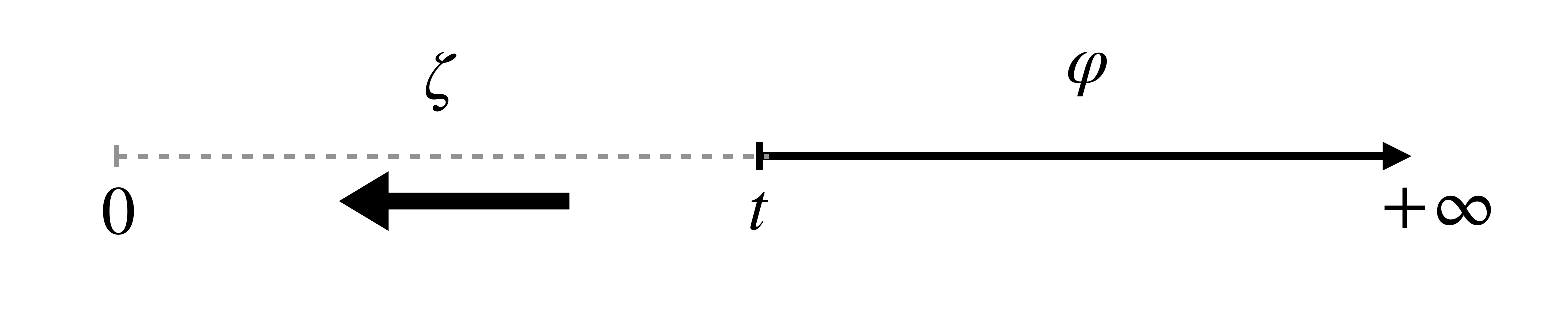}
\caption{
In the renormalisation group approach, the small scales $\zeta$ are averaged out and one considers the projection of the measure to the variables $\varphi$ encoding the large scales; in the figure above, the Polchinski flow goes from $0$ to $+\infty$. 
In fact, $\zeta$ and $\varphi$ play symmetric roles: in particular for $t=0$ the original measure is coded by $\varphi$, while instead for $t =+\infty$ the  original measure is coded by $\zeta$. Stochastic localisation puts the emphasis on the variable  $\zeta$ and therefore flows in the opposite direction (depicted by the thick arrow). \label{fig:direction-flow}
}
\end{figure}

\subsection{Stochastic localisation perspective}
\label{sec:stochloc}

The stochastic evolution \eqref{e:stochastic} has so far been interpreted  as the characteristics associated with the Polchinski equation~\eqref{e:polchinski-bis}.
In this section, we are going to see that this stochastic process is also, after a suitable change of parametrisation, the flow of the \emph{stochastic localisation}, introduced by Eldan. We refer to \cite{Eldan_ICM} for a survey on this method and its numerous applications in general, and to \cite{2203.04163} for more specific developments on modified log-Sobolev inequalities.
The relation between stochastic localisation and a semigroup approach was already pointed out in \cite{2107.09496}.

From Lemma~\ref{lem:fluctuationmeasure}, we recall that
the gradient and Hessian of the renormalised potential $V_t$ can be interpreted as a mean and covariance of the fluctuation measure
$\mu_t^\varphi$ defined in \eqref{e:fluctuationmeasure-def} by
\begin{equation} \label{e:fluctuationmeasure}
  \PP_{0,t}F(\varphi) = \E_{\mu_t^\varphi}[F].
\end{equation}
The measure $\mu_t^\varphi$ is related to $\mu_t^0$ by the exponential tilt $e^{(C_t^{-1}\varphi,\zeta)}$, i.e.,
by the external field $C_t^{-1}\varphi$.
In particular, by Lemma~\ref{lem:fluctuationmeasure}, the gradient of $V_t$ can be written as
\begin{equation}
  \nabla V_t(\varphi) = \E_{\mu_{t}^{\varphi}} [C_t^{-1}(\varphi-\zeta)]
  = C_t^{-1}(\varphi-\E_{\mu_{t}^{\varphi}}[\zeta])
\end{equation}
where $\E_{\mu_{t}^{\varphi}}[\zeta] \in X$ is the mean of $\mu_t^\varphi$.
The stochastic representation \eqref{e:stochastic} can therefore be written
in terms of the fluctuation measure instead of the renormalised potential.
Indeed, let
\begin{equation}
\label{eq: def dot Sigma}
h_t= C_t^{-1}\varphi_t, \qquad
\mu_t = \mu_t^{\varphi_t} = \mu_t^{C_t h_t}, \qquad
\dot\Sigma_t =-\ddp{}{t}C_t^{-1} = C_t^{-1}\dot C_t C_t^{-1}.
\end{equation}
Since
\begin{equation}
  \dot C_t \nabla V_t(\varphi_t) = \dot C_t C_t^{-1}(\varphi-\E_{\mu_{t}^{\varphi}}[\zeta])
   = C_t \dot\Sigma_t(\varphi-\E_{\mu_{t}^{\varphi}}[\zeta]),
\end{equation}
the external field $h_t=C_t^{-1}\varphi_t$ satisfies the following SDE equivalent to \eqref{e:coupling}:
By the It\^o formula \eqref{e:ito} with $f(t,\varphi_t) = C_t^{-1}\varphi_t$,
\begin{align} \label{e:SDE-h}
  h_t
  &= - \int_t^\infty df(u,\varphi_u)\nnb
  &= \int_t^\infty \dot \Sigma_u \varphi_u \, du - \int_t^\infty C_u^{-1}\dot C_u \nabla V_u(\varphi_u) \, du + \int_{t}^\infty C_u^{-1}\dot C_u^{1/2} dB_u\nnb
  &= \int_t^\infty \dot\Sigma_u \E_{\mu_u}[\zeta] \, du + \int_{t}^\infty C_u^{-1}\dot C_u^{1/2} dB_u
    \nnb
  &= \int_t^\infty \dot\Sigma_u \E_{\mu_u}[\zeta] \, du + \int_{t}^\infty \dot\Sigma_u^{1/2} dB_u,
\end{align}
where the last equality holds in distribution in the case that $\dot C_u$ and $C_u^{-1}$ do not commute.

What is known as \emph{stochastic localisation} is the process $(h_t)$ with the direction of time reversed.
Thus in the stochastic localisation perspective, the renormalised potential and measure only play implicit roles, and the main object
of study is the
stochastic process \eqref{e:SDE-h} and the fluctuation measure \eqref{e:fluctuationmeasure}.
For this perspective, it is more convenient to assume that time is parameterised by $[0,T]$ (rather than our previous standard choice $[0,+\infty]$
--- but again everything is reparametrisation invariant, so this is only for notational purposes).
The fluctuation measure $\mu_t = \mu_t^{\varphi_t}$ then ``starts'' at the final time $t=T$ as the full measure
of interest, and as $t$ decreases (time runs backwards) its fluctuations get absorbed into the renormalised measure $\nu_t$
until the fluctuation measure $\mu_t$ ``localises'' to a random Dirac measure $\mu_0 = \delta_{\varphi_0}$ at time $t=0$,
with $\varphi_0$ distributed according to the full measure $\nu_0=\mu_T$. See also Figure~\ref{fig:direction-flow}.

Although time runs backwards from $T$ to $0$ in the stochastic localisation perspective written with our time convention, 
let us change time direction to obtain a forward SDE and connect with the literature on stochastic localisation.
Recalling \eqref{eq: def dot Sigma}, the initial measure $\nu_0 = \mu_T$ coincides with the fluctuation measure 
at time $T$ as $h_T =0$.
As done previously, we will always use tildes to denote change of time:
\begin{equation}
\label{eq: time change}
  \tilde \varphi_t = \varphi_{T-t},
  \quad
  \tilde C_t = C_{T-t},
  \quad \tilde V_t = V_{T-t},
  \quad
  \tilde \mu_t = \mu_{T-t},
  \qquad
  \tilde h_t = h_{T-t},
  \qquad
  \tilde {\dot \Sigma}_t = \dot \Sigma_{T-t}.
\end{equation}
Using the notation ${\bf b}(\mu) = \E_\mu[\zeta]$ for the mean of $\mu$,
the SDE \eqref{e:SDE-h} for $\tilde h$ can then be written as:
\begin{equation}
  d\tilde h_t
  = \tilde{\dot\Sigma}_t {\bf b}(\tilde\mu_t) \, dt +  \tilde {\dot\Sigma}_t^{1/2} d\tilde B_t.
  \label{eq_stochloc}
\end{equation}
This equation is the same as the \emph{stochastic localisation} as it appears for example in \cite[Fact 14]{2203.04163}
(after dropping tildes from the notation and with $y_t$ there corresponding to $\tilde h_t$).

\medskip

The stochastic localisation perspective is different from our renormalisation group perspective
in that the object of interest is (again) the fluctuation measure. 
For example, in the one-variable case $|\Lambda|=1$,
starting from a measure $\tilde\mu_0 (dx)
= \nu_0(dx)  \propto e^{-H(x)}$ (possibly log-concave), the strategy is to make it more convex by considering 
\begin{equation}
\tilde\mu_t (d \zeta) \propto e^{-H(  \zeta ) - \frac{t}{2} \zeta^2 + \tilde h_t  \zeta}\, d\zeta
\end{equation}
with the choice of the process $\tilde h_t$ such that for any test function 
\begin{equation}
\label{eq: constant measure}
\forall t \geq 0: \qquad 
\bbE_{\nu_0} [F] = \bbE \Big[ \bbE_{\tilde\mu_t} ( F) \Big].
\end{equation}
In this one variable example, the fluctuation measure above is the counterpart of \eqref{eq: stricte convexite fluct mesure}
for the choice $C_t  = 1/(1+t)$ with $t$ decreasing  from $+\infty$ to $0$ instead of $C_t = t$ with $t \in [0,1]$.
With this reparametrisation, one gets from \eqref{eq: def dot Sigma} that $\dot\Sigma_t =1$  
so that
\begin{equation}
t \geq 0: \qquad
d \tilde h_t = {\bf b}(\tilde\mu_t) \, dt + d \tilde B_t,
\qquad \text{with} \quad \tilde h_0 =0.
\end{equation}

Starting from a general measure $\tilde\mu_0$, the primary concern in the stochastic localisation perspective is the measure $\tilde\mu_t$ which is now uniformly convex with Hessian at least $t$ (if say $H$ is log-concave), thus general concentration inequalities hold for the twisted measure and can be transferred to $\tilde\mu_0$ thanks to \eqref{eq: constant measure}.  For example, this is a key tool  in current progress on the KLS conjecture, see \cite{Eldan_ICM} for a review.  The larger $t$ is, the better in this respect. 
However, as $t$ grows the twisted measure $\tilde\mu_t (d\zeta)$ loses the features of the original $\tilde\mu_0$ so there is a trade-off in the choice of $t$. 
Contrary to our renormalisation point of view, in the stochastic localisation point of view, the distribution of   $h_t = C_t^{-1} \varphi_t$ (which is given in terms of $\tilde\mu_t$ in~\eqref{eq_stochloc} but can also be written 
in terms of our renormalised measure) does not play an important role (see Figure \ref{fig:direction-flow}).
The process $\tilde h_t$ is there to twist the measure and sometimes if one adds the correct $\dot C_t$ there are preferred directions to add the convexity.

\section{Variational and transport perspectives on the Polchinski flow}
\label{sec:variational-transport}

In this section, we discuss variational and transport-related perspectives on the Polchinski flow.
We refer to \cite{2202.11737} for additional perspectives such as an interpretation in terms of the Otto calculus that we do not discuss here.

\subsection{F\"ollmer's problem}
\label{sec:Foellmer}

By \eqref{e:coupling}, the distribution $\nu_0$ can be realised as the final time distribution $\varphi_0$ of the SDE:
\begin{equation} 
\label{e:coupling-bis}
  \varphi_t
  = - \int_t^\infty \dot C_u \nabla V_u(\varphi_u) \, du + \int_t^\infty \sqrt{\dot C_u} \, dB_u,
\end{equation}
where we recall that the backwards SDE can be interpreted by reversing time as in \eqref{e:coupling-reversed};
as pointed out in Remark \ref{rk:LSI-mon2}, one could have also considered a parametrisation on a bounded time interval.
One can ask whether the distribution $\nu_0$ can be obtained more efficiently if $\nabla V_u(\varphi_u)$
is replaced by another \emph{drift} $U_u(\varphi_u)$, i.e., as the distribution of $\varphi_0^U$ when
$\varphi^U$ is a strong solution of the SDE (again written backward in time):
\begin{equation} 
\label{e:coupling-U}
  \varphi_t^U
  = - \int_t^\infty \dot C_u U_u(\varphi_u^U) \, du + \int_t^\infty \sqrt{\dot C_u} \, dB_u,
\end{equation}
where the parameter $t$ takes values in $[0,+\infty]$ and $\varphi_\infty^U = 0$.
More generally, one could consider non-Markovian adapted processes $U$ and the following remains valid, see e.g. \cite{MR3112438}.

Denote by $\gamma_0 = \Pg_{C_\infty}$ the distribution of the Gaussian reference measure, i.e., of $\int_0^\infty \sqrt{\dot C_u}\, dB_u$.

\begin{theorem}
\label{thme: Follmer general}
The gradient of the renormalised potential $V_t$ of the Polchinski flow \eqref{e:V-def}  can be interpreted as the optimal drift in \eqref{e:coupling-U} in the following sense:
\begin{equation}
\label{eq: Follmer twisted entropy}
\bbH ( \nu_0 | \gamma_0)
= \frac{1}{2} \bbE \qa{ \int_0^\infty   | \nabla   V_t ( \varphi_t) |_{  {\dot C}_t}^2\,  dt  }
\leq  \; \frac{1}{2} \bbE \qa{ \int_0^\infty |  U_t( \varphi_t^U) |_{  {\dot C}_t}^2 \, dt },
\end{equation}
for any drift $U$ such that \eqref{e:coupling-U} has a strong solution with $\varphi_0\sim \nu_0$.
Recall that $(\varphi_t)$ follows \eqref{e:coupling-bis}.
\end{theorem}

\begin{proof}
Let $U$ be a such that there is a strong solution $(\varphi_t^U)$ of  \eqref{e:coupling-U} 
with $\bbE \big[ \int_0^\infty |U_t( \varphi_t^U)|_{  {\dot C}_t}^2 \, dt \big] < \infty$.
By construction $\varphi_0^U$ has law $\nu_0$ so that the relative entropy is given by 
\begin{equation}
\bbH ( \nu_0 | \gamma_0)
= V_\infty (0) - \bbE [ V_0 (  \varphi_0^U) ]
=  \int_0^\infty dt \; \frac{\partial}{\partial t} \bbE [ V_t ( \varphi_t^U ) ],
\end{equation}
with $\varphi^U$ evolving according to \eqref{e:coupling-U}, and
where we used that 
$\nu_0(d\varphi) = e^{+V_\infty(0)} e^{-V_0(\varphi)} \gamma_0(d\varphi)$
with normalisation factor given by 
$e^{-   V_\infty (0)} =  \Eg_{C_\infty} \q{e^{- V_0(\zeta)}}$ as in \eqref{e:nu-def-bis}.
The renormalised potential follows the  Polchinski equation \eqref{e:polchinski-bis}:
\begin{equation}
\label{e:polchinski-bis upside down}
t \in (0,\infty), \qquad 
\ddp{}{t}   V_t =  \frac12 \Delta_{  {\dot C}_t}   V_t - \frac12 (\nabla   V_t)_{  {\dot C}_t}^2.
\end{equation}
Therefore, by It\^o's formula,
\begin{align}
\frac{\partial}{\partial t} \bbE [    V_t (  \varphi_t^U) ]
& = \bbE \qa{ \ddp{}{t}    V_t (  \varphi_t^U) 
+ \big( \nabla   V_t (  \varphi_t^U) ,    U_t (  \varphi_t^U) \big)_{  {\dot C}_t}   
- \frac12 \Delta_{  {\dot C}_t}    V_t (  \varphi_t^U ) } \nnb
& = \bbE \qa{ -  \frac12 (\nabla   V_t(  \varphi_t^U) )_{  {\dot C}_t}^2 +  \big( \nabla   V_t (  \varphi_t^U) ,    U_t (  \varphi_t^U ) \big)_{  {\dot C}_t}  
  } \nnb
& =  \frac12 \bbE \qbb{ -  \big( \nabla   V_t(  \varphi_t^U)   - U_t (  \varphi_t^U) \big)^2_{  {\dot C}_t}  }
 + \frac12 \bbE \qbb{ (U_t (  \varphi_t^U) )_{  {\dot C}_t}^2  } ,
\end{align}
where we used the Polchinski equation \eqref{e:polchinski-bis upside down} on the second line. Thus 
\begin{equation}
\frac12 \bbE \qa{ \int_0^\infty  ( U_t)_{{\dot C}_t}^2 \, dt }  = \bbH ( \nu_0|\gamma_0)
+ \frac12 \bbE \qa{ \int_0^\infty  \big( \nabla   V_t -    U_t  \big)^2_{  {\dot C}_t} \, dt } ,
\end{equation}
and the gradient of the renormalised potential $V_t$ provides the optimal drift. 
This completes the proof of Theorem~\ref{thme: Follmer general}.
\end{proof}

It turns out that the right-hand of \eqref{eq: Follmer twisted entropy} is, in fact, the relative entropy $\bbH({\bf Q}|{\bf P})$
of the path measure $\bf Q$ associated with \eqref{e:coupling-U} with respect to that of the
Gaussian reference process $\bf P$.

\begin{proposition}
The relative entropy of the path measure $\bf Q$  associated with a strong solution of  \eqref{eq: Follmer stochastic}
with respect to the path measure $\bf P$ of the Gaussian reference measure is given by
\begin{equation}
\label{eq: dynamic cost}
\bbH ( {\bf Q} | {\bf P})
= \frac{1}{2} \, \bbE \qa{ \int_0^\infty   | U_t (\varphi_t^U) |_{\dot C_t}^2 \, dt  }
.
\end{equation}
\end{proposition}

\begin{proof}
  This is essentially a consequence of Girsanov's theorem, see \cite{MR3112438} for details.
\end{proof}

Since $\bbH ( {\bf Q} | {\bf P}) \geq \bbH(\nu_0|\gamma_0)$ always holds,
by the entropy decomposition \eqref{eq: entropy decomposition} and the fact that the laws of $\nu_0$ and $\gamma_0$
are marginals of the path measures ${\bf Q}$ and ${\bf P}$ respectively, the above shows that
the optimal drift $U_t=\nabla V_t$ in fact achieves equality: $\bbH ( {\bf Q} | {\bf P}) = \bbH(\nu_0|\gamma_0)$.
\bigskip

The above question was already studied by F\"ollmer \cite{zbMATH00193914}, and we refer to
\cite{MR3112438} for an exposition of this and connections with Gaussian functional inequalities.
F\"ollmer's objective was to find the optimal drift $b_t$ such that the process $(X_t)_{t \in [0,1]}$ defined
by the following SDE and distributed at time $t =1$ according to a given target measure $\nu$:
\begin{equation}
\label{eq: Follmer stochastic}
X_0 = 0, \quad d X_t = b_t (X_t) \, dt + dB_t \quad \text{and} \quad 
X_1 \sim \nu, 
\end{equation}
minimises the dynamical cost
\begin{equation}
  \frac{1}{2} \bbE \qa{ \int_0^1   |b_t (X_t) |^2 \, dt  }
\end{equation}
over all possible drifts $b$. 
Up to time reversal, parametrisation by $[0,+\infty]$ instead of $[0,1]$,
and introduction of the covariances $\dot C_t$, this is exactly the set-up
of \eqref{e:coupling-U}. 
For us the introduction of the covariances $\dot C_t$ is an important point, though, with the interpretation that the integral
is now an integral over scales measured by the infinitesimal covariances $\dot C_t$ which can also be interpreted as metrics as in Section~\ref{sec:geometry}.

\bigskip

More generally, one can look for the optimal drift to built a target probability measure of the form 
$F(\varphi)  \, \nu_0 (d\varphi)$ using now the optimal stochastic flow
as a reference process, i.e., we want to determine the drift $U$ such that for the process $(\psi_t)_{t\in[0,+\infty]}$ given for $t\geq 0$ by
\begin{equation}
\label{eq: EDS drift U 2}
 \psi_t = - \int_t^\infty    {\dot C}_s    U_s (  \psi_s)\,  ds
 -     \int_t^\infty  {\dot C}_s \nabla    V_s (  \psi_s)\,  ds +   \int_t^\infty \sqrt{  {\dot C}_s} \, d  B_s
 ,
\end{equation}
the cost $\frac{1}{2} \bbE \big[ \int_0^1|U_t (   \psi_t)|^2_{  {\dot C}_t} dt  \big]$ is minimised
and $\psi_0$  is distributed according to $F \, d \nu_0 $.
Proceeding as in the proof of Theorem~\ref{thme: Follmer general}, the optimal drift is given in terms of 
the Polchinski semigroup \eqref{e:P-def-bis} as the gradient of
\begin{equation}
W_t(\varphi) = - \log \PP_{0,t}F(\varphi)
  =  - V_t ( \varphi)    - \log \Eg_{C_t}
    \Big[ F ( \varphi+\zeta ) 
    e^{-    V_0 ( \varphi+\zeta)} \Big].
\end{equation}
Thus one can check that 
\begin{equation}
\label{eq: Follmer twisted entropy 2}
\bbH ( F \nu_0 | \nu_0)
=  \ent_{\nu_0}(F)
= \frac{1}{2} \bbE \qa{ \int_0^\infty  | \nabla   W_t ( \varphi_t) |_{  {\dot C}_t}^2\,  dt  } .
\end{equation}
In this way, we recover from \eqref{eq: Follmer twisted entropy 2} the entropy decomposition \eqref{e:Ent-P}:
\begin{align}
2\int_0^\infty \E_{\nu_{t}} \qa{|\nabla \sqrt{\PP_{0,t} F}|_{\dot C_t}^2} \, dt
& =
\frac{1}{2} \int_0^\infty \E_{\nu_{t}} \qa{|\nabla  \log \PP_{0,t} F|_{\dot C_t}^2 \; \PP_{0,t} F } \, dt \nnb
& = 
\frac{1}{2} \bbE \qa{ \int_0^\infty  | \nabla   W_t ( \varphi_t)|_{  {\dot C}_t}^2\,  dt  } ,
\end{align}
where we used that  the process \eqref{eq: EDS drift U 2} is distributed at time $t$ with density proportional to
\begin{equation}
e^{  -W_t(\varphi) -     V_t (\varphi)} \Pg_{C_\infty-C_t}(d\varphi) \propto
\PP_{0,t}F(\varphi) \, \nu_t (d\varphi)  .
\end{equation}

\bigskip

The above is an instance of the more general version of the Schr\"odinger problem which is to find the optimal drift so that the stochastic evolution \eqref{e:coupling-U} interpolates between two probability measures $\mu$ and $\nu$.
Here, we discussed only the special case where the process starts from a Dirac measure $\mu= \delta_0$,
and refer the reader to the survey \cite{zbMATH06224697} for a general overview  and to \cite{zbMATH07058374}
for a discussion on the role of the convexity of the potential.

In Section~\ref{sec:transport}, we address a related issue, namely that in some cases, the previous flow can be modified in order to achieve an interpolation between the measure of interest and some Gaussian measure.

\subsection{Variational representation of the renormalised potential}
\label{subsec: Variational representation of the renormalised potential}

Let $\nabla V_t = \nabla V_t(\varphi_t)$ and recall that Proposition~\ref{prop:V-mart} states:
\begin{equation}
  V_t(\varphi)
  = \E_{\varphi_t=\varphi}\qa{V_0 \Big( \varphi - \int_0^t \dot C_s 
  \nabla V_s
  \, ds + \int_0^t \sqrt{\dot C_s}\,
  dB_s \Big) + \frac12 \int_0^t 
  |\nabla V_s|^2_{\dot C_s} 
  \, ds}.
\end{equation}
In particular,
\begin{equation}
  V_t(\varphi) \geq \inf_{U} \E \qa{V_0 \Big(\varphi-\int_0^t \dot C_s U_s\, ds + \int_0^t \sqrt{\dot C_s} \,
  dB_s
  \Big) + \frac12 \int_0^t |U_s|_{\dot C_s}^2 \, ds},
\end{equation}
where the above infimum is over all
adapted processes $U: [0,t] \to X$ (where adapted means backwards in time in our convention)
called \emph{drifts}.
For our current purposes, it suffices to consider $U_s=U_s(\varphi_s)$ associated with a strong solution
to the (backward in time) SDE
\begin{equation} \label{e:SDE-U}
  \varphi_s = \varphi-\int_s^t \dot C_uU_u(\varphi_u)\, du + \int_s^t \sqrt{\dot C_u}\, 
  dB_u, 
  \qquad (s\leq t).
\end{equation}
The following  proposition is a special case of the Bou\'e-Dupuis or Borell formula, see 
\cite{MR3112438},
which gives equality in the infimum and is the starting point for the Barashkov--Gubinelli method \cite{MR4173157}.
An in-depth treatment of stochastic control problems of which this is a special case is given in \cite{MR2179357}.

\begin{proposition} \label{prop:BD} 
  \begin{equation}
    V_t(\varphi) = \inf_{U}
    \E \qa{V_0 \Big( \varphi-\int_0^t \dot C_s U_s\, ds + \int_0^t \sqrt{\dot C_s}\,  
    dB_s
     \Big) 
    + \frac12 \int_0^t |U_s|_{\dot C_s}^2 \, ds}.
  \end{equation}
\end{proposition}

\begin{proof}[Sketch]
  The entropy inequality \eqref{e:entineq2} with $G(\zeta)=-V_0(\varphi+\zeta)$ applied to the  Gaussian measure $\Pg_{C_t}$ implies that for any density $F$ with $\Eg_{C_t}[F]=1$:
  \begin{equation}
    V_t(\varphi) = - \log \Eg_{C_t}[e^{-V_0(\varphi+\zeta)}] 
    \leq \ent_{\Pg_{C_t}}(F) + \Eg_{C_t} \qa{F(\zeta)V_0(\varphi+\zeta)}.
  \end{equation}
  Given any drift $U_s$, let $F \, d\Pg_{C_t}$ denote the law of $\varphi_0-\varphi$ solving \eqref{e:SDE-U}:
  \begin{equation}
    \varphi_0-\varphi =  -\int_0^t \dot C_s U_s \, ds + \int_0^t \sqrt{\dot C_s} \, 
    dB_s
    .
  \end{equation}
  Then
  \begin{equation}
    V_t(\varphi)
    \leq \E\qa{\frac12 \int_0^t |U_s(\varphi_s)|_{\dot C_s}^2\, ds 
    + V_0 \Big( \varphi-\int_0^t \dot C_s U_s\, ds + \int_0^t \sqrt{\dot C_s}\, 
    dB_s
    \Big)},
  \end{equation}
  where we used that the entropy is bounded by the first term, exactly as in  Theorem~\ref{thme: Follmer general}.
  As already discussed, the converse direction follows from Proposition~\ref{prop:V-mart}.
\end{proof}

The point of view is now that by estimating the expectation on the right-hand side above, for a general drift $U$,
one can obtain estimates on $V_t(\varphi)$, and in particular on $V_\infty(\varphi)$ which we recall from \eqref{e:Vinfty}
is equivalent to the logarithmic moment generating function of the measure $\nu_0$.
For further details and application to construction of the $\varphi^4_d$ measures, we refer to \cite{MR4173157,MR4269211}.

\subsection{Lipschitz transport}
\label{sec:transport}

Instead of a stochastic process, one could also define a map 
$\hat S_t : X \mapsto X$  transporting some measure $\hat \nu_t$ to the desired target measure $\nu_0$:
\begin{equation}
\label{eq: def hat S}
\E_{\nu_0} \qa{ F ( \varphi )} = \E_{\hat\nu_t} \qa{ F \big(\hat S_t ( \varphi) \big)} .
\end{equation}
Under an assumption on its gradient, such a transport map allows to recover functional inequalities for   $\nu_0$ from $\hat \nu_t$ as follows.
For example, assume that $\hat \nu_t$ satisfies a log-Sobolev inequality
(with quadratic form $(\cdot,\cdot)_{Q}$ and denoting the corresponding norm $|f|^2_{Q} = (f,f)_{Q}$ on $X$):
\begin{equation}
  \ent_{\hat \nu_t}[F] \leq
  2\E_{\hat\nu_t}\qB{ |\nabla \sqrt F|_{Q}^2},
\end{equation}
and that  the transport map $\hat S_t$ from $\hat\nu_t$ to $\nu_0$ has
Jacobian  $\nabla \hat S_t(\varphi)$   satisfying the uniform bound:
\begin{equation}
\label{eq: Lipschitz bound}
\forall \varphi \in X, f  \in X: \qquad 
|{}^t\nabla \hat S_t(\varphi) f |_{Q}^2
\leq 
C^2|f|^2 . 
\end{equation}
Then  $\nu_0$ satisfies also a log-Sobolev inequality: 
\begin{align}
\ent_{\nu_0} \qa{ F ( \varphi )} 
& =
\E_{\hat \nu_t} [\Phi(F \circ \hat S_t )]- \Phi(\E_{\hat \nu_t} [F \circ \hat S_t]) \nnb
& \leq 2 \E_{\hat \nu_t} \qa{  \big|\nabla \sqrt{F \circ \hat S_t }\big|_{Q}^2} \nnb
&= 2\E_{\hat  \nu_t} \qa{  \big|^t \nabla \hat S_t  \, (\nabla \sqrt{F} \circ \hat S_t) \big|_{Q}^2} \nnb
  & \leq 2C^2 \E_{\hat  \nu_t} \qa{  \big| (\nabla \sqrt{F} \circ \hat S_t)\big|^2} \nnb
&= 2C^2 \E_{\nu_0} \qa{  |\nabla \sqrt{F  }|^2} .
\label{eq_LSI_transport}
\end{align}
An analogous argument can be applied to more general functional inequalities.

This line of research was first investigated in \cite{zbMATH01631921} where it was understood that a convex perturbation of a Gaussian measure leads to a $1$-Lipschitz transport map $\hat S_t$ to this Gaussian measure,
i.e., \eqref{eq: Lipschitz bound-bis} holds with $Q=\id$ and $C=1$.
We refer to \cite{zbMATH06134307, neeman2022lipschitz, mikulincer2021brownian, mikulincer2022lipschitz, 2107.09496, Jordan_Serres,fathi2023transportation} for more recent developments as well as other applications of the Lipschitz properties of transport maps 
to functional inequalities.
In this section, we follow the work \cite{Shenfeld2022ExactRG} which derived a Lipschitz estimate of the form \eqref{eq: Lipschitz bound} from the multiscale Bakry--\'Emery criterion \eqref{e:assCt-mon} for
the covariance decomposition $\dot C_t=e^{-tA}$, and then we generalise the result also to other decompositions relevant for applications
(see Section~\ref{sec:appl-transport}).
For this generalisation, it is important that $Q$ in the condition  \eqref{eq: Lipschitz bound} is not necessarily equal to the identity.
Using the identity $\|M\|= \|^tM\|$ ($\|\cdot\|$ denotes the operator norm), 
this condition can also be equivalently stated in terms of $\nabla \hat S_t$ instead of $^t\nabla \hat S_t$ as
\begin{equation} \label{eq: Lipschitz bound-bis}
  |\nabla\hat S_t(\varphi) f|^2 \leq C^2|f|_{Q^{-1}}.
\end{equation}

Recall that the measure $\nu_0$ gets renormalised to $\nu_t$  by the Polchinski flow.
By construction  $\E_{\nu_t} [\cdot ] \propto \Eg_{C_\infty -C_t} [ e^{-V_t} \cdot ]$ and 
$\nu_t$ converges to a Dirac mass.
 As the measure  $\nu_t$ degenerates, it is more convenient to consider the measure $\hat \nu_t$ obtained by rescaling $\nu_t$ by some matrix $D_t$ so that the   measures $\nu_0$ and $\hat \nu_t$ are comparable:
\begin{equation}
\E_{\hat \nu_t} \qa{ F  (\varphi) } = \E_{\nu_t} \qa{ F (D_t \varphi)} .
\end{equation}
If $V_0 =0$, a natural choice for $D_t$ is to preserve the Gaussian measure, i.e.,
\begin{equation}
\label{eq: tuning D_t}
D_t^{-1} (C_\infty -C_t)^{-1} D_t^{-1}  = A 
\quad \Rightarrow \quad D_t A D_t = (C_\infty -C_t)^{-1} 
,
\end{equation}
where all inverses are understood to be taken on the range of $A$. 
This is implicitly assumed in the rest of the section, with $\id$ also denoting the identity matrix on the range of $A$.
Assuming that all the matrices depend smoothly on $t$ and commute, the choice~\eqref{eq: tuning D_t} implies the following useful relation:
\begin{equation}
\label{eq: derivative D_t}
2 D_t^{-1} \dot D_t  = \dot C_t (C_\infty -C_t)^{-1} 
= 
\dot C_t D_t A D_t
.
\end{equation}
Using again~\eqref{eq: tuning D_t} we get (recall $C_0=0$ and notice $D_0 =\id$):
\begin{equation}
\ddp{}{t}(D_t^{-2}) = -A\dot C_t
\quad \Rightarrow\quad
D_t = (\id -AC_t)^{-1/2}
.
\label{eq: tuning D_t_bis}
\end{equation}
By construction $\lim_{t \to \infty} D_t^{-1} \varphi = 0$ and the renormalised potential satisfies $\lim_{t \to \infty} V_t ( D_t^{-1} \varphi )= V_\infty(0)$.
This follows from the representation \eqref{e:V-def} of $V_t$ and the standing assumption that $V_0$ is bounded below.
In the same way, we can also show that
the following convergence in distribution to a Gaussian measure holds:
\begin{equation}
\label{eq: Gaussian limit}
\lim_{t \to \infty} \E_{\hat \nu_t} \qa{ F  (\varphi) } = \Eg_{C_\infty}  \qa{ F (\varphi)}
=
\Eg_{A^{-1}}  \qa{ F (\varphi)} .
\end{equation}
We are now going to study the  transport map
$\hat S_t$  between $\nu_t$ and $\hat \nu_0$ defined in \eqref{eq: def hat S}.

The properties of $\hat S_t$ are sensitive to the covariance decomposition $\dot C_t$.
It was realised in \cite{Shenfeld2022ExactRG} that for the choice $\dot C_t = e^{-t A}$ (often used in applications, see Section~\ref{sec: Applications})
the Lipschitz structure associated with $\hat S_t$ is directly related to (a variant of)  the multiscale Bakry--\'Emery criterion \eqref{e:assCt-mon}.
\begin{theorem}[\!\!\cite{Shenfeld2022ExactRG}]
\label{thme: Lipschitz}
Let $\dot C_t =e^{-tA}$ for $t\geq 0$.
Under the assumption 
\begin{equation} 
  \label{e:assCt-mon transport}
  \forall \varphi \in X: \qquad
    \dot C_t^{1/2} \He V_t(\varphi) \dot C_t^{1/2} \geq \dot \mu_t \id, 
    \quad \text{with} \quad \mu_t := \int_0^t ds \, \dot \mu_s,
\end{equation}
the transport map $\hat S_t   : X \mapsto X$ introduced in \eqref{eq: def hat S} is $\exp ( -  \frac12 \mu_t)$-Lipschitz,
i.e., \eqref{eq: Lipschitz bound-bis} holds with $Q=\id$ and $C= \exp ( -  \frac12 \mu_t)$.
\end{theorem}

From the convergence \eqref{eq: Gaussian limit} to the Gaussian measure and if $\mu_\infty := \int_0^\infty ds \, \dot \mu_s < \infty$, 
then from the previous theorem, one can extract a $\exp (- \frac12 \mu_\infty)$-Lipschitz map from the Gaussian measure
${\sf P}_{C_\infty} = {\sf P}_{A^{-1}}$ to $\nu_0$ (see \cite[Lemma 2.1]{neeman2022lipschitz}).

\bigskip

Another useful covariance decomposition (see Theorem~\ref{thm_continuous_phi4}) is of the form 
$\dot C_t = (t A + \id)^{-2}$. 
The proof of Theorem~\ref{thme: Lipschitz} can be extended to that case as follows.

%
\begin{theorem}
\label{thme: Lipschitz2}
Let $\dot C_t=(tA+\id)^{-2}$ for $t\geq 0$. 
Under the multiscale Bakry-\'Emery criterion
\begin{equation} 
  \label{e:assCt-mon transport2}
  \forall \varphi \in X: \qquad
    \dot C_t \He V_t(\varphi) \dot C_t - \frac{1}{2}\ddot C_t \geq \dot \lambda_t \dot C_t, 
    \quad \text{with} \quad \lambda_t := \int_0^t ds \, \dot \lambda_s,
\end{equation}
the inverse map $\hat S_t = S_t^{-1} : X \mapsto X$ satisfies 
\begin{equation}\label{eq_Lipschitz_hatS}
    \forall  \varphi \in X, f\in X: \qquad
    |\nabla \hat S_t(\varphi) f|^2\leq e^{-\lambda_t}|\sqrt{1+tA}\, f|^2
    = e^{-\lambda_t} |f|_{1+tA}^2
.
\end{equation}
\end{theorem}
\begin{remark}
Contrary to the Lipschitz transport map of Theorem~\ref{thme: Lipschitz}, 
the gradient of the map of Theorem~\ref{thme: Lipschitz2}
is bounded with respect to a different input norm that increases with $t$.
In particular, if $A=-\Delta+1$, then $|\cdot|_{1+tA}$ is a (discrete) Sobolev norm.
In our examples (see~Section~\ref{sec:appl-transport}),
the constants $\lambda_t$ diverge like $\log (1+t)$ (while the constants $\mu_t$ in Theorem~\ref{thme: Lipschitz} remain bounded)
and therefore the combination of $e^{-\lambda_t}(1+tA)$ remains bounded by $A$ uniformly in $t$.
Thus $\nabla\hat S_t(\varphi)$ is uniformly bounded from $|\cdot|_A$ to $|\cdot|$
and thus $^t\nabla \hat S_t(\varphi)$ from $|\cdot|$ to $|\cdot|_{A^{-1}}$.
As seen in \eqref{eq_LSI_transport},
when using transport maps to prove log-Sobolev inequalities, one can use that the Gaussian measure $\hat\nu_\infty$
satisfies a log-Sobolev inequality with quadratic form $(\cdot,\cdot)_{A^{-1}}$ to compensate the loss of regularity in the transport map and recover a log-Sobolev inequality with standard quadratic form.
\end{remark}

\medskip

\begin{proof}[Proof of Theorem \ref{thme: Lipschitz}]
Consider $S_t$ a transport map  between $\nu_0$ and $\hat \nu_t$, so that 
\begin{equation}
\E_{\nu_0} \qa{ F \big( S_t (\varphi) \big)} = \E_{\hat \nu_t} \qa{ F  (\varphi) }
=
\E_{\nu_t} \qa{ F (D_t \varphi)} .
\label{eq_def_S_t}
\end{equation}
Note that ultimately we are interested in $\hat S_t$ which is the inverse of $S_t$, see \eqref{eq: def hat S}.
We are going to determine an evolution for $S_t : X \to X$ for general general covariance decompositions $\dot C_t$ and use the precise form only to conclude the proof.
On the one hand,
\begin{equation}
 \ddp{}{t}   \E_{\nu_0} \qa{ F \big( S_t (\varphi) \big)} = 
 \E_{\nu_0} \qa{ \big( \nabla F \big( S_t (\varphi) \big) , \partial_t S_t (\varphi) \big)},
\end{equation}
and on the other hand from \eqref{e:polchinski-semigroup}:
\begin{equation}
 \ddp{}{t}  \E_{\nu_t} \qa{ F (D_t \varphi)} 
= \E_{\nu_t} \qa{ - \LL_t F (D_t \varphi) +  ( \dot D_t \varphi, \nabla F (D_t \varphi)) } ,
\end{equation}
with $    \LL_tF = \frac12 \Delta_{\dot C_t} F - (\nabla V_t, \nabla F)_{\dot C_t}$.
Integrating the Laplacian by parts, this gives 
\begin{align}
   \ddp{}{t} \E_{\nu_t} \qa{ F (D_t \varphi)} 
& = \E_{\nu_t} \qa{ \frac12 (\dot C_t\nabla V_t  (\varphi) - \frac12 \dot C_t \, (C_\infty - C_t)^{-1} \varphi 
+ D_t^{-1} \dot D_t \varphi, D_t  \nabla F (D_t \varphi)) } \nnb
& = \E_{\nu_t} \qa{ \frac12 (\nabla V_t (\varphi), D_t  \nabla F (D_t \varphi))_{\dot C_t} } \nnb
& = \E_{\nu_0} \qa{ \frac12 ( \nabla V_t (D_t^{-1} \cdot ), D_t  \nabla F )_{\dot C_t} ( S_t (\varphi) )} ,
\end{align}
where we used \eqref{eq: derivative D_t}.
Thus the evolution of $S_t$ is given by 
\begin{align}
\label{eq: evolution St}
\ddp{}{t}  S_t (\varphi)
= \frac12  D_t \, \dot C_t   \nabla V_t (D_t^{-1} S_t (\varphi) ) .
\end{align}

For $\dot C_t = e^{-t A}$, the definition \eqref{eq: tuning D_t_bis} implies $D_t = e^{\frac12 tA} = \dot C_t^{-1/2}$ and the evolution \eqref{eq: evolution St} becomes particularly simple:
\begin{align}
\label{eq: evolution St 1}
\ddp{}{t}  S_t (\varphi)
=  \frac12  \dot C_t^{1/2}  \nabla V_t \big( \dot C_t^{1/2} S_t (\varphi) \big) .
\end{align}
The Jacobian evolves according to
\begin{align}
\ddp{}{t}  \nabla S_t (\varphi)
= \frac12  \dot C_t^{1/2}  \He V_t \big( \dot C_t^{1/2} S_t (\varphi) \big)  \dot C_t^{1/2} \nabla S_t (\varphi) .
\end{align}
As a consequence of \eqref{e:assCt-mon transport}, the Gr\"onwall inequality implies
\begin{align}
\forall f \in X, \quad 
\ddp{}{t}  | \nabla S_t (\varphi)f|^2 \geq \dot \mu_t ( \nabla S_t (\varphi) f)^2
\quad \Rightarrow \quad
| \nabla S_t (\varphi) f|^2 \geq \exp ( \mu_t ) |f|^2 .
\end{align}
By the inverse function theorem, we deduce that the operator norm of $\nabla \hat S_t $ is less than 
$\exp ( - \frac12 \mu_t)$.
\end{proof}
\begin{proof}[Proof of Theorem~\ref{thme: Lipschitz2}]
For each $t\geq 0$, we look for a matrix $B_t$ depending only on $C_t$ and its derivatives, 
commuting with them, and such that we can set up a Gr\"onwall estimate for $(B_t\nabla S_t(\varphi)f)^2$ for each $\varphi,f\in X$. 
Using that $\dot C_t = (tA+\id)^{-2}$, the definition~\eqref{eq: tuning D_t_bis} of $D_t$ implies $D_t=\dot C_t^{-1/4}$, and
in particular $D_t$ commutes with $B_s,C_s,\dot C_s,\ddot C_s$ for any $s$. 
Equation~\eqref{eq: evolution St} gives:
\begin{equation}
\ddp{}{t} \nabla S_t(\varphi)
=
\frac{1}{2}\dot C_t D_t\He V_t(D_t^{-1}S_t(\varphi)) D_t^{-1}\nabla S_t(\varphi)
.
\end{equation}
Therefore:
\begin{align}
&\ddp{}{t} \big|B_t\nabla S_t(\varphi)f\big|^2\nnb
&\qquad =
\Big(B_t\nabla S_t(\varphi)f, \Big[ B_t\dot C_t D_t \He V_t(D_t^{-1}S_t(\varphi))D_t^{-1}B_t^{-1}  + 2\dot {B_t} B_t^{-1}\Big]B_t\nabla S_t(\varphi)f\Big)
.
\end{align}
In order to use the Hessian bound~\eqref{e:assCt-mon transport2}, 
we need to choose $B_t$ in such a way that, for each $t\geq 0$:
\begin{align}
& B_t\dot C_t D_t \He V_t(D_t^{-1}S_t(\varphi))D_t^{-1}B_t^{-1}  + 2\dot {B_t} B_t^{-1}
\nnb
&\qquad 
\geq 
\dot C_t^{1/2}\He V_t(D_t^{-1}S_t(\varphi))\dot C_t^{1/2}
-\frac{1}{2}\dot C_t^{-1/2} \ddot C_t\dot C_t^{-1/2} 
.
\label{eq: lower bound operators}
\end{align}
With the choice $B_t = D_t^{-1}\dot C_t^{-1/2} = \dot C_t^{-1/4}=D_t$, we have $\dot B_t B_t^{-1} =  -\frac{1}{4}\ddot C_t \dot C_t^{-1}$
and  the left- and right-hand sides in \eqref{eq: lower bound operators} are in fact equal.
As a result, 
\begin{equation}
\forall t\geq 0:\qquad
\ddp{}{t} \big|D_t\nabla S_t(\varphi)f\big|^2
\geq 
\dot \lambda_t\big|D_t\nabla S_t(\varphi)f\big|^2
.
\end{equation}
The Gr\"onwall inequality then yields:
\begin{equation}
\forall t\geq 0:\qquad
\big|D_t\nabla S_t(\varphi)f\big|^2
\geq 
 e^{\lambda_t}|f|^2
,
\label{eq_lip_norm_th55}
\end{equation}
and the inverse function theorem yields
\begin{equation}
\forall t\geq 0:\qquad
\big|\nabla \hat S_t(\varphi)D_t^{-1}f\big|^2
\leq 
 e^{-\lambda_t}|f|^2
 ,
\end{equation}
which is equivalent to the claim.
\end{proof}

\section{Applications}
\label{sec: Applications}
In this section, we present concrete examples to which the multiscale Bakry-\'Emery criterion of Theorem~\ref{thm:LSI-mon} can be applied. 
The criterion gives a bound on the log-Sobolev constant in terms of real numbers $\dot\lambda_t$ ($t>0$) obtained  through convexity lower bounds on the renormalised potential:
\begin{equation}
\forall\varphi\in X, t\geq 0: \qquad
\dot C_t \He V_t(\varphi)\dot C_t - \frac{1}{2}\ddot C_t\geq \dot\lambda_t\dot C_t
.
\label{eq_lower_bound_Hess_V_t_sec5}
\end{equation}
These lower bounds depend on the choice of the covariance decomposition $(C_t)$. 
While Theorem~\ref{thm:LSI-mon} holds for any decomposition, checking~\eqref{eq_lower_bound_Hess_V_t_sec5} for concrete models often requires
a specific choice of decomposition. This will be illustrated in examples in the following sections.
We expect that the precise choice of the covariance decomposition is technical, as long as 
it takes into account the  important physical features of the model, e.g., the mode structure explained
in Example~\ref{example_free_field}.

For now, we discuss the sharpness of the criterion in Theorem~\ref{thm:LSI-mon}.
To fix ideas, suppose we have a model defined on $\Lambda_{\epsilon,L}=L\T^d\cap \epsilon\Z^d$ for $d\geq 2$, 
where either $\epsilon=1$ is fixed and $L\to \infty$ (statistical mechanics model)
or $\epsilon \to 0$ is a small regularisation parameter and $L$ is fixed or $L\to\infty$ (continuum field theory model in finite or infinite volume). 
From the discussion in Sections~\ref{sec: Difficulties arising from statistical physics perspective}--\ref{sec: Difficulties arising from continuum perspective} recall that the speed of convergence of an associated dynamics is often related to the presence of phase transitions in equilibrium.
These phase transitions are phenomena arising in the limit of large volumes, i.e., large $L$. 
In this limit, one typically expects that the log-Sobolev constant should be bounded from below independently of $\Lambda_{\epsilon,L}$ as long as no phase transition occurs. 
On the other hand, the additional regularisation parameter $\epsilon>0$ is not expected to affect the dynamics, i.e.,
the log-Sobolev constant should also be bounded from below as $\epsilon \to 0$.

Sharpness of the criterion of Theorem~\ref{thm:LSI-mon} is therefore evaluated through the following questions:
\begin{enumerate}
	\item In absence of a phase transition, does the criterion provide a lower bound on the log-Sobolev constant uniform in the volume (i.e., in $L$)?
	\item Does it give a bound on the log-Sobolev constant independent of the regularisation parameter $\epsilon$ for continuum models?
	\item If the first point holds, can one correctly estimate how the log-Sobolev constant vanishes as a function of the distance to the critical point, or at the critical point as a function of $L$?
\end{enumerate}
The third point is considerably more involved than the first two. 
To provide a guideline to read the next sections, we collect here the answers to the above three questions   obtained by studying the examples presented below. 
\begin{itemize}
\item For statistical mechanics models ($\epsilon=1$) at high temperature, i.e., far away from the critical point,
the multiscale Bakry-\'Emery criterion implies point (i) very generally, see, e.g.,
Theorem~\ref{thm: high_temp_Ising} for the specific case of the Ising model (which extends similarly to a much broader class of models).
This regime is also covered by many other criteria (see, e.g., the monographs~\cite{MR2352327,MR1971582} and~\cite{MR1715549}), 
with the notable exception of mean-field spin-glass models, see the discussion in~\cite{MR3926125},
for which the spectral nature of the criterion is important.

\item The criterion~\eqref{eq_lower_bound_Hess_V_t_sec5} can be sharp enough to reach the critical point, 
  see, e.g., Theorem~\ref{thm:ising} and Example~\ref{ex_lattice_phi4_mf}
  below for the Ising and $\varphi^4$ models.
  In other words, there are models for which the criterion implies (i) up to the phase transition. 
  We expect that the criterion should imply (i) up to the critical point for a large class of models.
  
\item The criterion can provide a bound on the log-Sobolev constant that is uniform as $\epsilon \to 0$ 
  (see Theorems~\ref{thm:sinegordon} and \ref{thm_continuous_phi4} for the $\varphi^4$ and sine-Gordon models), thus satisfying point (ii).
\end{itemize}
Point (iii) is in general open and in this generality hopelessly difficult, but some positive results exist.
The simplest models are ones with quadratic mean-field interaction such as the Curie-Weiss model,
in which case one can answer (iii) in the affirmative. 
Using this perspective, a detailed analysis above and below the critical temperature was also carried out for mean-field $O(n)$ models in~\cite{zbMATH07286868}. 
For certain more general continuum particle systems with mean-field interaction,
the behaviour close to the critical point is the subject of ongoing work~\cite{meanfield}.
For more models with more complicated spatial structure
in which computations can still be carried out, such as the hierarchical $\varphi^4_4$ model~\cite{MR4061408}, 
the criterion provides almost matching upper and lower bounds on how fast the log-Sobolev constant vanishes at the critical point.
For the nearest-neighbour Ising model in $d\geq 5$, the criterion implies polynomial bounds on the log-Sobolev constant at and near the critical temperature
\cite{MR4705299}.

\subsection{Applications to Euclidean field theory}
\label{sec: Applications to Euclidean field theory}

In this section, $\Lambda =\Lambda_{\epsilon,L}$ will be a discrete torus of mesh size $\epsilon$ and side length $L$ (assumed to be a multiple of $\epsilon$),
i.e., $\Lambda_{\epsilon,L} = L\T^d\cap \epsilon\Z^d$, and the discrete Laplacian on $\Lambda_{\epsilon,L}$ is given by
\begin{equation}
\forall \varphi\in\R^{\Lambda_{\epsilon,L}}:
\qquad
(\Delta^\epsilon\varphi)_x =\epsilon^{-2}\sum_{y\sim x}(\varphi_y-\varphi_x).\label{eq_def_continuous_Laplacian}
\end{equation}
The (lattice regularised) Euclidean field theory models we consider are of the form
\begin{equation} 
  \nu^{\epsilon,L}(d\varphi) \propto
  \exp\bigg[-\frac{\epsilon^d}{2}\sum_{x\in\Lambda_{\epsilon,L}}  \varphi_x(-\Delta^\epsilon\varphi)_x 
  -\,
  \epsilon^d\sum_{x\in \Lambda_{\epsilon,L}} V^\epsilon(\varphi_x)
  \bigg]\, d\varphi,
  \label{eq_lattice_EQFT}
\end{equation}
where $d\varphi$ denotes the Lebesgue measure on $\R^{\Lambda_{\epsilon,L}}$ and 
the single-site potential $V^\epsilon$ is a real-valued function chosen in such a way that
the $\epsilon\rightarrow 0$ limit of the measure
exists (in a suitable space of generalised functions) and is non-Gaussian. 
Writing $\nabla V^\epsilon(\varphi) = ((V^\epsilon)'(\varphi_x))_x$ for $\varphi\in\R^{\Lambda_{\epsilon,L}}$, 
the dynamics is the 
(lattice regularised) SPDE
\begin{equation} \label{e:SDE}
  d\varphi_t = \Delta^\epsilon \varphi_t\, dt - \nabla V^\epsilon(\varphi_t)\, dt + \sqrt{2}
  dB_t^{\epsilon,L}
\end{equation}
where $dB^{\epsilon,L}$ is space-time white noise on $\R_+ \times \Lambda_{\epsilon,L}$, i.e., the 
$t\mapsto B^{\epsilon,L}_{t,x}$ are independent Brownian motions with variance $\epsilon^{-d}$ for $x\in \Lambda_{\epsilon,L}$, or equivalently a standard Brownian motion with respect to the continuum inner product $(u,v)_\epsilon = \epsilon^d \sum_{x\in\Lambda_{\epsilon,L}} u_xv_x$. 
The $\epsilon\to 0$ limit of~\eqref{e:SDE} is a singular SPDE.
The pathwise (short time) limit theory for such SPDEs is the subject of Hairer's regularity
structure theory
\cite{MR3785597,MR3274562,MR3468250},
the paracontrolled method of Gubinelli et al.~\cite{MR3406823,MR3846835,MR3758734},
and the pathwise renormalisation group approach \cite{MR3459120,duch2022flow}.  
The log-Sobolev inequality for the associated dynamics takes the form:
\begin{equation}
\ent_{\nu^{\epsilon,L}}(F)
\leq 
\frac{2}{\gamma} D_{\nu^{\epsilon,L}}(\sqrt{F}),
\end{equation}
with the standard Dirichlet form with respect to the gradient $\nabla^\epsilon$ corresponding to $(\cdot,\cdot)_{\epsilon}$:
\begin{equation}
D_{\nu^{\epsilon,L}}(F)
= \E_{\nu^{\epsilon,L}}\qB{(\nabla^\epsilon F,\nabla^\epsilon F)_\epsilon}
=
\frac{1}{\epsilon^d}\sum_{x\in\Lambda_{\epsilon,L}} \E_{\nu^{\epsilon,L}}\Bigg[\bigg|\ddp{F}{\varphi_x}\bigg|^2\Bigg]
,
\end{equation}
i.e., $(\nabla^\epsilon F)_x = (\nabla_\varphi^\epsilon F)_x = \epsilon^{-d} \ddp{F}{\varphi_x}$. 
(Thus this gradient acts on functionals of fields $F: \R^{\Lambda_{\epsilon,L}} \to \R$
while the Laplacian \eqref{eq_def_continuous_Laplacian} acts on fields $\varphi \in \R^{\Lambda_{\epsilon,L}}$.)
\bigskip

We now discuss two prototypical models.

\paragraph{Continuum sine-Gordon model}

Let $d=2$.
For $0<\beta<8\pi$ and $z\in \R$,
the sine-Gordon model is defined by the single-site potential
\begin{equation} 
\forall \varphi \in \R:  \qquad 
  V^\epsilon(\varphi) = 2z \epsilon^{-\beta/4\pi} \cos(\sqrt{\beta}\varphi).
\end{equation}
One can also add a convex quadratic part to the measure: the massive sine-Gordon model with mass $m>0$ corresponds to the single-site potential $V^\epsilon(\varphi) +\frac12 m^2 \varphi^2$.

\paragraph{Continuum $\varphi^4$ model}

Let $d=2$ or $d=3$.
For $g>0$ and $r\in \R$,
the $\varphi^4_d$ measure is defined by
\begin{equation} \label{e:Veps-phi4}
\forall \varphi \in \R:  \qquad 
  V^\epsilon(\varphi) = \frac{g}{4}\varphi^4 + \frac{r +a^\epsilon(g)}{2} \varphi^2 ,
\end{equation}
where $a^\epsilon(g)$ is a divergent counterterm.
The notation  $g,r$ of the parameters introduced in \eqref{eq: Ginzburg--Landau potential} are also used for the continuum setting.
Explicitly, for an arbitrary fixed $m^2 > 0$, 
one can take $a^\epsilon(g)=a^\epsilon(g,m^2)$ with
\begin{equation} \label{e:counterterm1}
  a^\epsilon(g,m^2) := -3g \big(-\Delta^\epsilon+m^2)^{-1}(0,0) + 6g^2 \big\|\big(-\Delta^\epsilon+m^2\big)^{-1}(0,\cdot)\big\|^3_{L^3(\Lambda_{\epsilon,L})}
  ,
\end{equation}
and the notation $\|f\|^p_{L^p(\Lambda_{\epsilon,L})} = \epsilon^d \sum_{x\in\Lambda_{\epsilon,L}} |f(x)|^p$ for $p>0$. 
The counterterms defined in terms of different $m^2$ differ by additive constants and thus the choice of $m^2$ corresponds
to a normalisation. In the following we take $m^2=1$ for the definition of the counterterm.

\medskip

The sine-Gordon and $\varphi^4_d$ models are defined on the discretised torus $\Lambda_{\epsilon,L} = L\T^d\cap \epsilon\Z^d$. 
As explained in Section~\ref{sec: Difficulties arising from continuum perspective}, 
these models should be thought of as discretised versions of limiting models defined on the continuum torus $L\T^d$. 
Recall that the criterion of Theorem~\ref{thm:LSI-mon} asks for a lower bound on the Hessian of the renormalised potential uniformly in the field:
for some $\dot\lambda_t\in\R$,
\begin{equation}
\forall \varphi \in\R^{\Lambda_{\epsilon,L}}, t\geq 0:\qquad
\dot C_t \He V_t(\varphi)\dot C_t - \frac{1}{2}\ddot C_t
\geq 
\dot\lambda_t\dot C_t
.
\end{equation}
To obtain information on the dynamics in the $\epsilon\to 0$ limit, 
one is therefore interested in estimates of the renormalised potential that are uniform in $\epsilon,\varphi$.

For the sine-Gordon model, this was carried out in~\cite{MR4303014} by providing an explicit description of the renormalised potential at each scale
following \cite{MR914427}. By writing the renormalised potential as the Fourier type series
\begin{equation} \label{e:Vt-fourier-SG}
  V_t(\varphi) = \sum_{n=0}^\infty \frac{1}{n!} \int_{(\Lambda_{\epsilon,L} \times \{\pm 1\})^n} \tilde V_t(\xi_1,\dots, \xi_n) e^{i\sum_{i=1}^n \sigma_i\varphi_{x_i}} \, d\xi_1 \cdots d\xi_n,
\end{equation}
where $\xi_i= (x_i,\sigma_i) \in \Lambda_{\epsilon,L} \times \{\pm 1\}$ and we used the notation
\begin{equation}
  \int_{\Lambda_{\epsilon,L} \times \{\pm 1\}} F(\xi) d \xi = \epsilon^d \sum_{x\in\Lambda_{\epsilon,L}} \sum_{\sigma \in \{\pm 1\}} F(x,\sigma),
\end{equation}
the Polchinski equation for $V_t$ reduces to a triangular system of ODEs for the Fourier coefficients $\tilde V_t(\xi_1,\dots, \xi_n)$.
For $\beta<6\pi$ one can obtain the following control  on $\He V_t$ by estimating these Fourier coefficients.
It remains an interesting problem to extend such estimates to the optimal regime $\beta<8\pi$.
(In this regime $\beta<8\pi$, weaker estimates that are sufficient for the construction of the limiting continuum measure
are known \cite{MR1777310,MR849210},
see also the discussion in \cite{2010.07096}. These estimates are however insufficient for our purposes.)

\begin{theorem}\label{thm:sinegordon}
  For the massive continuum sine-Gordon model with mass $m>0$, let $A=-\Delta^\epsilon+m^2\id$ and $\dot C_t = e^{-tA}$ ($t\geq 0$). 
  Then, 
  if $\beta<6\pi$, 
  there is a constant $\mu^* = \mu^*(\beta,z,m,L)>0$ that does not depend on $\epsilon,t$, 
  such that with
  \begin{equation}
    V_0(\varphi) = \epsilon^2 
      \sum_{x\in\Lambda_{\epsilon,L}} 2 \epsilon^{-\frac{\beta}{4\pi}}z \cos(\sqrt{\beta}\varphi_x)
  \end{equation}
  the renormalised potential satisfies
  \begin{equation}
    \forall \varphi \in \R^{\Lambda_{\epsilon,L}}, t\geq 0:\qquad 
  \dot C_t  \He V_t(\varphi)\dot C_t  \geq 
     \dot \mu_t \dot C_t
    \quad \text{with} \quad \sup_{t\geq 0}\Big| \int_0^t \dot \mu_s  ds \Big| \leq \mu^*
\label{eq: condition sine Gordon}.
  \end{equation}
Given $\beta,z,m,L$, this yields a lower bound on the log-Sobolev constant $\inf_\epsilon \gamma^{\epsilon,L}(\beta,z,m)>0$.

  Moreover,
  the log-Sobolev constant is uniform in the large-scale parameter $L$ under the following condition. 
  If $L$ satisfies $m\geq 1/L$ and the coupling constant $z$ is such that $|z|\leq \delta_\beta m^{2+\beta/4\pi}$ for a small enough $\delta_\beta>0$, 
  then $\inf_{\epsilon} \gamma^{\epsilon,L}(\beta,z,m)>m^2 - O_\beta \big( m^{\beta/4 \pi} |z| \big)$, uniformly in $L\geq 1/m$.
\end{theorem}

For the sine-Gordon model, 
the multiscale Bakry-\'Emery criterion of Theorem~\ref{thm:LSI-mon} is thus seen to provide the optimal independence of $\epsilon$ of the log-Sobolev constant in finite volume, 
as well as of $L$ under an additional small coupling assumption depending on the external mass.
Under the condition \eqref{eq: condition sine Gordon} derived in Theorem \ref{thm:sinegordon}, 
Shenfeld~\cite{Shenfeld2022ExactRG} also used the multiscale Bakry-\'Emery criterion to construct a Lipschitz transport map between
the continuum sine-Gordon model and the free field with the same mass (i.e., the model~\eqref{eq_lattice_EQFT} with $V^\epsilon(\varphi)=\frac12 m^2\varphi^2$),
see Example~\ref{ex_sine_gordon_transport}.
\bigskip

The log-Sobolev inequality for the $\varphi^4_d$ model (where $d=2,3$) was obtained in~\cite{MR4720217}.
In this case, a sufficiently strong description of the renormalised potential is difficult to obtain directly.
To provide some context, we remark that
Polchinski's original article \cite{polchinski1984269} assumes a representation of $V_t$ as a \emph{formal} power series
analogous to \eqref{e:Vt-fourier-SG},
\begin{equation}
  V_t(\varphi) = \sum_{m=0}^\infty \frac{1}{(2m)!} \int_{(\Lambda_{\epsilon,L})^{2m}}  \tilde V_t(x_1,\dots, x_{2m}) \varphi_{x_1}\cdots\varphi_{x_{2m}} \, dx_1\cdots dx_{2m}.
\end{equation}
In this representation, the Polchinski equation again formally reduces to a system of ODEs for the kernels $\tilde V_t$,
and this representation is very successful for understanding the renormalised potential as a formal power series.
However,  
the series representation in powers of $\varphi$ is necessarily divergent as it would imply analyticity of the solution
around $0$ as a function of the coupling constant $g$ in front of the $\varphi^4$ term.
It therefore seems difficult to use this naive expansion in powers of $\varphi$ to obtain nonperturbative analytic information about  $V_t$,
as is needed for the construction of the measure and even more so to prove the log-Sobolev inequality.

Nonetheless, the renormalised potential  itself is well defined,
and it is immediate from the definition \eqref{e:V-def}
that its derivatives correspond to correlation functions of the fluctuation measure $\mu_t^{\varphi}$
defined in \eqref{e:fluctuationmeasure-alt}.
In particular, by Lemma~\ref{lem:fluctuationmeasure},
\begin{equation}
  \He V_t(\varphi) = C_t^{-1} - C_t^{-1}\cov(\mu_t^\varphi) C_t^{-1},
\end{equation}
where $\cov(\mu_t^\varphi)$ is the covariance matrix of $\mu_t^\varphi$.
Thus understanding the Hessian of the renormalised potential is equivalent to understanding the two-point function of the fluctuation measure.
To analyse this two-point function, the following choice of covariance decomposition (known as Pauli--Villars regularisation) is very helpful:
\begin{equation}
  C_t= (A+1/t)^{-1} 
  .
\end{equation}
Indeed, with this choice, the fluctuation measure $\mu^\varphi_t$ has the same structure as the original measure, except for  an additional
mass term $\frac{1}{2t} \zeta^2$ and an external field term $(\zeta,C_t^{-1}\varphi)$.
Therefore various correlation inequalities available for the original $\varphi^4$ measure also apply to the fluctuation measure.
These include the FKG inequality (for all $\varphi$) which shows that the covariance matrix has positive entries,
the Griffiths inequalities (when $\varphi=0$) which shows that the two-point function is monotone in $t$ and $\mu$, and 
the recent inequality of Ding--Song--Sun (DSS)~\cite{MR4586225} which shows that the covariance matrix is entrywise maximised at $\varphi=0$.
In this setting, with $\dot C_t = (tA+1)^{-2}$ and $-\frac12 \ddot C_t = A(tA+1)^{-3}$,
the left-hand side of the multiscale Bakry--\'Emery criterion becomes:
\begin{equation}
  \dot C_t \He V_t(\varphi) \dot C_t- \frac12 \ddot C_t = \dot C_t^{1/2}\Big(\frac{1}{t} - \frac{\cov(\mu_t^\varphi)}{t^2}\Big) \dot C_t^{1/2},
\end{equation}
and the combination of the FKG and DSS inequalities with the Perron--Frobenius theorem imply that the right-hand
side is lower bounded as a quadratic form by the following term involving only $\varphi=0$
(where $\|M\|$ denotes the operator norm of a matrix $M$):
\begin{equation}
\Big(\frac{1}{t}- \frac{\|\cov(\mu_t^{\varphi=0})\|}{t^2}\Big)\dot C_t
,
\end{equation}
see  \cite{MR4720217} for details.

This representation is the starting point for the proof of the following theorem, proved in \cite{MR4720217}.
For concreteness, we choose the counterterms \eqref{e:counterterm1} in the definition of the $\varphi^4_d$ model with $m=1$
and denote the 0-field susceptibility of the $\varphi^4_d$ model by
\begin{equation}
  \chi^{\epsilon,L}(g,r)
  :=
  \epsilon^d \sum_{x\in\Lambda_{\epsilon,L}}\big<\varphi_0\varphi_x\big>^{\epsilon,L}_{g,r}
  ,
\end{equation}
where $\big<\cdot\big>^{\epsilon,L}_{g,r}$ denotes the expectation of the $\varphi^4_d$ measure.
The second Griffiths inequality implies that $\chi^{\epsilon,L}(g,r)$ is monotone in $r$.
We then define the critical point by
\begin{equation}
  r_c (g) =  \inf\Big\{ r \in \R: \chi^{\epsilon,L}(g,r) \text{ is uniformly bounded as $\epsilon \to 0$ and $L \to \infty$} \Big\}.
\end{equation}
Choosing a linear test function,
the Poincar\'e constant of the $\varphi^4_d$ model (and thus the log-Sobolev constant) is not uniformly bounded away from $0$ when $r< r_c(g)$.
The following theorem shows that it is bounded when $r > r_c(g)$.

\begin{theorem}\label{thm_continuous_phi4}
  For the continuum $\varphi^4$ model in $d=2,3$, let
  $A=-\Delta^\epsilon+\id$ and
  $\dot C_t = (tA+\id)^{-2}$ ($t\geq 0$)  and take $m=1$ in the definition~\eqref{e:counterterm1} of the counterterm.
  Then with
  \begin{equation}
    V_0(\varphi) = \epsilon^d\sum_{x\in\Lambda_{\epsilon,L}} (\frac{g}{4}\varphi^4_x + \frac{r-1+a^\epsilon(g)}{2}\varphi_x^2),
    \label{eq_def_V_0_phi4}
  \end{equation}
  the renormalised potential satisfies
  \begin{equation}
        \forall \varphi \in \R^{\Lambda_{\epsilon,L}}, t\geq 0:\qquad 
    \dot C_t \He V_t(\varphi)\dot C_t - \frac12 \ddot C_t \geq 
    \Big(\underbrace{\frac{1}{t}-\frac{\chi(g, r+1/t)}{t^2}}_{\dot \lambda_t}\Big) \dot C_t
    ,
    \label{eq_hessian_bound_phi4}
  \end{equation}
  and $\lambda_t \geq \log(1+t)-C(g,r,L)$ uniformly in $\epsilon>0$ for any $g>0$, $r \in \R$ and any fixed $L$,
  and $C(g,r,L)$ is independent of $L$ if $r> r_c(g)$.
  In particular, the following integral is bounded under the same conditions:
  \begin{equation}
    \int_0^\infty e^{-2 \lambda_t }\, dt, \qquad \lambda_t =\int_0^t \dot \lambda_s \, ds
    .
  \end{equation}
  The log-Sobolev constant $\gamma^{\epsilon,L}(g,r)$ thus satisfies 
  $\inf_{\epsilon,L}\gamma^{\epsilon,L}(g,r)>0$ for any $g>0,r \in\R$ for which the measure does not have a phase transition ($r>r_c(g)$),
  and $\inf_\epsilon \gamma^{\epsilon,L}(g,r)>0$ for any $g>0,r\in\R$ and $L$ fixed.
\end{theorem}

In the $\varphi^4_d$ case, 
the multiscale criterion gives a log-Sobolev constant bounded uniformly in the volume and $\epsilon$ in the entire single-phase region. 
However, the bound obtained on the log-Sobolev constant has accurate dependence on $g, r$ only far away from the transition, 
corresponding to values of $g,r$ such that $r >0$ and $g \ll r$.
In $d>4$ where the $\varphi^4$ model does not have a continuum limit,
one can obtain a polynomial bound near the critical point on the log-Sobolev constant for the lattice $\varphi^4_d$ model.
The same applies to the Ising model.
The next sections detail these proofs of the log-Sobolev inequality for the lattice $\varphi^4$ and Ising models.

\subsection{Applications to lattice $\varphi^4$ models}

The lattice $\varphi^4$ model corresponds to the case $\epsilon=1$ in~\eqref{eq_def_V_0_phi4},
and we denote its single-site potential for $g>0,r\in\R$ by:
\begin{equation}
V(\varphi) 
=
\frac{g}{4}\varphi^4 + \frac{r}{2}\varphi^2,
\end{equation}
and choose the coupling matrix as $A=-\Delta$ where $\Delta$ is the lattice Laplace operator corresponding to $\epsilon=1$
in \eqref{eq_def_continuous_Laplacian}.
Boundary conditions do not matter much, but for concreteness we consider periodic boundary conditions, i.e.,
the state space is $\Lambda_L = \Z^d/(L\Z)^d$ for any $d\geq 1$.
The associated expectation is denoted by $\big<\cdot\big>_{g,r}$ and the finite volume susceptibility is defined by
\begin{equation}
  \chi^L(g,r)
= 
\sum_{x\in \Lambda_L}\big<\varphi_0\varphi_x\big>_{g,r}
.
\end{equation}
The critical point $r_c(g)$ is again defined as the infimum over all $r \in \R$ such that $\chi^L(g,r)$ is bounded from above uniformly in $L$.
It is the case that $r_c(g)<0$.
As a special case, Theorem~\ref{thm_continuous_phi4} yields a uniform lower bound on the log-Sobolev constant for all $r>r_c(g)$,
but the proof of this statement can be simplified considerably in this case as we now outline.

\begin{example}[\!\!{\cite[Example 3.1]{MR4720217}}]
  \label{ex_lattice_phi4}
For the lattice $\varphi^4$ model on $\Lambda_L = \Z^d/(L\Z)^d$ in any dimension,
let $\dot C_t = (tA+\id)^{-2}$. 
Then analogous to \eqref{eq_hessian_bound_phi4}:
\begin{equation}
  \forall \varphi \in \R^{\Lambda_{\epsilon,L}}, t\geq 0:\qquad 
  \dot C_t \He V_t(\varphi)\dot C_t - \frac12 \ddot C_t \geq 
  \Big(\underbrace{\frac{1}{t}-\frac{\chi(g,r+1/t)}{t^2}}_{\dot \lambda_t}\Big) \dot C_t
  ,
  \label{eq_hessian_bound_phi4_lattice}
\end{equation}
and the following integral (and thus the inverse log-Sobolev constant) is bounded for all $r>r_c(g)$:
\begin{equation}   \label{eq_hessian_bound_phi4_lattice_integral}
  \int_0^\infty e^{-2 \lambda_t} \, dt.
\end{equation}
\end{example}

\begin{example} \label{ex_lattice_phi4_mf}
Assume that the susceptility satisfies the mean-field bound:
for some $D>1/2,\delta\in [0,1]$:
\begin{equation} \label{e:lattice_phi4_mfb}
 \chi^L (g,r+1/t) \leq \frac{D}{\delta+1/t}.
\end{equation}
This bound holds, in particular, for the near-critical $\varphi^4$ model on $\Z^d$ when $d \geq 5$, for some $D>1/2$
and $\delta=L^{-d} + r-r_c(g)$ ($r\geq r_c(g)$), see \cite{MR678000}.
Then
\begin{equation}
  \int_0^\infty e^{-2 \lambda_t} \, dt \leq C(g,D) \delta^{-2D+1}
  \label{eq: LS constant delta}
\end{equation}
is bounded polynomially in $1/\delta$.
\end{example}

\begin{proof}[Sketch of Examples~\ref{ex_lattice_phi4}--\ref{ex_lattice_phi4_mf}]
For $t_0 >0$ small enough such that $r+ 1/t_0 >0$, the measure with coupling constants $(g,r+1/t_0)$ is log-concave and
the Brascamp--Lieb inequality \eqref{e:BL} implies
\begin{equation}
\forall s \leq t_0, \quad 
\chi^L (g,r+1/s) \leq \frac{1}{r+1/s}  
\end{equation}
and hence
\begin{equation} \label{e:phi4-ellt-1}
  \lambda_{t_0}  = \int_0^{t_0} \pa{\frac{1}{s} - \frac{\chi^L (g,r+1/s)}{s^2}} \, ds
  \geq \int_0^{t_0} \frac{r}{1+rs} \, ds
  = \log(1+r {t_0}).
\end{equation}
On the other hand, for any $t >0$, the second Griffiths inequality implies $\chi^L (g,r+1/t) \leq \chi^L (g,r)$ and thus
\begin{equation} \label{e:phi4-ellt-2}
  \lambda_{t} - \lambda_{t_0} \geq \int_{t_0}^t \pa{\frac{1}{s} - \frac{\chi^L (g,r)}{s^2}} \, ds 
  \geq 
  \log(\frac{t}{t_0}) - \frac{\chi^L (g,r)}{t_0}.
\end{equation}
Combining both bounds with $t_0$ such that $r+1/t_0>0$, it follows that the integral
\eqref{eq_hessian_bound_phi4_lattice_integral} (and thus the log-Sobolev constant) is bounded below whenever $\chi^L(g,r)$ is bounded above.

Now assume that the susceptibility satisfies the mean-field bound \eqref{e:lattice_phi4_mfb}.
Then
\begin{equation} \label{e:phi4-ellt-3}
  \lambda_t - \lambda_{t_0}
  \geq \int_{t_0}^t \pa{\frac{1}{s} - \frac{D}{\delta s^2 + s}} \, ds
  \geq C(D,t_0) + \log t - D\log(\frac{t}{1+\delta t})
  .
\end{equation}
Combining this bound with $\lambda_{t_0} \geq \log(1+r t_0)$ for $t_0\leq 1/|r|$ gives
\begin{equation}
  \int_0^\infty e^{-2 \lambda_t } \, dt   \leq C(g,D) \delta^{1-2D}
\end{equation}
as claimed.
\end{proof}

\subsection{Applications to transport maps}
\label{sec:appl-transport}

Another application of the bounds on the Hessian of renormalised potential are  to the transport maps of
Theorems~\ref{thme: Lipschitz}--~\ref{thme: Lipschitz2}, which can be used to recover the log-Sobolev inequalities.

\begin{example}\label{ex_sine_gordon_transport}
  For the continuum sine-Gordon model (under the same assumptions as in Theorem~\ref{thm:sinegordon}),
  the transport map of Theorem~\ref{thme: Lipschitz} has Lipschitz constant bounded uniformly in $t$:
    \begin{equation}
    \forall \varphi\in X, f\in X: \qquad |\nabla \hat S_t(\varphi) f| \leq C(\beta,z,L)|f|.
  \end{equation}
  The constant $C(\beta,z,L)$ is independent of $L$ under the same assumptions as in Theorem~\ref{thm:sinegordon}.
\end{example}

\begin{proof}[Sketch]
  Using bounds on the Hessian of the renormalised potential from Theorem~\ref{thm:sinegordon},
  this is a direct consequence of Theorem~\ref{thme: Lipschitz}.
\end{proof}

\begin{example}\label{ex_continuum_phi4_transport}
  For the continuum $\varphi^4$ model (under the same assumptions as in Theorem~\ref{thm_continuous_phi4}) 
  where $A = -\Delta^\epsilon +\id$,
  the gradient of the transport map of Theorem~\ref{thme: Lipschitz2} is uniformly bounded from $H^1$ to $L^2$ in the following sense:
  \begin{equation} \label{eq_continuum_phi4_transport}
    \forall \varphi\in X, f\in X: \qquad |\nabla \hat S_t(\varphi) f| \leq C(g,r,L)|f|_{A}
  \end{equation}
  where $|f|_A = |\sqrt{A}f|$ is the (discrete) Sobolev norm and $C(g,r,L)$ is independent of $L$ if $r>r_c(g)$.
\end{example}

In particular, in the application \eqref{eq_LSI_transport},
the reference Gaussian measure $\hat \nu_\infty$ satisfies the log-Sobolev inequality with quadratic form $Q=A^{-1}$ and
the bound on the transport map \eqref{eq_continuum_phi4_transport} is exactly the required assumption
\eqref{eq: Lipschitz bound-bis} to recover the log-Sobolev inequality for the continuum $\varphi^4$ measure.

\begin{proof}[Sketch]
  Theorem~\ref{thme: Lipschitz2} gives the following bound on the gradient of the transport map:
  \begin{equation}
    |\nabla \hat S_t(\varphi) f|^2 \leq e^{-\lambda_t}|f|_{1+tA}^2.
  \end{equation}
  By Theorem~\ref{thm_continuous_phi4}), $\lambda_t \geq \log(1+t)-C(g,r)$ and therefore
  \begin{equation}
    |\nabla \hat S_t(\varphi) f|^2 \leq \frac{C(g,r)}{1+t}|f|_{1+tA}^2 \leq C(g,r) |f|_A^2
  \end{equation}
  where we used that $(1+tA)/(1+t) \leq A$ if $A \geq \id$.
\end{proof}

\begin{example} \label{ex_lattice_phi4_transport}
For the lattice $\varphi^4$ model, for any $r>r_c(g)$,
the transport map of Theorem~\ref{thme: Lipschitz2} has Lipschitz constant  bounded uniformly in $L$ and $t$.
Moreover, if the mean-field bound  \eqref{e:lattice_phi4_mfb} holds then the Lipschitz constant of the transport map is of order $\delta^{-D/2}$:
  \begin{equation}
    \forall \varphi\in X, f\in X: \qquad |\nabla \hat S_t(\varphi) f| \leq C(g,D)\delta^{-D/2}|f|.
  \end{equation}
\end{example}

\begin{proof}[Sketch]
Using that the operator norm $\|A\|$ is bounded,
Theorem~\ref{thme: Lipschitz2} gives the following bound on the Lipschitz constant of the transport map:
\begin{equation}
\sqrt{t\|A\| +1}\, e^{- \frac12 \lambda_t} 
\leq C \sqrt{1+t}\, e^{- \frac12 \lambda_t}
.
\label{eq: Lipschitz constant}
\end{equation}
On the other hand, by \eqref{e:phi4-ellt-1} and \eqref{e:phi4-ellt-2}, for $r>r_c(g)$, we have
\begin{equation}
  e^{- \frac12 \lambda_t} \leq C(g, r) \frac{1}{\sqrt{1+t}}.
\end{equation}
Combining both bounds gives the uniform bound on the Lipschitz constant for any $r>r_c(g)$.

Substituting the bound on the Hessian obtained from the mean-field bound~\eqref{e:lattice_phi4_mfb} on the susceptibility into the above bound on the Lipschitz constant of the transport map
again yields a polynomial bound in $\delta$. 
Indeed, by \eqref{e:phi4-ellt-3}, the bound obtained is of order
\begin{equation}
  \sqrt{t\|A\| +1}\, e^{- \frac12 \lambda_t}
  \leq C \sqrt{1+t}\, e^{- \frac12 \lambda_t}
  \leq C(g,D) \pa{\frac{t}{1+\delta t}}^{D/2}
  \leq C(g,D) \delta^{-D/2}
\end{equation}
which is the claimed bound.
\end{proof}

\subsection{Applications to Ising models}

In this section, we explain how to apply the ideas developed in Section~\ref{sec_gaussian_integration_and_polchinski} to Ising models with discrete spins.
Using the Bakry--\'Emery criterion and its multiscale version, these ideas were developed in \cite{MR3926125,MR4705299},
while closely related results were obtained using spectral and entropic independence and stability estimates in \cite{MR4408509,2106.04105,2203.04163}.
Similar ideas apply to $O(n)$ models (see \cite{MR3926125}) 
for which it is however not known in general that the critical point can be reached due to the lack of appropriate correlation inequalities.
For concreteness and because the results are most complete in this case, we focus on the situation of Ising models.

\subsubsection{Renormalised potential}
\label{subsubsec: Renormalised potential}

The Ising model with coupling matrix $A$ at inverse temperature $\beta>0$ and (site-dependent) external field $h$ on a finite set $\Lambda$ is defined by:
\begin{equation}
  \E_{\mu}[F] = \E_{\mu_{\beta,h}}[F] \propto \sum_{\sigma \in \{\pm 1\}^\Lambda} e^{-\frac12 (\sigma, \beta A\sigma) + (h,\sigma)} F(\sigma).
  \label{eq: Ising model-bis}
\end{equation}
Since $\sigma_x^2=1$ for each $x$, 
the measure is invariant under the change $A\to A+\alpha\id$, $\alpha\in\R$. 
Therefore, without loss of generality, 
we can assume that the coupling matrix $A$ is positive definite. 
We also assume that it has spectral radius bounded by $1$; this just amounts to a choice of normalisation for $\beta$. 
It is helpful to think of the Ising model as a denerate case  of the $\varphi^4$ measure in which the Gaussian part has covariance $(\beta A)^{-1}$
and the potential is singular to enforce the values $\{\pm 1\}$, i.e., $\int_{\R^\Lambda} (\cdot) e^{-V_0(\varphi)} \, d\varphi$ is replaced
by $\sum_{\sigma \in \{\pm 1\}^\Lambda} (\cdot)$. It turns out that the Polchinski equation still makes sense and that the renormalised
potential indeed becomes smooth immediately.
A natural covariance decomposition of $(\beta A)^{-1}$ is
\begin{equation}
\label{eq: inverse covariance}
\forall 0 \leq t \leq \beta: \qquad
 C_t= (tA + (\alpha -t)\id)^{-1}, 
\end{equation}
for a parameter $\alpha>\beta$ which will be unimportant. 
For $t<\alpha$ the matrix $C_t$ is positive definite and for $\beta<\alpha$, 
as explained above, the Ising model at inverse temperature $\beta$ can be written as
\begin{equation}
  \E_{\mu}[F] \propto \sum_{\sigma \in \{\pm 1\}^\Lambda} e^{-\frac12 (\sigma, C_\beta^{-1}\sigma) + (h,\sigma)} F(\sigma).
\end{equation}
Strictly speaking, $C_t$ is not a covariance decomposition since $C_0=\alpha^{-1}\id \neq 0$, 
different from the assumption in Section~\ref{sec_gaussian_integration_and_polchinski}.  
However, all results from that section can be applied to the covariance decomposition $C_t-C_0$, 
and we will do this without further emphasis.
The renormalised potential can be defined analogously to the continuous setting as:
\begin{equation}
  V_t(\varphi) = -\log \sum_{\sigma\in \{\pm 1\}^\Lambda} e^{-\frac12(\sigma-\varphi, C_t^{-1} (\sigma-\varphi)) + (h,\sigma)}.
  \label{eq_def_V_ising}
\end{equation}
This leads to a decomposition of the Ising measure~\eqref{eq: Ising model-bis} as:
\begin{equation}
\forall 0\leq t<\beta: \qquad 
\E_\mu[F]
=
\E_{\nu_{t,\beta}}\big[\E_{\mu_{t}^\varphi}[F]\big].
\label{eq_decomp_Ising_measure}
\end{equation}
By analogy with \eqref{e:fluctuationmeasure-alt}, the fluctuation measure $\mu^\varphi_{t}$ used above
is the Ising measure with coupling matrix $C_{t}^{-1}$ and external field $C_t^{-1}\varphi + h$: 
\begin{equation}
\label{eq: nouvelle mesure fluctuation}
\mu^\varphi_{t} (\sigma) = \mu_{t,C_t^{-1}\varphi+h} \propto   e^{-\frac12 (\sigma, C_t^{-1}\sigma) + (C_t^{-1}\varphi +h,\sigma)},
\end{equation}
and the renormalised measure $\nu_t = \nu_{t,\beta}$ is supported on the image of $C_\beta-C_t$ in $\R^\Lambda$:
\begin{equation}   \label{eq_nut_ising}
\nu_{t,\beta}(d\varphi)
\propto
e^{-V_t(\varphi)} \, \Pg_{C_\beta-C_t}(d\varphi)
\propto
\exp\qa{-\frac{1}{2}(\varphi,(C_\beta-C_t)^{-1}\varphi) - V_t(\varphi)}\, d\varphi
.
\end{equation}
Even though $\sigma:\Lambda\to\{\pm 1\}$
is discrete, the renormalised field $\varphi: \Lambda\to\R$ is continuous as soon as $t>0$.
Convexity-based criterions, 
such as the Bakry-\'Emery or multiscale Bakry-\'Emery criterions of Theorems~\ref{thm:BE} and \ref{thm:LSI-mon}, can therefore be used to
derive log-Sobolev inequalities for $\nu_{t,\beta}$. 

\medskip
Before discussing these, we summarise results about the infinite temperature (product) Ising model, which serve as input to these arguments.

\subsubsection{Preliminaries: single-spin inequalities}\label{sec: single_spin_ineq}

At infinite temperature $\beta=0$ the spin models we consider become product measures.
By tensorisation, it thus suffices to know the log-Sobolev constant of a single spin (see Example \ref{example: tensorisation}).
The following summarises known results for these.

\paragraph{Ising model, standard Dirichlet form}

Let $\mu$ be the probability measure on $\{\pm 1\}$ with $\mu(+1)=p=1-q$.
The standard Dirichlet form is
\begin{equation}
  D_\mu(F) = \frac12 \E_\mu (F(\sigma^x)-F(\sigma))^2= \frac12 (F(+1)-F(-1))^2 .
\end{equation}

\begin{proposition}\label{prop: standard_single_spin_LSI}
  Let $\mu$ be the probability measure on $\{\pm 1\}$ with $\mu(+1)=p=1-q$. Then
  \begin{equation}
    \ent_\mu (F) \leq \frac{pq(\log p-\log q)}{p-q}(\sqrt{F(+1)}-\sqrt{F(-1)})^2.
  \end{equation}
  Thus the log-Sobolev constant with respect to the standard Dirichlet
  form on $\{\pm 1\}$ is at least $2$:
  \begin{equation}
    \ent_\mu (F) \leq D(\sqrt{F})
    = \frac{2}{\gamma_0}D(\sqrt{F}), \qquad \gamma_0=2.
  \end{equation}
\end{proposition}

\begin{proof}
  The proof can be found in \cite{MR1490046} or \cite{MR1845806}.
\end{proof}

\paragraph{Ising model, heat bath Dirichlet form}
With $\mu$ as above, the heat-bath Dirichlet form is
\begin{equation}
  D^{\rm HB}(F)
  =
  \frac12 \sum_{\sigma} \frac{\mu(\sigma)\mu(\sigma^x)}{\mu(\sigma)+\mu(\sigma^x)} (F(\sigma^x)-F(\sigma))^2
  =
  pq (F(+1)-F(-1))^2
  .
\end{equation}

\begin{proposition}  \label{prop: HB_single_spin_LSI}
    Let $\mu$ be the probability measure on $\{\pm 1\}$ with $\mu(+1)=p=1-q$. Then
  \begin{equation}
    \ent_\mu(F) \leq  pq(\log F(+1)-\log F(-1))(F(+1)-F(-1)).
  \end{equation}
  Thus the modified log-Sobolev constant with respect to the heat-bath Dirichlet form is at least $1/2$:
  \begin{equation}
    \ent_\mu F \leq D^{\rm HB}(\log F, F)
    = \frac{1}{2\gamma_0} D^{\rm HB}(\log F, F), \qquad \gamma_0 =  \frac{1}{2}
    .
  \end{equation}  
\end{proposition}

\begin{proof}
  See  \cite[Example~3.8]{MR2283379}.
\end{proof}

Similar single-spin inequalities are available for $O(n)$ models \cite{MR2854732,MR4276491} and allow to extend for example
Theorem~\ref{thm: high_temp_Ising} below for the Ising model to these models with little change \cite{MR3926125}.

\subsubsection{Entropy decomposition}
To prove a log-Sobolev inequality (or modified log-Sobolev inequality) for the Ising measure,
we start from the decomposition~\eqref{eq_decomp_Ising_measure} with $t=0$,
decomposing $\mu=\mu_\beta^0$ into two parts: an infinite temperature Ising measure $\mu^\varphi_0$ with external field $C_{0}^{-1}\varphi+h$ 
and the renormalised measure $\nu_{0,\beta}$. 
The corresponding entropy decomposition \eqref{eq: entropy decomposition} is:
\begin{equation} \label{e:ent-decomp-Ising}
\ent_{\mu}(F)
=
\E_{\nu_{0,\beta}}[\ent_{\mu_0^\varphi}(F)]
+
\ent_{\nu_{0,\beta}}\big(\E_{\mu_0^\varphi}[F]\big)
.
\end{equation}
To prove the log-Sobolev inequality (or modified log-Sobolev inequality) with respect to a Dirichlet form $D_\mu$,
we want to bound both terms by a multiple of $D_\mu(\sqrt{F})$
(or $D_\mu(F,\log F)$).
In this discussion, we focus on the log-Sobolev inequality
with respect to the standard Dirichlet form:
\begin{equation} \label{e:Ising-Dirichlet-standard}
  D_\mu(F) = \frac12 \sum_{x\in\Lambda} \E_\mu\qB{(F(\sigma)-F(\sigma^x))^2}.
\end{equation}
However, the same strategy applies with different jump rates
and a modified log-Sobolev inequality or a spectral inequality as input (and output),
see Sections~\ref{sec: single_spin_ineq} and~\ref{sec:choice-Dirichlet}.

Under the assumption that $\mu_0^\varphi$ satisfies a log-Sobolev inequality with constant $\gamma_0$
uniformly in $\varphi$, the first term on the right-hand side of \eqref{e:ent-decomp-Ising} is bounded by
\begin{equation}
  \frac{2}{\gamma_0}\E_{\nu_{0,\beta}}[D_{\mu_0^\varphi}(\sqrt{F})]=\frac{2}{\gamma_0}D_{\mu}(\sqrt{F}).
\end{equation}
Since $\mu_0^\varphi$ is a product measure, this assumption is well understood and one can take $\gamma_0=2$ for the standard Dirichlet form,
as discussed in Section~\ref{sec: single_spin_ineq}.
The last equality relies on the specific jump rates in the standard Dirichlet form \eqref{e:Ising-Dirichlet-standard}. 
For other Dirichlet forms, one can also get an inequality, see Section~\ref{sec:choice-Dirichlet}.

Bounding the second term on the right-hand side of \eqref{e:ent-decomp-Ising} essentially amounts to estimating the log-Sobolev constant of the renormalised measure $\nu_{0,\beta}$. Indeed, if $\nu_{0,\beta}$ satisfies a log-Sobolev inequality with constant $\gamma_{0,\beta}$ (and standard Dirichlet form) then
this term is bounded by
\begin{equation}
  \frac{2}{\gamma_{0,\beta}}\E_{\nu_{0,\beta}}\Big[\big|\big(\nabla_\varphi \E_{\mu^\varphi_{0}}[F]^{1/2}\big)\big|^2\Big]
  \leq
  \frac{4\alpha^2}{\gamma_0\gamma_{0,\beta}} D_\mu(\sqrt{F}),
\end{equation}
where the second inequality is elementary and follows from the next exercise (which is similar to the argument in the tensorisation proof of
Example~\ref{example: tensorisation}) and the Cauchy--Schwarz inequality.

\begin{exercise}\label{exo: covariance}
For each $x\in\Lambda$, 
let $\mu^{\varphi,x}_0$ denote the law of $\sigma_x$ under the product measure $\mu^\varphi_0$. 
Then:
\begin{equation}
\nabla_{\varphi_x} \big(\E_{\mu^\varphi_{0}}[F]^{1/2}\big)
=
\frac{\alpha}{2\sqrt{\E_{\mu^\varphi_0}[F]}}\cov_{\mu^{\varphi}_0}(F,\sigma_x)
=
\frac{\alpha}{2\sqrt{\E_{\mu^\varphi_0}[F]}}\E_{\mu_0^\varphi}\Big[\cov_{\mu^{\varphi,x}_0}(F,\sigma_x)\Big]
\end{equation}
(recall $C_0 = \alpha^{-1}\id$) and:
\begin{equation}
\label{eq: borne cov 2nd inegalite}
\cov_{\mu^{\varphi,x}_0}(F,\sigma_x)^2
\leq 
8\E_{\mu_0^{\varphi,x}}[F]\var_{\mu^{\varphi,x}_0}(\sqrt{F})
\leq 
\frac{8}{\gamma_0}\E_{\mu_0^{\varphi,x}}[F]D_{\mu^{\varphi,x}_0}(\sqrt{F})
.
\end{equation}
\end{exercise} 

\begin{proof}[Sketch]
The first equation is a simple computation.
The first inequality in \eqref{eq: borne cov 2nd inegalite} follows from Cauchy-Schwarz inequality using the following general expression for the covariance of functions $G_1,G_2$ under a measure $m$:
\begin{equation}
\cov_m(G_1,G_2)
=
\frac{1}{2}\int \big(G_1(x)-G_1(y)\big)\big(G_2(x)-G_2(y)\big)\, dm(x)\, dm(y)
.
\end{equation}
The second inequality is the general fact that the spectral gap is always larger than the log-Sobolev constant, see Proposition~\ref{prop:spectralgap}.
Details can be found in~\cite{MR3926125}.
\end{proof}

In summary, if $\mu_0^\varphi$ satisfies a log-Sobolev inequality with constant $\gamma_0$ and $\nu_{0,\beta}$ satisfies
a log-Sobolev inequality with constant $\gamma_{0,\beta}$ then the inverse log-Sobolev constant $\gamma^{-1}$ of $\mu$ satisfies
\begin{equation} \label{e:Ising-summary1}
\frac{1}{\gamma}
\leq 
  \frac{1}{\gamma_0}\qa{1+\frac{2 \alpha^2}{\gamma_{0,\beta}}}.
\end{equation}

At this point the objective is to bound the log-Sobolev constant $\gamma_{0,\beta}$ of the renormalised measure $\nu_{0,\beta}$. Results are stated next under different conditions, with the corresponding verifications postponed to Section~\ref{sec_hessian_Ising}.

  Recall that the renormalised measure is $\nu_{0,\beta}(d\varphi) \propto 
  e^{- V_0(\varphi)}\Pg_{C_\beta-C_0}(d\varphi)$ with the renormalised potential $V_0$ defined in~\eqref{eq_def_V_ising}. 
If the temperature is sufficiently high, namely  if $\beta<\alpha<1$,
then it turns out that $V_0$ is strictly convex 
so that the standard
Bakry--\'Emery criterion is applicable \cite{MR3926125} and gives (see Exercise~\ref{ex_high_temperature_ising}):
\begin{equation}
\label{eq: BE haute temperature}
  \gamma_{0,\beta} \geq \alpha-\alpha^2.
\end{equation}
Taking $\alpha \downarrow \beta$, this leads to the following theorem. 
We recall the convention fixed below   \eqref{eq: Ising model-bis} that the coupling matrix $A$ has spectrum in $[0,1]$.
This spectral condition appeared in \cite{MR3926125} (and unlike previously existing conditions applies to the Sherrington--Kirkpatrik spin glass).
See also \cite{MR4408509,2106.04105,2208.07844,2307.10466} for recent results on spin glasses.

\begin{theorem}[Spectral high temperature condition \cite{MR3926125}]\label{thm: high_temp_Ising}
  For $\beta<1$ (under the conventions stated below   \eqref{eq: Ising model-bis}),
  the Ising model $\mu$ satisfies the log-Sobolev inequality: 
for each $F:\{-1,1\}^\Lambda\to\R_+$,
\begin{equation} \label{e:lsiht}
\ent_{\mu}(F)
\leq 
\Big(1+\frac{2 \beta}{1-\beta}\Big)D_\mu(\sqrt{F})
.
\end{equation}
\end{theorem}
For $\beta>1$, the renormalised potential $V_0$ is in general not convex and the log-Sobolev constant $\gamma_{0,\beta}$ of the renormalised measure $\nu_{0,\beta}$ cannot be bounded using the Bakry--\'Emery criterion. 
However, the argument above can be generalised by using the multiscale Bakry--\'Emery criterion instead. 
Assuming $\dot\lambda_t$ are as in the multiscale Bakry--\'Emery condition \eqref{eq_lower_bound_Hess_V_t_sec5},
Theorem~\ref{thm:LSI-mon} gives:
\begin{equation}
  \frac{1}{\gamma_{0,\beta}}
  \leq
  |\dot C_0|\int_0^\beta e^{-2\lambda_t \, dt}
  =
  \frac{1}{\alpha^2}\int_0^\beta e^{-2\lambda_t}\, dt,
    \qquad \lambda_t = \int_0^t \dot\lambda_s \, ds,
\end{equation}
and substituting this into \eqref{e:Ising-summary1} gives the following log-Sobolev inequality:
\begin{equation} \label{e:Ising-summary-polchinski}
  \ent_{\mu}(F)
  \leq \frac{2}{\gamma} D_\mu(\sqrt{F}), \qquad \frac{1}{\gamma}\leq
  \frac{1}{\gamma_0}\qa{1+2\int_0^\beta e^{-2\lambda_t} dt },
  \qquad \lambda_t = \int_0^t \dot\lambda_s \, ds.
\end{equation}

In the same situation, 
instead of using the multiscale Bakry--\'Emery criterion to bound $\gamma_{0,\beta}$ in~\eqref{e:Ising-summary1}, 
one can prove a log-Sobolev inequality for $\mu$ by using the entropic stability estimate
discussed in Section~\ref{subsec: Entropic stability estimate}.
This was done in \cite{2203.04163} (to prove a modified log-Sobolev inequality, but the argument generalises to a log-Sobolev inequality for the standard Dirichlet form, see below). 
It turns out that, for the decomposition \eqref{eq: inverse covariance}, the conditions
\eqref{e:assCt-mon} and \eqref{e:ass-entstab bis} of the multiscale Bakry--\'Emery criterion and of the entropic stability estimate are \emph{identical}
provided $\dot\lambda_t = -\alpha_t$, see Exercise~\ref{ex:Ising_pol_entstab}.
This is not the case for other covariance decompositions, and in particular not for those used for continuous models,
see the discussion in Section~\ref{subsec: Entropic stability estimate}.
Using the entropic stability estimate \eqref{eq: LSI intermediaire}, the entropy is therefore bounded by:
\begin{equation}
\ent_{\mu}(F)
\leq 
e^{-\lambda_\beta}\E_{\nu_{0,\beta}}[\ent_{\mu_0^\varphi}(F)],
\qquad \lambda_t = \int_0^t \dot\lambda_s \, ds,
\end{equation}
and with the uniform log-Sobolev inequality for $\mu_0^\varphi$ this gives
\begin{equation} \label{e:Ising-summary-eldanchen}
\ent_{\mu}(F)
\leq 
\frac{2}{\gamma} D_\mu(\sqrt{F}), \qquad \frac{1}{\gamma} \leq 
\frac{1}{\gamma_0} e^{-\lambda_\beta},
\qquad \lambda_t = \int_0^t \dot\lambda_s \, ds
.
\end{equation}
To be precise, for the Ising model where $C_0 \neq 0$,
the estimate \eqref{eq: LSI intermediaire} holds with the left-hand side there replaced by $\ent_\mu(F)$ instead of $\ent_{\nu_0}(F)$.
To see this, replace $\PP_{0,t}$ by $\PP_t := \E_{\mu_t^\varphi}[\cdot ] = \PP_{0,t} \E_{\mu_0^\varphi}[\cdot]$
in the argument leading to \eqref{e:entstab-last}.
Then \eqref{e:entstab-last} continues to hold and gives the claim.

Estimate \eqref{e:Ising-summary-eldanchen} is very similar to the estimate \eqref{e:Ising-summary-polchinski} obtained using the multiscale Bakry--\'Emery criterion,
but not exactly identical. Both estimates can be applied up to the critical point in a very general setting for ferromagnetic Ising models and yield a
polynomial bound on the log-Sobolev constant under the mean-field bound which holds
on $\Lambda\subset \Z^d$ in $d\geq 5$, see Theorem~\ref{thm:ising} below.
The strategies of the two proofs have different advantages.
We summarise the results as follows.

\begin{theorem}[Covariance conditions for Ising models \cite{MR4705299,2203.04163}]\label{thm: near_crit_Ising}
  The log-Sobolev constant of the Ising model at inverse temperature $\beta$ is bounded by
  \eqref{e:Ising-summary-polchinski} or \eqref{e:Ising-summary-eldanchen}.
\end{theorem}

As discussed previously, we formulated the results for the log-Sobolev inequality with respect to the standard Dirichlet form.
This  is a canonical choice (as already explained in Section~\ref{sec:generalities}),
but the argument can be adapted easily to other choices of jump rates with the conclusion of a possibly
modified log-Sobolev inequality, see Section~\ref{sec:choice-Dirichlet}.
As pointed out in \cite{MR4408509}, other choices are of interest when the jump rates are unbounded.

\subsubsection{Hessian of the renormalised potential and covariance}\label{sec_hessian_Ising}

In both strategies, using the multiscale Bakry--\'Emery criterion or the entropic stability estimate,
the estimate of the log-Sobolev constant reduces to estimating the constants $\dot\lambda_t=-\alpha_t$ bounding the
Hessian of the renormalised potential from below.
From Lemma~\ref{lem:fluctuationmeasure}, 
 recall that these estimates follow from bounds on the covariance of the fluctuation measure, a point of view that is particularly useful for the Ising model.
Indeed, 
the Hessian of the renormalised potential can be represented as follows.

\begin{exercise}
\label{exo: formula_Hessian_sec5}
Show that
  \begin{equation}
    \He V_t(\varphi) = C_t^{-1} - C_t^{-1} \Sigma_{t}(C_t^{-1}\varphi+h)C_t^{-1},
 \label{eq: Hessian Ising case}
  \end{equation}
  where $\Sigma_t(g)= (\cov_{\mu_{t,g}}(\sigma_x,\sigma_y))_{x,y\in\Lambda}$
  is the covariance matrix of the Ising model $\mu_{t,g}$ at inverse temperature $t$ and site-dependent magnetic field $g$
  (so that \eqref{eq: nouvelle mesure fluctuation} reads $\mu_t^\varphi = \mu_{t,C_t^{-1}\varphi + h}$).
\end{exercise}

For $\beta<1$, one can obtain the following convexity directly from this representation,
which allows to apply the standard Bakry--\'Emery criterion to derive \eqref{eq: BE haute temperature} and conclude the proof of Theorem~\ref{thm: high_temp_Ising}.

\begin{exercise}\label{ex_high_temperature_ising}
  Let $1 \geq \alpha > \beta$, and set $C_t = (t A + (\alpha-t)\id)^{-1}$.
  Then $V_t$ is convex for all $t\in [0,\alpha]$.
\end{exercise}

\begin{proof}
Since $\mu_{0,g}$ is a product measure and $|\sigma_x|\leq 1$,
\begin{equation}
  \Sigma_0(g) = \diag(\var_{\mu_{0,g}}(\sigma_x))_{x\in\Lambda}
  \leq \id
  .
\end{equation}
Using that $C_0=\alpha^{-1}\id$, we deduce from \eqref{eq: Hessian Ising case} that 
\begin{equation}
  \He V_0(\varphi) = \alpha\id - \alpha^2 \Sigma_0(\alpha \varphi+h)  
  \geq (\alpha-\alpha^2)\id.
\end{equation}
Thus if $\alpha \leq 1$, it follows that $V_0$ is convex. By Proposition~\ref{prop:convex}, $V_t$ is convex for all $t>0$.
\end{proof}

For general $\beta >0$, semi-convexity criteria on the Hessian can equivalently be formulated as covariance estimates that hold uniformly in an external field.

\begin{exercise} \label{ex:Ising_pol_entstab}
From $C_t= (tA+(\alpha-t)\id)^{-1}$, one has
$\dot C_t = (\id -A)C_t^2$, $\ddot C_t = 2(\id-A)\dot C_tC_t$. 
The multiscale Bakry--\'Emery criterion \eqref{e:assCt-mon} and the entropic stability criterion \eqref{e:ass-entstab bis} thus hold with
\begin{equation}
-\dot\lambda_t=\alpha_t = \bar\chi_t
\end{equation}
where
\begin{equation}
\bar\chi_t 
= 
\sup_{g\in\R^\Lambda}\chi_t(g)
,
\qquad
\chi_t(g)
= 
\|\Sigma_t(g)\|
\end{equation}
is a uniform upper bound on the spectral radius of the covariance matrix $\Sigma_t(g)$ of an Ising model uniformly in an external field $g$.
\end{exercise}

Now a significant simplification occurs for Ising models with ferromagnetic interaction, 
meaning $A_{xy}\leq 0$ for $x\neq y$. 
This includes the case of the lattice Laplacian $\Delta$ acting on configurations according to $(\Delta\sigma)_x = \sum_{y\sim x}[\sigma_y-\sigma_x]$.
For ferromagnetic interactions, it turns out that the spectral radius of the covariance matrix is maximal at $0$ field:
\begin{equation}
\bar\chi_t 
= 
\chi_t(0)
.
\label{eq_bound_spectral_radius_Ising}
\end{equation}
This is a consequence of
the FKG inequality (which implies that the covariance matrix has pointwise nonnegative coefficients),
the Perron--Frobenious theorem (which therefore implies that the largest eigenvector has nonnegative entries),
and the following remarkable correlation inequality due to Ding--Song--Sun \cite{MR4586225} which implies that the covariance between any two spins is maximised at $0$ field.
\begin{proposition}[Ding--Song--Sun inequality {\cite[Corollary 1.3]{MR4586225}}] \label{prop:DSS}
Let $\mu=\mu_{\beta,h}$ be the Ising measure~\eqref{eq: Ising model-bis} with ferromagnetic interaction $A$ 
and external field $h\in [-\infty,\infty]^\Lambda$, 
with values $\pm\infty$ corresponding to boundary conditions. 
Then: 
\begin{equation}
\forall (x,y)\in\Lambda^2:\qquad
\cov_{\mu_{\beta,h}}(\sigma_x,\sigma_y)
\leq 
\cov_{\mu_{\beta,0}}(\sigma_x,\sigma_y) 
=
\E_{\mu_{\beta,0}}[\sigma_x\sigma_y],
.
\end{equation}
\end{proposition}

In particular, if the interaction $A$ is ferromagnetic 
and (for simplicity) in addition translation invariant, i.e., $A_{x,y}=A_{0,x-y}$,  then
\begin{equation}
\bar\chi_t = \chi_t =
\chi_t(0)
=
\sum_{x\in\Lambda}\E_{\mu_{t,0}}[\sigma_0\sigma_x]
\end{equation}
is the \emph{susceptibility} of the Ising model. 
It characterises the phase transition of the ferromagnetic Ising model, in the sense that, e.g., for $\Lambda\subset\Z^d$, the critical value $\beta_c$ of $\beta$ satisfies:
\begin{equation}
\beta_c 
:=
\sup\ha{ \beta>0 : \sup_{\Lambda\uparrow\Z^d}\sum_{x\in\Lambda}\E_{\mu_{\beta,0}}[\sigma_0\sigma_x]<\infty}
.
\end{equation}
By combining the Ding--Song--Sun correlation bound with the multiscale Bakry--\'Emery criterion, and in view of the above characterisation of $\beta_c$, 
the following log-Sobolev inequality up to the critical point for ferromagnetic Ising models on general geometries was proven in~\cite{MR4705299}. 
\begin{theorem}[\!\!{\cite[Theorem 1.1]{MR4705299}}] \label{thm:ising}
The log-Sobolev constant $\gamma_{\beta,h}$ of the Ising measure~\eqref{eq:  Ising model-bis} with ferromagnetic interaction satisfies:
\begin{equation}
\frac{1}{\gamma_{\beta,h}} 
\leq 
\frac{1}{2}+\int_0^\beta e^{2\int_0^t\chi_s\, ds}\, dt
,
\label{eq_bound_LSI_Ising}
\end{equation}
with $\chi_\beta = \sup_x \sum_y \E_{\mu_{\beta,0}}[\sigma_x\sigma_y]$ or more generally equal to the largest eigenvalue of $(\E_{\mu_{\beta,0}}[\sigma_x\sigma_y])_{x,y}$.
\end{theorem}
The above result implies that $\gamma_{\beta,h}$ is bounded below uniformly in the size of the lattice as long as $\beta<\beta_c$.
For special geometries, this could already be argued from~\cite{MR4586225} (which establishes the strong spatial mixing property).

As seen in Example \ref{ex_lattice_phi4_mf} for the lattice $\varphi^4$ models,
\eqref{eq_bound_LSI_Ising} also gives an explicit bound on the log-Sobolev constant. 
For Ising models with mean-field interactions, 
the bound implies that $\gamma_{\beta,h}$ is of order $\beta_c-\beta$, 
which is the correct scaling. 
More significantly, the bound~\eqref{eq_bound_LSI_Ising} 
implies a polynomial bound on $1/\gamma_{\beta,h}$ on $\Lambda \subset \Z^d$ in dimension five and higher
(and more generally under the so-called mean-field bound on the susceptibility, i.e., when the susceptibility diverges linearly),
where polynomial means as a function of $\beta_c-\beta$ when $\beta<\beta_c$, 
and of the lattice size when $\beta=\beta_c$.
The degree of this polynomial is not expected to be sharp unless the constant $D$ in Example~\ref{example: mean field} is equal to $1$.

\begin{example}
\label{example: mean field}
  If $\chi_\beta \leq D/(\beta_c-\beta)$ then $\gamma_{\beta,h}$ is bounded polynomially in $\beta_c-\beta$,
  and if $\chi_\beta \leq D/(\beta_c-\beta+L^{-\alpha})$ then $\gamma_{\beta_c,h}$ is polynomial in $L$.
  These assumptions hold for the ferromagnetic nearest-neighbour Ising model on $\Lambda\subset \Z^d$ when $d\geq 5$ \cite{MR678000}.
\end{example}

Similarly, one can recover the main result of \cite{MR3059200} from \eqref{eq_bound_LSI_Ising}, which improves the high temperature condition 
of \eqref{e:lsiht} in the case of ferromagnetic Ising models on graphs with maximal degree $d$ (for the mixing time and spectral gap).
For the ferromagnetic Ising model, $A$ is the negative adjacency matrix of the graph
which we now (for comparison) \emph{do not} normalise to have spectrum contained in $[0,1]$.
The condition of \eqref{e:lsiht} cannot be improved for general non-ferromagnetic interactions with bounded spectral radius,
but for ferromagnetic models the condition established in \cite{MR3059200} is
$\beta < \operatorname{artanh}(1/(d-1))$, whereas the condition of \eqref{e:lsiht} translates in this case to $\beta< 1/(2d)$.
The value $\beta_u=\operatorname{artanh}(1/(d-1))$ is the uniqueness threshold for the ferromagnetic Ising model on the infinite $d$-regular tree
and also the critical point for ferromagnetic Ising models on random $d$-regular graphs \cite{MR3059200}.
This application is summarised in the next example.

\begin{example}
   Let $A$ be the negative adjacency matrix of a finite graph of maximal degree $d$ (not normalised to have spectrum in $[0,1]$).
   If $(d-1) \tanh \beta < 1$ then $\gamma_{\beta,h}$ is uniformly bounded below (and diverges polynomially in $\beta \uparrow \beta_u= \operatorname{artanh}(1/(d-1))$ and in the size of the graph if $\beta=\beta_u$).
 \end{example}

 \begin{proof}[Sketch]
From  \cite[Section~5.4, Eq. (39)]{MR3059200}, we deduce that for any sites $x,y$ in the graph:
\begin{equation}
\label{eq: decay comparison with tree}
\E_{\mu_{\beta,0}}[\sigma_x \sigma_y]\leq 
\sum_{w \in \mathcal{S}(y)}  \big( \tanh \beta \big)^{\text{dist}(x,w) },
\end{equation}
where the bound is obtained by comparing the graph with a  $d$-regular tree  as in \cite{zbMATH06373885}
 and $\mathcal{S}(y)$ stands for the set of sites associated with $y$ in this tree.
Summing over the sites $y$ in the graph boils down to sum over all the sites in $\cup_y \mathcal{S}(y)$, i.e., in the $d$-regular tree.
Thus we get 
\begin{equation}
\chi_\beta = \sup_x \sum_y \E_{\mu_{\beta,0}}[\sigma_x\sigma_y]
\leq  1+ C_d \sum_{\ell \geq 1}  \big( (d-1) \tanh \beta \big)^\ell
\leq \frac{C_d}{1 - (d-1) \tanh \beta},
\end{equation}
i.e., one finds a divergence of the susceptibility as in Example~\ref{example: mean field} as $\beta \uparrow \beta_u=\operatorname{artanh}(1/(d-1))$.
Moreover, when $\beta$ approaches $\beta_u$,
one can use that for a graph with $N$ sites then trivially $\chi_\beta \leq N$, 
so that 
\begin{equation}
  \chi_\beta
  \leq \frac{C_d}{1 - (d-1) \tanh \beta + N^{-1}}
  .
\end{equation}
By Theorem~\ref{thm:ising}, we deduce a  polynomial lower bound for the log-Sobolev constant in the size of the graph at $\beta=\beta_u=\operatorname{artanh}(1/(d-1))$. 
\end{proof}

\subsubsection{Choice of Dirichlet form} \label{sec:choice-Dirichlet}
As in the original references \cite{MR3926125,MR4705299},
the above discussion is formulated in terms of the standard Dirichlet form \eqref{e:Ising-Dirichlet-standard}.
There exist general comparison arguments between Dirichlet forms that ensure that one can transfer the log-Sobolev inequality obtained for a certain dynamics to another one, 
see, e.g., Chapter~4 in~\cite{MR1490046}. 
Namely, if $c_1,c_2$ are families of jump rates reversible with respect to the same measure $\nu$ on $\{-1,1\}^\Lambda$ 
and $\gamma_1,\gamma_2$ denote the associated log-Sobolev constants, 
then, for $K>0$:
\begin{equation}
\forall \sigma,\sigma'\in\{-1,1\}^\Lambda: 
\qquad 
K^{-1} c_2(\sigma,\sigma')\leq c_1(\sigma,\sigma')\leq K c_2(\sigma,\sigma')
\quad \Rightarrow\quad 
K\gamma_2
\leq 
\gamma_1
\leq 
\gamma_2/K
.
\label{eq_comparison_jump_rates}
\end{equation}
One can for instance check that the heat-bath- and canonical jump rates associated with an Ising measure with interaction $\beta A$ and external field $h\in\R^\Lambda$ satisfy such a bound, 
with a constant $K$ that depends only on $\|h\|_\infty$ and $\beta \max_i\sum_j|A_{ij}|$. 
The heat-bath Dirichlet form is \cite[(3.4)--(3.6)]{MR1746301}:
\begin{equation}\label{e: HB_Dirichlet_form}
  D_\mu^{\rm HB}(F)
  =
  \frac12 \sum_{\sigma\in\{\pm 1\}^\Lambda} \sum_{x\in\Lambda} \Psi(\mu(\sigma),\mu(\sigma^x)) (F(\sigma)-F(\sigma^x))^2, \qquad \Psi(a,b)=\frac{ab}{a+b}.
\end{equation}

There are however situations in which the comparison argument~\eqref{eq_comparison_jump_rates} is not applicable.
This observation was made in~\cite{MR4408509} in the situation of the SK model treated in~\cite{MR3926125},
which is an Ising model with random couplings which can take arbitrarily large values.
This makes the constant $K$ in~\eqref{eq_comparison_jump_rates} arbitrarily small. 
In such cases, 
if one is interested in a dynamics different from the one induced by the canonical jump rates, 
it is desirable to directly obtain the log-Sobolev inequality (or modified log-Sobolev inequality) for this dynamics. 
This was done for the heat-bath dynamics, for the modified log-Sobolev inequality, in 
\cite{MR4408509,2106.04105,2203.04163}.

One can however check that the arguments of~\cite{MR3926125,MR4705299} sketched above 
are not specific to the standard Dirichlet form.  
It is in fact straightforward to apply them to other dynamics as explained below in the heat-bath case, provided:
\begin{itemize}
	\item the associated single-spin LSI constant (or modified LSI constant) is uniform in the field;
	\item the associated Dirichlet form is a concave function of the measure (this can be generalised).
\end{itemize}
Let us show how this works in the heat-bath case~\eqref{e: HB_Dirichlet_form}, for which both points are satisfied in view of Proposition~\ref{prop: HB_single_spin_LSI} and the concavity of $\Psi(a,b)=ab/(a+b)$. 
Instead of using the single-spin log-Sobolev inequality for the standard Dirichlet form of Proposition~\ref{prop: standard_single_spin_LSI}, 
one can instead apply the single-spin modified LSI with the heat-bath Dirichlet form of Proposition~\ref{prop: HB_single_spin_LSI} since it also holds uniformly in the external field.
For instance, for $\E_{\nu_{0,\beta}}[\ent_{\mu_0^\varphi}(F)]$ the bound becomes:
\begin{align}
  &\E_{\nu_{0,\beta}} \qa{\ent_{\mu_0^\varphi}(F(\sigma))}
    \nnb
  &\leq \frac12 \sum_{\sigma} \sum_x \E_{\nu_{0,\beta}}\qb{ \Psi(\mu_0^\varphi(\sigma),\mu_0^\varphi(\sigma^x)) } (F(\sigma)-F(\sigma^x))(\log F(\sigma)-\log F(\sigma^x))
    \nnb
  &\leq \frac12 \sum_{\sigma} \sum_x \Psi(\mu(\sigma),\mu(\sigma^x)) (F(\sigma)-F(\sigma^x))(\log F(\sigma)-\log F(\sigma^x))
    = D_\mu^{\rm HB}(F,\log F),
\end{align}
where the first inequality is the single-spin modified log-Sobolev inequality from Proposition~\ref{prop: HB_single_spin_LSI},
and the second inequality is Jensen's inequality, using
$\mu(\sigma) = \E_{\nu_{0,\beta}}[\mu_0^\varphi(\sigma)]$ and that $\Psi$ is concave. 
The bound for the other term in~\eqref{e:ent-decomp-Ising} works analogously and gives $4D_\mu^{\rm HB}(\sqrt{F}) \leq D_\mu^{\rm HB}(F,\log F)$
instead of $D_\mu(\sqrt{F})$ with the standard Dirichlet form.

The conclusion is that the modified log-Sobolev constant for the heat-bath Dirichlet form satisfies exactly the same bound
(up to an overall factor $4$ from different normalisations): 
\begin{equation}
  \ent_{\mu}(F) \leq \frac{1}{2\gamma_0^{\rm HB}} D_\mu^{\rm HB}(F,\log F),
  \qquad
  \frac{1}{\gamma_{0}^{\rm HB}} \leq 2 + 4 \int_0^\beta e^{-2\lambda_t} \,dt.
\end{equation}
This strategy can be further generalised to Dirichlet forms which are not concave in the measure,
see \cite{2310.04609}.

\subsection{Applications to conservative dynamics}
The criterion of Theorem~\ref{thm:LSI-mon} in principle also applies to dynamics with a conservation law. 
This is for instance the case for spin models with constrained magnetisation:
\begin{equation}
\mu_{N,m}(d\varphi)
\propto 
	e^{-V_0(\varphi)}\, d\varphi|_{X_{N,m}}
,
\label{eq_conservative_measure}
\end{equation}
with $X_{N,m}$ the hyperplane of spins with magnetisation $m\in\R$:
\begin{equation}
X_{N,m} 
:= 
\ha{\varphi\in\R^N : \sum_i\varphi_i = Nm}
.
\end{equation}
The infinite temperature case is $V_0=\sum_{i=1}^NV(\varphi_i)$ and $V:\R\to\R$ a $C^2$ potential, 
assumed to be strictly convex outside of a segment for definiteness.  
The associated conservative dynamics reads:
\begin{equation}
d\varphi_t
=
-\nabla V_0(\varphi_t)\, dt + \sqrt{2} \, dB_t
,
\label{eq_conservative_dynamics}
\end{equation}
where $(B_t)$ is a standard Brownian motion on $X_{N,0}$.

See~\cite{2310.04609} and the forthcoming work \cite{kawasakicontinuum} for results in discrete and continuous spin settings,
and \cite{MR4303014} for such results for the continuum sine-Gordon model.

\appendix
\section{Classical renormalised potential and Hamilton--Jacobi equation}
\label{app:HJ}

\subsection{Hamilton--Jacobi equation}
In classical field theory, fields are typically minimisers of an action functional $S$:
\begin{equation}
  \varphi_0 \in \operatorname{argmin}_\varphi S(\varphi) , \qquad S(\varphi) = \frac12 |\varphi|^2 + V(\varphi).
\end{equation}
These can be related to the Hamilton--Jacobi equation
\begin{equation} 
\label{e:HJ}
  \ddp{V_t}{t} = - \frac12 (\nabla V_t)^2, \qquad V_0(\varphi) = V(\varphi).
\end{equation}
Indeed, its unique viscosity solution is given by
the Hopf--Lax formula \cite[Section~3.3.2]{MR1625845}:
\begin{equation}
\label{eq: def Vt HJ}
  V_t(\varphi) = \min_{\zeta}\pa{V_0(\zeta)+\frac{t}{2} |\frac{\varphi-\zeta}{t}|^2}.
\end{equation}
In particular, the minimum of the action $S$ is given by $V_1(0)$.
Note the analogy with the renormalised potential from~\eqref{e:V-def}.
The constructions of Sections~\ref{sec:pathwise} and~\ref{sec:variational-transport}
have classical analogues in which the role of the Polchinski equation is replaced by that
of the Hamilton--Jacobi equation, as we will see in this appendix.

Unlike solutions of the Polchinski equation, which are smooth at least for $\dot C_t$
nondegenerate and a finite number of variables (as in our discussion), the Hamilton--Jacobi equation can develop shocks and the appropriate weak solutions are not necessarily smooth.
However, we can assume that $V$ is locally Lipschitz continuous and that \eqref{e:HJ} holds almost everywhere.
We refer to \cite[Chapters 3 and 10]{MR1625845} for an introduction.
In statistical physics, Hamilton--Jacobi equations are well-known to arise in mean-field \emph{limits}
of statistical mechanical models, see \cite{MR2310196} and in particular \cite{MR4275243,dominguez-Mourrat-2023} and references for recent
work in the context of disordered models, as well as Section~\ref{sec:CW} below.
Shocks of the Hamilton--Jacobi equations are related to phase transitions.
Note that the Polchinski equation in a finite number of variables 
describes finite systems (rather than limits) and thus has smooth solutions.
It therefore provides a complete description of the models (no information is lost)
while the Hamilton--Jacobi equations describing mean-field systems
are effective equations describing macroscopic information.
In the thermodynamic limit, where the number of variables tends to infinity, shocks can also form in the Polchinski equation and then likewise correspond to phase transitions
in the statistical mechanical models. We illustrate the above in the simple example of the mean-field Ising model in Section~\ref{sec:CW} below.

\medskip

We now discuss the `classical' analogues of the constructions of Sections~\ref{sec:pathwise}--\ref{sec:variational-transport} for the Hamilton--Jacobi equation.
Our goal is to emphasise the analogy and to provide a different intuition also for the stochastic
constructions. We will therefore impose convenient regularity assumptions in all statements.

We begin with a classical analogue of Corollary~\ref{cor:coupling}. Define the
classical renormalised action by
\begin{equation}
  S_t(\varphi)=\frac{|\varphi|^2}{2(1-t)}+V_t(\varphi).
\end{equation}
Minimisers of $S_t$ will take the role of the renormalised measure introduced in \eqref{e:nu-def-bis}.

\begin{proposition} \label{prop:coupling-classical}
Assume that $(\varphi_t)_{t\in [0,1]}$ is differentiable in $t\in (0,1)$,
that $\varphi_t\to 0$ as $t\to 1$,
that $V$ is smooth along $(\varphi_t)_{t\in (0,1)}$,
and that $\varphi_t$ is an isolated local minimum of $S_t$ for each $t\in [0,1)$.
Then for $t\in [0,1)$,
\begin{equation} \label{e:coupling-classical}
  \varphi_t = - \int_t^1 \nabla V_u(\varphi_u)\, du.
\end{equation}
\end{proposition}

Different from the situation in Corollary~\ref{cor:coupling},
the choice of $\varphi$ is not necessarily unique because $S$ may have multiple minimisers and $V$ need not be globally smooth.
Indeed, the equation \eqref{e:coupling-classical} for $\varphi$ is nothing but the equation for
the characteristics of Hamilton--Jacobi equation \eqref{e:HJ}, see \cite[Section 3.3]{MR1625845}.
These are the curves $(\varphi_t)$ such that $\frac12 U_t(\varphi_t)^2$ is constant in $t$,
where $U_t = \nabla V_t$. 
Since by the Hamilton--Jacobi equation \eqref{e:HJ},
\begin{equation} \label{e:HJ-U}
  \ddp{U_t}{t} = -(\nabla U_t, U_t)
  ,
\end{equation}
provided $U$ is smooth at $(t,\varphi)$, then
\begin{equation} \label{e:HJ-U2}
  \ddp{}{t}U_t(\varphi_t) = -(\nabla U_t(\varphi_t), U_t(\varphi_t)) + (\nabla U_t(\varphi_t),\dot\varphi_t)
  = -(U_t(\varphi_t)-\dot\varphi_t,\nabla U_t(\varphi_t)) = 0.
\end{equation}

\begin{proof}[Sketch of Proposition~\ref{prop:coupling-classical}]
Since $\varphi_t \to 0$ as $t\to 1$, the claim is equivalent to proving that 
\begin{equation}
  \dot\varphi_t = \nabla V_t(\varphi_t).
\end{equation}
Since $\varphi_t$ minimises $S_t$, it satisfies the Euler--Lagrange equation
\begin{equation}
  \frac{\varphi_t}{1-t} + \nabla V_t(\varphi_t) =   \nabla S_t(\varphi_t) = 0.
\end{equation}
Differentiating this equation in $t$, 
\begin{equation}
  \frac{\dot\varphi_t}{1-t} + \frac{\varphi_t}{(1-t)^2} + (\partial_t \nabla V_t)(\varphi_t) + \He V_t(\varphi_t)\dot\varphi_t= 0.
\end{equation}
Using the Euler--Lagrange equation $\varphi_t/(1-t)=-\nabla V_t(\varphi_t)$
again and $\partial_t \nabla V_t = -\He V_t \nabla V_t$ 
 (which is \eqref{e:HJ-U}), we obtain
\begin{equation}
  \frac{\dot\varphi_t - \nabla V_t(\varphi_t)}{1-t} - \He V_t(\varphi_t)\nabla V_t(\varphi_t) + \He V_t(\varphi_t)\dot\varphi_t= 0.
\end{equation}
Therefore
\begin{equation}
  \He S_t(\varphi_t)[-\nabla V_t(\varphi_t) + \dot\varphi_t]
  =(\frac{1}{1-t}+\He V_t(\varphi_t))[-\nabla V_t(\varphi_t) + \dot\varphi_t]=0.
\end{equation}
Since $\varphi_t$ is an isolated local minimum of $S_t$, the Hessian on the left-hand side
is strictly positive definite so that necessarily $\nabla V_t(\varphi_t)-\dot\varphi_t=0$.
\end{proof}

The classical version of F\"ollmer's problem is the following control problem.
We continue to use the convention that the ODE is backwards in time so that the
Hamilton--Jacobi equation has initial rather than terminal condition.
Suppose that $U = (U_u(\cdot))_{u\in [0,1]}$ are given smooth functions and that $\varphi^U$ solves
the classical analogue of  \eqref{e:coupling-U}: $\varphi^U_t \to 0$ as $t\to 1$ and
\begin{equation} 
\label{e:coupling-U-HJ}
  \varphi_t^U
  = - \int_t^1 U_u(\varphi_u^U) \, du,
\end{equation}
with $\varphi_0^U$ an absolute minimum of $S$ (which we recall is the classical analogue of demanding that $\varphi_0^U$
is a random sample from a desired target measure $\nu_0 \propto e^{- \beta S}$).
The classical analogues of Theorem~\ref{thme: Follmer general} and Proposition~\ref{prop:BD} are as follows.

\begin{proposition} \label{prop:ent-classical}
  The optimal drift is given by $U= \nabla V$  in the following sense.
  For any smooth $U$ and associated trajectory $\varphi^U$ that satisfies
  \eqref{e:coupling-U-HJ}, 
  \begin{equation} \label{e:ent-classical}
    \frac12 |\varphi^U_0|^2 \leq \frac12 \int_0^1 |U_u(\varphi_u^U)|^2\, du,
  \end{equation}
  with equality if $U = \nabla V$. 
\end{proposition}

Comparing with Theorem~\ref{thme: Follmer general},
the classical analogue of $\bbH(\nu_0|\gamma_0)$ is simply $\frac12 |\varphi_0|^2$
and the classical analogue of the path space entropy $\bbH({\bf Q}|{\bf P})$ is the cost $\frac12 \int_0^1 |U_u(\varphi_u^U)|^2\, du
= \frac12 \int_0^1 |\dot\varphi_u^U|^2 \, du$.

\begin{proof}[Sketch]
  The proof is analogous to that of Theorem~\ref{thme: Follmer general}.
  Indeed, by the Hopf--Lax formula,
  \begin{equation}
    V_1(0) - V_0(\varphi^U_0) \leq V_0(\varphi^U_0) + \frac12 |\varphi^U_0|^2 - V_0(\varphi^U_0) = \frac12 |\varphi_0^U|^2,
  \end{equation}
  with equality if and only if $\varphi_0^U$ is an absolute minimum of $S$.
  On the other hand, assuming $V$ is locally Lipschitz continuous (so that the fundamental theorem of calculus hold), one has
  \begin{equation}
V_1 (0) - V_0 (  \varphi_0^U)
=  \int_0^1 dt \; \frac{\partial}{\partial t} V_t ( \varphi_t^U ),
\end{equation}
with $\varphi^U$ evolving according to \eqref{e:coupling-U-HJ}, and therefore (whenever the derivative exists classically)
\begin{align}
\frac{\partial}{\partial t} V_t (  \varphi_t^U)
& = (\ddp{}{t}    V_t) (  \varphi_t^U) 
+ \big( \nabla   V_t (  \varphi_t^U) ,    U_t (  \varphi_t^U) \big)
  \nnb
& = -  \frac12 (\nabla   V_t(  \varphi_t^U) )^2 +  \big( \nabla   V_t (  \varphi_t^U) ,    U_t (  \varphi_t^U ) \big)
  \nnb
& =  - \frac12  \big( \nabla   V_t(  \varphi_t^U)   - U_t (  \varphi_t^U) \big)^2  
 + \frac12 (U_t (  \varphi_t^U) )^2 ,
\end{align}
where we used the Hamilton--Jacobi equation on the second line. In particular, if $\varphi_0^U$ is an absolute minimum of $S$,
\begin{equation}
  \frac12 \int_0^1  (\nabla   U_t(\varphi_t^U))^2 \, dt  
  =
  \frac12 |\varphi_0^U|^2
  + \frac12 \int_0^1  \big( \nabla   V_t(\varphi_t^U) -    U_t(\varphi_t^U)  \big)^2 \, dt
  \geq   
  \frac12 |\varphi_0^U|^2,
\end{equation}
and the gradient of the renormalised potential $V_t$ provides the optimal drift. 
\end{proof}

As in \eqref{e:stochastic},
instead of \eqref{e:coupling-U-HJ}, 
one can also consider  the equations for the characteristics in a reduced time interval $[0,t]$ with $t \leq 1$ and
$\varphi_t^U = \varphi$,
\begin{equation} 
\label{e:coupling-U-HJ time s}
  \varphi_s^U
  = \varphi - \int_s^t U_u(\varphi_u^U) \, du , \qquad (s \leq t).
\end{equation}

\begin{proposition} \label{prop:BD-classical}
  Under certain regularity conditions on $V_0$,  and for  $t \leq 1$ then
  \begin{equation}
  \label{eq: variation HJ temps intermediaire}
    V_t(\varphi) = \inf_{U}
    \qa{V_0 \Big( \varphi- \int_0^t U_s(\varphi_s^U) \, ds \Big) 
    + \frac12 \int_0^t |U_s(\varphi_s^U)|^2 \, ds}.
  \end{equation}
\end{proposition}

This is discussed in \cite[Section I.9 and I.10]{MR2179357}. Namely, with $L(t,x,v) = \frac12 v^2$ and $\psi = V_0$
(and opposite time direction), the statement follows from  \cite[Theorem 10.1]{MR2179357}.
This minimiser does not have to be unique if the Hamilton--Jacobi equation has a shock, see \cite[Theorem 9.1 and 10.2]{MR2179357}.

\begin{proof}[Sketch]
  The Hopf--Lax formula implies that
  \begin{equation} \label{e:BD-classical-pf1}
      V_t(\varphi)
      = \min_{\zeta}\pa{\frac{t}{2} |\frac{\varphi-\zeta}{t}|^2+ V_0(\zeta)}
      \leq
      \frac{1}{2t} |\varphi-\varphi_0^U|^2 + V_0(\varphi_0^U)
    \end{equation}
    where
    given any drift $U_s$, we let $\varphi_0^U$ be the final condition of \eqref{e:coupling-U-HJ time s}. 
    The first term on the right-hand side of  \eqref{e:BD-classical-pf1}  is bounded as in \eqref{e:ent-classical}:
  \begin{equation}
    \frac{1}{2t} |\varphi-\varphi_0^U|^2 \leq \frac12 \int_0^t |U_s(\varphi_s^U)|^2\, ds.
  \end{equation}
  This shows
  \begin{equation}
    V_t(\varphi)
    \leq \frac12 \int_0^t |U_s(\varphi_s^U)|^2\, ds 
    + V_0 \Big( \varphi-\int_0^t U_s(\varphi_s^U)\, ds  \Big)
    .
  \end{equation}

  On the other hand, 
  equality if $U=\nabla V$ follows from the Hamilton--Jacobi equation 
  as in Proposition~\ref{prop:V-mart}:
  if $\varphi$ is a solution to \eqref{e:coupling-classical},
 \begin{equation}
     \ddp{}{s} \qa{V_s(\varphi_s)
       + \frac12 \int_s^t (\nabla V_u(\varphi_u))^2\, ds}
       = \ddp{V_s}{s}(\varphi_s) + (\nabla V_s(\varphi_s))^2 - \frac12 (\nabla V_s(\varphi_s))^2 = 0,
     \end{equation}
     i.e.,
     \begin{equation}
       V_t(\varphi) - V_0(\varphi_0) = \frac12 \int_0^t (\nabla V_u(\varphi_u))^2\, du.
     \end{equation}
     This solution need not be unique.
\end{proof}

\subsection{Example: Mean-field Ising model}
\label{sec:CW}

We conclude this section with the example of the mean-field Ising model,
which can be described both in terms of a Polchinski equation and, in the limit, by a Hamilton--Jacobi equation. 
The mean-field Ising model is given by the measure
    \begin{equation}
      \E_\nu[G] = \frac{1}{Z_N(\beta, {\bf h} )} \sum_{\sigma \in \{\pm 1\}^N} e^{-\frac{\beta}{4 N}\sum_{i,j}(\sigma_i - \sigma_j)^2 + (\sigma, {\bf h})} G(\sigma),
    \end{equation}
where the vector ${\bf h} \in \R^N$ is a possibly site-dependent external field and $Z_N(\beta, {\bf h} )$ is a normalisation factor.
It is convenient to rewrite it     as
    \begin{equation}
      \E_\nu[G] = \frac{1}{Z_N(\beta, {\bf h} )} \sum_{\sigma \in \{\pm 1\}^N} e^{-\frac{\beta}{2}(\sigma,P\sigma) + (\sigma, {\bf h})} G(\sigma),
    \end{equation}
    where $P= \id-Q$ and $Q$ is the orthogonal projection onto constants:
    $Q f = \big( \frac{1}{N}\sum_i f_i \big) {\bf 1}$ with ${\bf 1} = (1,\dots, 1) \in \R^N$.
    For ${\bf h} = h {\bf 1}$ with $h \in \R$, the free energy $F(\beta,h)$ is the limit $N\to\infty$ of $F_N(\beta,h)$ where
    \begin{equation}
      F_N(\beta,h) = -\frac{1}{N} \log Z_N(\beta, h {\bf 1}).
    \end{equation}
    (Physically more correctly, the right-hand side should have been divided by $\beta$, but it is here more convenient to omit this.)
    It is well-known and easy to check that 
    \begin{equation}
      \ddp{F_N}{\beta} = \frac{1}{2N} \ddp{^2F_N}{h^2} -\frac12 (\ddp{F_N}{h})^2, \qquad F_N(0,h) = -\log\cosh(h),
    \end{equation}
    and thus that the limiting free energy $F$ is the viscosity solution of the Hamilton--Jacobi equation
    \begin{equation} \label{e:CW-F-HJ}
      \ddp{F}{\beta} = -\frac12 (\ddp{F}{h})^2, \qquad F(0,h) = -\log\cosh(h)
      ,
    \end{equation}
    see in particular \cite{MR4275243}. 
    Equivalently, $F$ is given by the Hopf--Lax formula which coincides with the well-known variational formula
    for the free energy alternatively obtained from Laplace's Principle (see for example \cite[Chapter~1]{MR3969983} or \cite{MR1239893}):
    \begin{equation}
      F(\beta,h)
      = \min_{g\in\R} \qa{ \frac{1}{2\beta}(g-h)^2 - \log \cosh(g)}
      = \min_{\varphi\in\R} \qa{ \frac{\beta}{2}\varphi^2 - \log \cosh(\beta\varphi +h)}.
    \end{equation}

    This Hamilton--Jacobi equation for the free energy can be related, as follows, to the Polchinski equation studied earlier  in Section \ref{subsubsec: Renormalised potential}.
    Recall that $\id =P+Q$ with $Q$ the orthogonal projection onto constant vectors in $\R^N$ and $PQ=0$.
    For $\alpha>\beta$, let
    \begin{equation} \label{e:CW-dotC}
      C_t = (tP + (\alpha-t))^{-1}, \qquad 
      \dot C_t 
      = (\alpha-t)^{-2}Q,
    \end{equation}
    where we used $PQ=0$ to simplify $\dot C_t$. For $\varphi \in \R^N$, the renormalised potential
    \eqref{eq_def_V_ising} then is
  \begin{equation}
    V_t(\varphi)
    = -\log \sum_{\sigma \in \{\pm 1\}} e^{-\frac12 (\sigma-\varphi,(tP+(\alpha-t))(\sigma-\varphi))} + \text{(constant)}
  \end{equation}
  and satisfies the Polchinski equation (for appropriate otherwise irrelevant choice of the constants):
  \begin{equation}
    \ddp{V_t}{t} = \frac12 \Delta_{\dot C_t} V_t - \frac12 (\nabla V_t)_{\dot C_t}^2
    = (\alpha-t)^{-2} \qa{\frac12 \Delta_{Q}V_t - \frac12 (\nabla V_t)_{Q}^2}    .
  \end{equation}
  Note that the right-hand side only depends on derivatives of $V_t$ in constant directions (i.e., in the image of $Q$).
  By \eqref{e:CW-dotC}, the covariance $C_\beta-C_t$ of the renormalised measure $\nu_t =  \nu_{t,\beta}$ defined
  in  \eqref{eq_nut_ising} is also proportional to $Q$ and therefore supported on constant fields.
  Thus one can restrict the renormalised potential $V_t$ to constant fields
  and the restriction satisfies a closed equation.
  Explicitly, for $\tilde\varphi \in \R$, define $\tilde V(\tilde\varphi) = \frac{1}{N} V(\tilde\varphi {\bf 1})$.
  In other words, $V_t(\varphi)= N \tilde V_t(Q\varphi)= N\tilde V_t(\frac{1}{N}\sum_i \varphi_i)$ holds for constant fields $\varphi=Q\varphi$ and
  \begin{align}
    \ddp{\tilde V_t}{t}(\tilde \varphi)
    = \frac{1}{N}\ddp{V_t}{t}(\tilde\varphi{\bf 1})
    &=
    (\alpha-t)^{-2} \frac{1}{N}\qa{ \frac12 \Delta_Q V_t - \frac12 (\nabla_Q V_t)^2} (\tilde \varphi{\bf 1})
    \nnb
    &=
      (\alpha-t)^{-2} \qa{ \frac12 \ddp{V_t}{\varphi_1^2} - \frac12 (\ddp{V_t}{\varphi_1})^2} (\tilde \varphi{\bf 1})
      \nnb
      &    = (\alpha-t)^{-2} \qa{\frac{1}{2N} \tilde V_t'' - \frac12 (\tilde V'_t)^2 }(\tilde\varphi).
  \end{align}
  Thus the reduced (one-variable) Polchinski equation that $\tilde V_t$ satisfies has a prefactor $1/N$ in front of the Laplacian term,
  and its limit $\overline{V_t}$ as $N\to\infty$ is the unique viscosity solution of the following (one-variable) Hamilton--Jacobi equation
  (now dropping tilde from the variable $\varphi$):
  \begin{equation}
  \label{eq: HJ mean field overline V}
    \ddp{\overline{V_t}}{t}(\varphi) = - \frac12 (\alpha-t)^{-2} (\overline{V_t}'(\varphi))^2, \qquad \overline{V_0}(\varphi) = \frac{\alpha}{2}\varphi^2 - \log \cosh(\alpha \varphi), \qquad (\varphi \in \R).
  \end{equation}

\medskip

To conclude this section,  we relate the Hamilton--Jacobi equation \eqref{eq: HJ mean field overline V} for the renormalised potential to the one satisfied by the free energy \eqref{e:CW-F-HJ}.
We follow the argument in Example~\ref {ex:V-F} and first note that for a constant field $\varphi {\bf 1}$ with $\varphi \in \R$, the renormalised potential can be written as (see \eqref{e:V-F})
  \begin{equation} \label{e:CW-V-F}
    V_t(\varphi {\bf 1})
    = N \qa{ \frac{\alpha-t}{2}\varphi^2 + F_N(t,(\alpha-t)\varphi) }.
  \end{equation}
  Thus $F(t,h) = - \frac12 (\alpha-t)^{-1} h^2 + \overline {V_t}((\alpha-t)^{-1}h)$
  and the Hamilton--Jacobi equation \eqref{e:CW-F-HJ} for $F$ follows from the one of $\overline{V}$ exactly as in Example~\ref{ex:V-F}.
  Indeed, in the setting of that example with the relation \eqref{e:V-F} between $V_t$ and $F_t$ one has in general that
  \begin{equation}
   \ddp{}{t} V_t = - \frac12 (\nabla V_t)_{\dot C_t}^2
    \qquad\Leftrightarrow\qquad
    \ddp{}{t} F_t 
    = -\frac12 (\nabla F_t)_{\dot \Sigma_t}^2,
  \end{equation}
  and in the present example the choice of $C_t$ corresponds to $\dot\Sigma_t = Q$.

\section*{Acknowledgements}

This work was supported by the European Research Council under the European Union's Horizon 2020 research and innovation programme
(grant agreement No.~851682 SPINRG). 

We thank D.~Chafai, R.~Eldan, N.~Gozlan, J.~Lehec, Y.~Shenfeld, and H.-T.~Yau for various discussions related to the material of this introduction,
encouragement, and for pointing out several additional references.

We thank the organisers of the following summer schools at which some of the material was presented:
the One World Probability Summer School on ``PDE and Randomness'' in Bath/Zoom organised by Hendrik Weber and Andris Gerasimovics;
the ``Summer School on SPDE and Related Fields'' in Beijing/Zoom organised by Hao Shen, Scott Smith, Rongchan Zhu, and Xiangchan Zhu;
and the Summer School on ``PDE and Randomness'' at the Max Planck Institute for Mathematics in the Sciences organised by
Rishabh Gvalani, Francesco Mattesini, Felix Otto, and  Markus Tempelmayr. In particular, we also thank Jiwoon Park for leading the exercise classes
at the last summer school.

\bibliography{all}
\bibliographystyle{plain}

\end{document}